\documentclass[12pt]{amsart}

\usepackage{amsmath}
\usepackage{amsfonts}
\usepackage{amssymb}
\usepackage{amsthm}
\usepackage{parskip}
\usepackage{enumerate}
\usepackage{cite}
\usepackage{pstricks}
\usepackage{pst-node}

\usepackage[colorlinks, citecolor=blue, linkcolor=black]{hyperref}

\usepackage{array}
\newcolumntype{x}[1]{>{\arraybackslash\hspace{0pt}}p{#1}}
\usepackage{multirow}

\usepackage[tmargin={1in},bmargin={1in},lmargin={1.1in},rmargin={1.1in}]{geometry}

\usepackage{etoolbox}
\patchcmd{\subsection}{\bfseries}{\itshape}{}{}

\makeatletter
\def\@seccntformat#1{%
  \protect\textup{\protect\@secnumfont
    \ifnum\pdfstrcmp{subsection}{#1}=0 \bfseries\fi% subsection # in \bfseries
    \csname the#1\endcsname
    \protect\@secnumpunct
  }%
}  
\makeatother

\makeatletter
\renewcommand\part{%
   \if@noskipsec \leavevmode \fi
   \par
   \addvspace{4ex}%
   \@afterindentfalse
   \secdef\@part\@spart}

\renewcommand{\thepart}{\Roman{part}} 

\def\@part[#1]#2{%
    \ifnum \c@secnumdepth >\m@ne
      \refstepcounter{part}%
      \addcontentsline{toc}{section}{\textbf{{\thepart.\hspace{1em}#1}}}%
    \else
      \addcontentsline{toc}{part}{#1}%
    \fi
    {\center 
     \interlinepenalty \@M
     \normalfont
     \ifnum \c@secnumdepth >\m@ne
       \large\sc \partname\nobreakspace\thepart: 
       \nobreak
     \fi
     \large \sc #2%
     \par}%
    \nobreak
    \vskip 3ex
    \@afterheading}
\def\@spart#1{%
    {\parindent \z@ \raggedright
     \interlinepenalty \@M
     \normalfont
     \huge \bfseries #1\par}%
     \nobreak
     \vskip 3ex
     \@afterheading}
\makeatother

\makeatletter
\def\swappedhead#1#2#3{%
  \thmnumber{\@upn{\the\thm@headfont#2\@ifnotempty{#1}{.~}}}%
  \thmname{#1}%
  \thmnote{ {\the\thm@notefont(#3)}}}
\makeatother
\usepackage{bbm}
\usepackage{parskip}
\usepackage{xcolor}
\usepackage[all]{xy}
\usepackage{tikz-cd}
\usepackage{bm}
\usepackage{graphicx}

\usetikzlibrary{decorations.pathmorphing}

\newtheorem{theoremalpha}{Theorem}

\numberwithin{equation}{subsection}

\swapnumbers
\newtheorem*{lemma*}{Lemma}
\newtheorem*{proposition*}{Proposition}
\newtheorem*{theorem*}{Theorem}

\newcommand{\ZZ}{\mathbb{Z}}

\DeclareMathOperator{\Lie}{Lie}
\DeclareMathOperator{\Der}{Der}

\DeclareMathOperator{\Aut}{Aut}
\DeclareMathOperator{\id}{id}

\DeclareMathOperator{\End}{End}

\DeclareMathOperator{\Int}{Int}
\DeclareMathOperator{\Br}{Br}
\DeclareMathOperator{\im}{im}
\DeclareMathOperator{\tr}{tr}
\DeclareMathOperator{\Str}{Str}

\DeclareMathOperator{\Char}{char}
\DeclareMathOperator{\GL}{GL}

\DeclareMathOperator{\GO}{GO}
\DeclareMathOperator{\PGO}{PGO}
\DeclareMathOperator{\PGL}{PGL}

\DeclareMathOperator{\SL}{SL}
\DeclareMathOperator{\ad}{ad}
\DeclareMathOperator{\Ad}{Ad}
\DeclareMathOperator{\Skew}{Skew}
\DeclareMathOperator{\Nuc}{Nuc}
\DeclareMathOperator{\cor}{cor}
\DeclareMathOperator{\res}{res}

\DeclareMathOperator{\Sim}{Sim}
\DeclareMathOperator{\Iso}{Iso}

\DeclareMathOperator{\KP}{KP}
\DeclareMathOperator{\Inv}{Inv}
\DeclareMathOperator{\Spin}{Spin}

\DeclareMathOperator{\ed}{{e.}{d.}}
\DeclareMathOperator{\Pf}{Pf}

\allowdisplaybreaks

\newtheorem{theorem}[subsection]{Theorem}
\newtheorem{proposition}[subsection]{Proposition}
\newtheorem{lemma}[subsection]{Lemma}
\newtheorem{corollary}[subsection]{Corollary}

\theoremstyle{definition}

\newtheorem*{definition*}{Definition}
\newtheorem*{examples*}{Examples}
\newtheorem*{example*}{Example}
\makeatletter
\def\thm@space@setup{%
	\thm@preskip=\parskip \thm@postskip=6pt
}
\makeatother

\makeatletter
\DeclareFontFamily{OMX}{MnSymbolE}{}
\DeclareSymbolFont{MnLargeSymbols}{OMX}{MnSymbolE}{m}{n}
\SetSymbolFont{MnLargeSymbols}{bold}{OMX}{MnSymbolE}{b}{n}
\DeclareFontShape{OMX}{MnSymbolE}{m}{n}{
    <-6>  MnSymbolE5
   <6-7>  MnSymbolE6
   <7-8>  MnSymbolE7
   <8-9>  MnSymbolE8
   <9-10> MnSymbolE9
  <10-12> MnSymbolE10
  <12->   MnSymbolE12
}{}
\DeclareFontShape{OMX}{MnSymbolE}{b}{n}{
    <-6>  MnSymbolE-Bold5
   <6-7>  MnSymbolE-Bold6
   <7-8>  MnSymbolE-Bold7
   <8-9>  MnSymbolE-Bold8
   <9-10> MnSymbolE-Bold9
  <10-12> MnSymbolE-Bold10
  <12->   MnSymbolE-Bold12
}{}

\let\llangle\@undefined
\let\rrangle\@undefined
\DeclareMathDelimiter{\llangle}{\mathopen}%
                     {MnLargeSymbols}{'164}{MnLargeSymbols}{'164}
\DeclareMathDelimiter{\rrangle}{\mathclose}%
                     {MnLargeSymbols}{'171}{MnLargeSymbols}{'171}
\makeatother

\makeatletter
\newcommand{\bigperp}{%
  \mathop{\mathpalette\bigp@rp\relax}%
  \displaylimits
}

\newcommand{\bigp@rp}[2]{%
  \vcenter{
    \m@th\hbox{\scalebox{\ifx#1\displaystyle2.1\else1.5\fi}{$#1\perp$}}
  }%
}
\makeatother

\title[Groups acting on the tensor product of composition algebras]{Galois cohomology of algebraic groups acting on the tensor product of two composition algebras}

\date{\today}
\author{Simon W. Rigby}
\address{Department of Mathematics: Algebra and Geometry, Ghent University, Belgium}
\email{simon.rigby@ugent.be}

\subjclass[2020]{Primary: 11E72; Secondary: 20G15, 17D99}

\begin{document}

\begin{abstract}
	This is a study of algebras with involution that become isomorphic over a separable closure of the base field to a tensor product of two composition algebras. We classify these algebras, provide criteria for isomorphism and isotopy, and determine their automorphism groups, structure groups, and norm-similitude groups.  The most interesting cases of these algebras are the $64$-dimensional \emph{bi-octonion algebras}, on which groups of absolute type $(G_2 \times G_2)\rtimes \ZZ/2\ZZ$ and $\mathbf{Spin}_{14}$ act faithfully by automorphisms and norm-preserving isotopies, respectively.  We classify the cohomological invariants of these algebras, characterise when they are division algebras, and study their 14-dimensional Albert forms and 64-dimensional octic norms. The main application that we reach is a classification of the mod 2 cohomological invariants of 14-dimensional quadratic forms in $I^3$, and of the group $\mathbf{Spin}_{14}$. This extends Garibaldi's classification of cohomological invariants for lower-dimensional Spin groups.	\end{abstract}

\maketitle

\tableofcontents

\section{Introduction}

Despite the importance and rich history of composition algebras, their tensor products are only well-understood in the associative case: for example, bi-quaternion algebras play an important role in the theory of central simple associative algebras \cite{knus1998book}. On the other extreme, tensor products of two octonion algebras (which we call bi-octonion algebras) are highly nonassociative and seriously difficult to work with in an elementary way. But there are good reasons to study them: they are representations of important algebraic groups such as $\mathbf{Spin}_{14}$, and for this reason they appear in some of the most famous constructions of the largest exceptional simple algebraic group, $E_8$.  For instance, bi-octonion algebras are a main ingredient in the Tits construction  of $E_8$ \cite{tits1966algebres,vinberg2005construction}, the Tits--Kantor--Koecher construction \cite{allison1979models}, and the very general Allison--Faulkner construction from \cite{allison1993nonassociative} (described also in \cite[\S11]{gross2021minuscule} and \cite[\S1]{petrov2021allison}).

 Often, these algebras and their relation to algebraic groups have been studied using the tools of representation theory rather than nonassociative algebra. Examples where the representation theoretic point of view is emphasised include \cite{abuaf2021gradings,garibaldi2017spinors, popov1980classification, rost14dim}, although there are some noteworthy articles such as \cite{allison1988tensor,aranda2020gradings, morandi2001tensor} where the treatment is more purely algebraic.
Tensor products of pairs of composition algebras belong to a class of algebras called \emph{structurable algebras}. The theory of structurable algebras, initiated by Allison in \cite{allison1978class}, has well-developed notions of invertibility, isotopy, norms, and traces \cite{allison1992norms, allison1981isotopes, schafer1989invariant}, which we make full use of in this work. A certain Tits--Kantor--Koecher construction can be applied to any structurable algebra, and this has also featured in recent research in connection with isotropic simple algebraic groups, 5-graded simple Lie algebras, and Moufang sets \cite{boelaert2019moufang,stavrova2018classification}.

To understand the value in the nonassociative algebraic approach, it is fitting to draw a comparison to the exceptional 27-dimensional simple Jordan algebras (commonly known as Albert algebras). Thanks to more than half a century of research, we understand a lot about these exceptional Jordan algebras, including information about constructing and classifying them, the division algebras among them, their cubic norms, quadratic traces, special subalgebras, isotopes, automorphism groups, structure groups, and cohomological invariants \cite{hooda2020survey,mccrimmon1970freudenthal,parimala1998classification,petersson2019survey,petersson1995invariants,petersson1996elementary,thakur2021cyclicity,thakur1999isotopy}. If we had no substantial theory of Jordan algebras, we would still know about the existence of a 27-dimensional representation for $F_4$ and $E_6$ and a certain cubic invariant polynomial. But so much else would likely have remained a mystery, such as the answer to the Kneser--Tits problem for groups of type $E_{7,1}^{78}$ and $E_{8,2}^{78}$ \cite{alsaody2019tits,thakur2019albert}. With this kind of motivation in mind, we aim to advance the theory of bi-octonion algebras.

The main application that this paper puts forward is the classification of cohomological invariants  of $14$-dimensional quadratic forms in $I^3$ and $\mathbf{Spin}_{14}$-torsors.    It is interesting to put these results in context.
The simple algebraic group $\mathbf{Spin}_n$ is a simply connected   double cover of the special orthogonal group $\mathbf{O}^+_n$. Roughly speaking, if we understand the Galois cohomology of spin groups, we understand the nature of quadratic forms in $I^3_n$ (as defined in \cite[\S3]{chernousov2014essential}). There are various ways to measure the complexity of $I^3_n$, such as the essential dimension    and the Pfister number \cite{brosnan2010essential}. Following some remarkable work by many mathematicians, we know the essential dimensions of $\mathbf{Spin}_{n}$ and $I_n^3$ for all $n$ \cite{brosnan2010essential,chernousov2014essential,garibaldi2017spinors}. Unlike $\mathbf{O}^+_n$ whose essential dimension grows linearly with $n$, the essential dimension of $\mathbf{Spin}_n$ grows exponentially with $n$. It begins to explode dramatically after $n = 14$, going from $\ed(\mathbf{Spin}_{14}) =  7$ to $\ed(\mathbf{Spin}_{15}) = 23$, $\ed(\mathbf{Spin}_{16}) = 24$, and $\ed(\mathbf{Spin}_{17}) = 120$. The Pfister number  $\Pf(3,n)$ is the least number of  general Pfister forms one needs to sum to get any form in $I_{n}^3$. These numbers also grow at least exponentially with $n$ \cite{brosnan2010essential} but exact values are not known. Looking at the situation around 14, we have $\Pf(3,12) = 2$, $\Pf(3,14) = 3$, and $\Pf(3,16) \ge 4$ with no known upper bound \cite{karpenko2017around}. In short, $\mathbf{Spin}_{14}$ is a boundary case among the spin groups.

The essential dimension  of $\mathbf{Spin}_{n}$ has implications for rational parameterisations and cohomological invariants of quadratic forms in $I_{n}^3$. There is a rational parameterisation for quadratic forms in $I_{14}^3$ that takes 7 parameters (this is the content of Rost's Theorem \cite{rost14dim}, which we reprove in Corollary \ref{cor:Rost}), and there exists a degree 7 cohomological invariant of $I_{14}^3$ (defined by Garibaldi \cite{garibaldi2009cohomological}, and appearing here in \ref{sec:I14-known-invariants}); both of these results are consistent with the value $\ed(\mathbf{Spin}_{14})= \ed(I_{14}^3) = 7$. In contrast, if there even exists a rational parameterisation for quadratic forms in $I_{16}^3$, it would need to have at least $24$  parameters. It is conjectured that no such parameterisation exists \cite[Conjecture~4.5]{merkurjev2017invariants}. And for all we know, there could exist cohomological invariants of   $\mathbf{Spin}_{16}$ in degree  as high as  24. So we have basically no hope (at least with current methods) of classifying or even finding the high-degree cohomological invariants of $\mathbf{Spin}_n$ beyond the frontier of $n = 14$. Cohomological invariants of $\mathbf{Spin}_n$ for $n \le 12$ were classified by Garibaldi in \cite[Table~23B]{garibaldi2009cohomological}, with the help of methods developed years ago mainly by Rost and Serre. The complete classification of invariants of $\mathbf{Spin}_{14}$ seemed within reach and yet was still unknown until now -- this is what made it an alluring problem.

Let us summarise the contents of this paper. The majority of the sections are broad in scope and do not limit themselves to the study of bi-octonion algebras. These are the most interesting and (in most cases) the most difficult, but for completeness we include results on  other tensor products whose factors are composition algebras of dimensions $(2^a,2^b)$ for $0 \le a,b \le 3$, $(a,b) \ne (1,1)$. The paper contains three tables. The first two tables are populated with data on groups and Lie algebras related to these algebras. Table~\ref{table.main} is a version of the magic square: it presents the absolute type of the Lie algebra  constructed from one of these algebras  using the Tits--Kantor--Koecher (TKK) construction, as well as the type of its grading expressed by a labelled Dynkin diagram. Table~\ref{table.structure-groups} contains data that completely determines the structure groups of these algebras, and hence the anisotropic kernels of the TKK constructions in Table \ref{table.main}. Table~\ref{table.summary-of-invariants} contains our new results classifying the cohomological invariants of bi-octonion algebras and other functors such as $I^3_{14}$, $PI_{14}^3$ ($= I^3_{14}$ modulo similitude), and $\mathbf{Spin}_{14}$-torsors.

In the rest of the introduction, we highlight only the most interesting results on bi-octonion algebras and their cohomological invariants, whereas results on the lower-dimensional algebras can be found in the tables and body of the paper. 
 In all of the theorems below, we assume that $k$ is a field with $\Char(k) \ne 2,3$. We say that a $k$-algebra with involution $(A,-)$ is a bi-octonion algebra if there exists a field extension $L/k$ and octonion algebras $C_1, C_2$ over $L$ such that $(A\otimes_k L, -)$ is isomorphic to $C_1 \otimes_L C_2$  with its canonical involution. If there is already such an isomorphism over $L = k$, then we call $(A,-)$ decomposable.
 
  One of the first main theorems about bi-octonion algebras that we obtain is the following. It is a direct consequence of the more detailed and more general statements in  Theorems \ref{thm.equiv2}, \ref{thm.aut2}, \ref{thm:aut-schemes-involution}, and Corollary \ref{cor:involution-dependence}, and makes reference to a well-known corestriction construction which we describe in \ref{sec:corestriction}.
  
\begin{theoremalpha}
	The corestriction construction $C \mapsto \cor_{E/k}(C)$ induces a one-to-one correspondence between $k$-isomorphism classes of octonion algebras over quadratic \'etale extensions $E/k$, and $k$-isomorphism classes of bi-octonion algebras over $k$.
	The group scheme $\mathbf{Aut}(\cor_{E/k}(C))$ of $k$-automorphisms of $\cor_{E/k}(C)$ is smooth with two connected components, and is isomorphic to the group scheme $\mathbf{Aut}(C_{/k})$ of $k$-automorphisms of $C$.
\end{theoremalpha}

To each bi-octonion algebra $(A,-)$ is associated a  quadratic form $Q$, called the Albert form, which is the unique quadratic form (up to similitude)  on $S = \Skew(A,-) = \{s \in A \mid \bar s = -s \}$ such that $Q(s) \ne 0$ if and only if the left-multiplication operator $L_s \in \End A$ is invertible. The Albert form $Q$ is a $14$-dimensional quadratic form in $I^3(k)$, because it is an additive transfer of the norm $n\in I^3(E)$ on the corresponding octonion algebra over some quadratic \'etale extension $E/k$.
 
 Based on evidence in characteristic 0, the Albert form is related to the equivalence relation of isotopy on the class of bi-octonion algebras. We find a good explanation for this by showing that the structure group (the group of self-isotopies) of a bi-octonion algebra is a reductive envelope of the spin group of its Albert form. Also interesting is that  the structure group is no less than the group of linear transformations that preserve the norm up to a scalar. The following theorem is distilled from Proposition \ref{main1} and Theorems~\ref{main2} and \ref{thm:Str-GO}.
   
 \begin{theoremalpha}
	The structure group $\Str(A,-)$, as an abstract group, is isomorphic to the quotient $\Omega(S,Q)/T$ where $\Omega(S,Q)$ is the extended Clifford group of the Albert form and $T = \{ce_1 \mid c \in k^\times\} \simeq k^\times$ is the subgroup of scalar multiples of a primitive central idempotent  $e_1 \in C^+(S,Q)$.
	As an algebraic group, $\mathbf{Str}(A,-)$ is the connected reductive group generated by its derived subgroup, which is isomorphic to $\mathbf{Spin}(S,Q)$, and its centre, which is the one-dimensional torus consisting of invertible scalar matrices.
	Moreover, $\mathbf{Str}(A,-)$ is equal to the group scheme $\mathbf{GO}(N_A)$ of similitudes of the octic norm of $(A,-)$.
\end{theoremalpha}

Taking a closer look at the Albert forms of bi-octonion algebras, we find that they can in fact be similar to any 14-dimensional quadratic form in $I^3(k)$, and that the Albert forms classify the algebras up to isotopy.  The first part of Theorem \ref{main-theorem:albert_form} is a result which was originally discovered by Rost \cite{rost14dim}, but until now the proof has only been available in sketch (e.g., \cite[Theorem~21.3]{garibaldi2009cohomological}). The second part of Theorem \ref{main-theorem:albert_form} was previously known in characteristic 0 only \cite[Theorem~5.4~(iii)]{allison1988tensor}, and its proof depended on Lie algebraic methods which do not work well in positive characteristics. We find new proofs for these results in Corollaries~\ref{cor:isotopic-similar} and \ref{cor:Rost}.
 
 \begin{theoremalpha} \label{main-theorem:albert_form}
 Every $14$-dimensional quadratic form whose Witt class is in $I^3(k)$ is similar to the Albert form of a bi-octonion algebra, and the bi-octonion algebra is uniquely determined up to isotopy by the similitude class of its Albert form.
 \end{theoremalpha}

In \S\ref{sec:cohom_inv_of_bi-oct_alg} we introduce, interpret, and examine the values and symbol lengths of three cohomological invariants of bi-octonion algebras taking values in mod 2 Galois cohomology. These invariants have very elementary descriptions. If $(A,-)$ is a bi-octonion algebra,  $\Skew(A,-)$ has the structure of a Malcev algebra whose centroid $E$ is a quadratic \'etale extension of $k$. The first invariant, which we call the \emph{discriminant}, is
\[
b_1(A,-) = [E] =  \text{ the class of the centroid of } \Skew(A,-).
\]
The second invariant, which we call the \emph{Albert form invariant}, is 
\[
b_3(A,-) = e_3(Q) = \text{ the Arason invariant of the Albert form of }(A,-).
\]
    The trace form $T_A$ is a 64-dimensional symmetric bilinear form on $A$ defined by $T_A(x,y) = \tr(L_{x\bar y+ y \bar x})$. The form $
	\langle 128\rangle T_A \perp 4\langle -1\rangle N_{E/k}$, where $N_{E/k}$ is the 2-dimensional norm of the centroid of $\Skew(A,-)$, is Witt equivalent to a 6-Pfister form. This defines the third invariant, which we call the \emph{quadratic trace form invariant}:
	\[
	b_6(A,-) = e_6(\langle 128 \rangle T_A - 4  N_{E/k}).
	\]
Summarised, here are some interesting facts about these invariants -- see Example \ref{example:split-case}, Theorem \ref{thm:symbols}, and Proposition \ref{prop:isotopy-invariants}:

\begin{theoremalpha}
	There exist non-trivial cohomological invariants $b_1, b_3, b_6$ (as defined above) each of which assigns a bi-octonion algebra $(A,-)$ over $k$ to a unique element
	\[
	b_i(A,-) \in H^i(k,\ZZ/2\ZZ)
	\]
	which depends only on the isomorphism class of $A$ as an algebra, and which is compatible with base change to field extensions.	Moreover,
	\begin{enumerate}[\rm (1)]
	\item 	The invariant $b_1$ detects decomposability in the sense that $b_1(A,-) = 0$ if and only if $(A,-)$ is decomposable.
	\item The symbol length of $b_3(A,-)$ can be $0$, $1$, $2$, or $3$.
	\item The symbol length of $b_6(A,-)$ is $0$ or $1$.
	\item If $(A,-)$ and $(A',-)$ are isotopic then $b_3(A,-) = b_3(A',-)$.
	\item If $b_3(A,-) = 0$ then $(A,-)$ is isotopic to the split bi-octonion algebra.
	\item If $(A,-)$ and $(A',-)$ are isotopic, $b_6(A,-) = b_6(A',-)$ modulo $(-1){\cdot}(-1){\cdot} H^4(k,\ZZ/2\ZZ)$.
	\end{enumerate}
\end{theoremalpha}

The invariants also bear relevance to the property of being a structurable division algebra. We give the first cohomological obstructions to the existence of noninvertible elements in a bi-octonion algebra, namely (7) and (8) of Theorem \ref{main-theorem:division}, and also prove the equivalence of a long list of algebraic conditions, namely items (1) to (6) of the theorem. The equivalence of (1), (2), and (5) was previously known in characteristic 0 only \cite{allison1988tensor} and the rest are new results in all characteristics. The statement is a combination of Theorems~\ref{thm:divisionness}, \ref{thm:divisionness0}, and \ref{thm:symbols} in the text.

\begin{theoremalpha} \label{main-theorem:division}
	Let $(A,-)$ be a bi-octonion algebra over $k$. The following are equivalent:
	\begin{enumerate}[\rm (1)]
		\item $(A,-)$ is not a structurable division algebra.
		\item $(A,-)$ has a non-invertible skew element.
		\item $(A,-)$ has a non-division biquaternion subalgebra stabilised by the involution.
		\item $(A,-)$ has a non-division associative subalgebra stabilised by the involution.
		\item  The Albert form of $(A,-)$ is isotropic and similar to $\phi_1' \perp \langle -1 \rangle \phi_2'$ where $\phi_1, \phi_2$ are a pair of $3$-Pfister forms with a common slot.
		\item $(A,-)$ is isotopic to a decomposable bi-octonion algebra $C_1\otimes C_2$ where $C_1$ and $C_2$ are octonion algebras whose norms have a common slot.
	\end{enumerate}
	Moreover, if {\rm (1)} holds, then
	\begin{enumerate}[\rm (1)] \setcounter{enumi}{6}
		\item The symbol length of $b_3(A,-)$ is at most $2$.
		\item $b_6(A,-) \in (-1){\cdot} H^5(k,\ZZ/2\ZZ)$.
	\end{enumerate}
\end{theoremalpha}

The set $\Inv(F)$ of mod 2 cohomological invariants of a functor $F:\mathsf{Fields}_{/k} \to \mathsf{Sets}$ is an $H(k)$-algebra, where $H(k) = \bigoplus_{i \ge 0} H^i(k, \ZZ/2\ZZ)$. When we speak of $\Inv(G)$ for an algebraic group $G$, this means $\Inv(F)$ where $F = H^1(*,G)$. We prove in Theorem \ref{thm:inv-bioctonions}:

\begin{theoremalpha}
	The invariants $1$, $b_1$, $b_3$, and $b_6$ are linearly independent generators for the free $H(k)$-module $\Inv((G_2 \times G_2) \rtimes \ZZ/2\ZZ)$ of mod $2$ invariants of bi-octonion algebras.
\end{theoremalpha}

Theorem \ref{thm:inv-bioctonions} also contains a classification of invariants of $\mathbf{PGO}_4$, which until now does not seem to have been done. The above classification of invariants of bi-octonion algebras is one of the key steps in the sought-after classification of invariants of $\mathbf{Spin}_{14}$. Another key step is the classification of invariants of $\mathbf{Spin}_{12}$. This was already done in \cite{garibaldi2009cohomological} under the assumption that $\sqrt{-1} \in k$, but removing that assumption is not  difficult: the same strategy works as a proof, only the calculations are slightly more technical. Because of this,  in Theorem~\ref{thm:I12-Spin12} we have redone the classification of invariants of $\mathbf{Spin}_{12}$. Despite  efforts, we were unable to remove the assumption $\sqrt{-1} \in k$ when it came to $\mathbf{Spin}_{14}$.  We prove in Theorem \ref{main}: 	
\begin{theoremalpha}
	If $\sqrt{-1} \in k$ then the invariants $1$, $a_3$ (The Arason invariant), $a_6$, and $a_7$ from \cite[\S22]{garibaldi2009cohomological} are linearly independent generators for the free $H(k)$-module $\Inv(I_{14}^3)$.
	Every mod $2$ invariant of $\mathbf{Spin}_{14}$ factors through $I_{14}^3$, so $\Inv(\mathbf{Spin}_{14})$ is a free $H(k)$-module having a basis of invariants of degrees $0$, $3$, $6$, and $7$. \end{theoremalpha}

This theorem resolves a question mark in the last row of \cite[Table 23B]{garibaldi2009cohomological}. Our approach to classifying the invariants of $\mathbf{Spin}_{14}$ is outlined in more detail in the introduction to Part~\ref{pt:II}.
 
\let\oldaddcontentsline\addcontentsline% Store \addcontentsline
\renewcommand{\addcontentsline}[3]{}% Make \addcontentsline a no-op
\section*{Acknowledgements}
I would like to thank Tom De Medts and Skip Garibaldi for helpful feedback, and Philippe Gille for very good  discussions, and for generously allowing the use of Lemma~\ref{lem:gille}. This research was supported by FWO project number G004018N.
\let\addcontentsline\oldaddcontentsline% Restore \addcontentsline

\part{The algebras and the algebraic groups}

\section{Preliminaries}

We work over an arbitrary field $k$ whose characteristic is not 2 or 3, and we write $k^s$ and ${k}^a$ for a separable and algebraic closure of $k$, respectively.

%The assumptions on the characteristic really are necessary for almost all of our results because of the lack of foundational material on structurable algebras in very low characteristics, and because statements in \ref{sec.adj.simple} and Theorem~\ref{thm.malc7} rely on them.% The letter $R$ usually stands for a commutative $k$-algebra.

\subsection{Algebraic groups}

We take the functorial view of (affine) algebraic groups, in the style of \cite{milne} and \cite{knus1998book}.
The notation $G^\circ$ refers to the connected component of the identity, $\pi_0(G) = G/G^\circ$ is called the group of components, and $G^{\rm der}$ refers to the derived subgroup of~$G$. 
The notation $C_G(H)$ stands for the (scheme-theoretic)  centraliser of a subgroup $H$ of $G$, and $Z(G)$ stands for the centre of $G$. If $\lambda$ is a homomorphism taking values in $G$, we write $C_G(\lambda)$ for the centraliser of the image of $\lambda$ in~$G$.

 We make use of several well-known algebraic groups and their standard notations, like the multiplicative and additive groups $\mathbf{G}_m$ and $\mathbf{G}_a$, the automorphism group scheme $\mathbf{Aut}(A)$ of a $k$-algebra $A$, the orthogonal group  $\mathbf{O}(V,q)$ of a quadratic space $(V,q)$ over~$k$, the special orthogonal group $\mathbf{O}^+(V,q)$, the group of similitudes $\mathbf{GO}(V,q)$,  and the group of projective similitudes $\mathbf{PGO}(V,q)$. The $k$-group scheme of $n$-th roots of unity is $\bm{\mu}_n$, or $\bm{\mu}_{n,k}$ if we want to be precise about the field in question.
 
 We count on background knowledge of root systems, associative central simple algebras with involution, and their relation to simple algebraic groups of classical type (for which excellent references are \cite[\S\S23--28]{knus1998book} or \cite[\S24]{milne}).

 \subsection{Homomorphisms and exact sequences}  \label{sec.exact} 
Short exact sequences of algebraic groups are  defined in \cite[Definition 1.61]{milne}.
 %; they are made of a quotient map and its kernel.
  We often construct short exact sequences the following way. Suppose $\varphi: G \to H$ is a morphism of algebraic groups, $G$ is smooth, and $C$ is a subgroup of $H$ such that $\varphi_{k^a}( G(k^a) )\subset C(k^a)$. Then $\varphi$ factors through the inclusion $C \subset H$ (so we can and sometimes do redesignate $\varphi$ as a morphism $G \to C$). If $N$ is another algebraic group with a morphism $\nu: N \to G$, the sequence of algebraic groups
 \[
 \begin{tikzcd}
	1 \arrow[r] & N 	\arrow[r,"\nu"] & G \arrow[r,"\varphi"] & C \arrow[r] & 1
 \end{tikzcd}
 \]
 is {exact} if and only if $C$ is smooth, $\varphi_{k^a}( G(k^a) )=C(k^a)$, and
 \[
 \begin{tikzcd}
	1 \arrow[r] & N(R) 	\arrow[r,"\nu_R"] & G(R) \arrow[r,"\varphi_R"] & C(R)
 \end{tikzcd}
 \]
is an exact sequence for all commutative $k$-algebras $R$. In this case, we say that $\varphi: G \to C$ is surjective.

\subsection{Galois cohomology functors}

We write $\Gamma_k = \mathcal{G}al(k^s/k)$ for the absolute Galois group of~$k$.
If $M$ is a group on which $\Gamma_k$ acts continuously by automorphisms (a $\Gamma_k$-group, for short), the nonabelian cohomology sets $H^0(k,M) = M^{\Gamma_k}$ and $H^1(k,M) = H^1(\Gamma_k, M)$ are defined as in \cite[I.\S5.1]{serre1997galois}. If $M$ is abelian (i.e.\ a $\Gamma_k$-module)  the abelian cohomology groups $H^n(k,M) = H^n(\Gamma_k, M)$ are defined for all $n \ge 0$. The cohomology is functorial in the second argument in the sense that any morphism of $\Gamma_k$-groups $M\to N$ induces a morphism $H^n(k,M)\to H^n(k,N)$.

If $M = G(k^s)$ for a smooth algebraic group $G$ over $k$, we write $H^n(k,G) = H^n(\Gamma_k,G(k^s))$. If $L_1/k \to L_2/k$ is any morphism of field extensions, there is a natural way to define $H^n(L_i,G)$  and  a restriction map 
$\res_{L_2/L_1}: H^n(L_1,G) \to H^n(L_2,G)$, which makes $H^n(*,G)$ into a functor $\mathsf{Fields}_{/k} \to \mathsf{Sets_*}$ (i.e., from the category of field extensions of $k$ to the category of pointed sets) \cite[II.\S1.1]{serre1997galois}. The set $H^1(L,G)$ is naturally isomorphic to the set of $L$-isomorphism classes of principal homogeneous spaces of $G$ (also known as $G$-torsors); see \cite[Proposition~28.14]{knus1998book}, \cite[Definition~2.66, Proposition~3.50]{milne},  or  \cite[I.\S5.2]{serre1997galois}.

Often, $G$-torsors can be viewed as $k$-forms of some algebraic structure. Say, if $G= \mathbf{Aut}(A)$ is the automorphism group scheme of an algebra $A$ over $k$, the cohomology set $H^1(k, G)$ is in bijection with the set of isomorphism classes of algebras $A'$ over $k$ such that $A'\otimes_k k^s \simeq A \otimes_k k^s$.  
Similarly,  if $(A,-)$ is an algebra with involution, $H^1(\mathbf{Aut}(A,-))$ classifies the isomorphism classes of algebras with involution that become isomorphic to $(A,-)$ over~$k^s$.

%Quadratic \'etale extension is a synonym for 2-dimensional composition algebra.

\subsection{Cohomology and short exact sequences} \label{sec:cohomology-sequences}
An exact sequence of smooth algebraic groups
\begin{equation*}
\begin{tikzcd}
	1 \ar[r] & A \ar[r,"f"] & B \ar[r,"g"] & C \ar[r] & 1
\end{tikzcd}
\end{equation*}
gives rise to a first connecting map $\delta$ and an exact sequence of pointed sets \cite[I.\S5 Proposition~38]{serre1997galois}
\begin{equation} \label{eq:long exact}
\begin{tikzcd}[sep=small]
1 \ar[r] & A(k) \ar[r] & B(k) \ar[r] & C(k)  \ar[r,"\delta"] &	H^1(k,A) \ar[r,"f_*"] & H^1(k,B) \ar[r,"g_*"] & H^1(k,C).
\end{tikzcd}
\end{equation}
If $A$ is abelian and $f(A)\subset Z(B)$, there is a second connecting map $\Delta: H^1(k,C)\to H^2(k,A)$ and appending this to \eqref{eq:long exact} gives a slightly longer exact sequence \cite[I.\S5.6 Proposition~42]{serre1997galois}.

\subsection{Cohomological invariants in \emph{mod 2} Galois cohomology} \label{sec:cohomological-invariants} The constant Galois module $\ZZ/2\ZZ$ (alternatively written as $S_2$ or as $\bm{\mu}_2$ since $\Char(k) \ne 2$) is important to us because the functor $H: \mathsf{Fields}_{/k} \to \mathsf{Ab}$,
\[
H(L) = \bigoplus_{i \ge 0}H^i(L, \ZZ/2\ZZ)
\] will be the target of all our cohomological invariants. There is a canonical isomorphism $k^\times/k^{\times 2} \simeq H^1(k,\ZZ/2\ZZ)$ where a square class $dk^{\times 2}$ corresponds to the symbol $ (d)$. 
  Since $\ZZ/2\ZZ \otimes \ZZ/2\ZZ \simeq \ZZ/2\ZZ$ and $H(k)$ is 2-torsion, the cup product (which we denote by $\cdot$) gives a commutative $\ZZ$-graded ring structure to $H(k)$ with nice properties \cite[Proposition III.9.15]{berhuy}.   Consequent to the Milnor conjectures, the entire cohomology ring $H(L)$ is additively generated by symbols; i.e., elements of the form $(a_1){\cdot} \cdots {\cdot}(a_m)$. For a cohomology class $\alpha \in H^i(L,\ZZ/2\ZZ)$, the symbol length of $\alpha$ is the least $\ell \ge 0$ such that $\alpha$ is a sum of $\ell$ symbols.

\begin{definition*} Let $F: \mathsf{Fields}_{/k} \to \mathsf{Sets}$ be a functor. A \emph{mod $2$ cohomological invariant} of $F$ is a natural transformation $F \to H$. The set of all mod 2 cohomological invariants of $F$ is denoted by $\Inv(F)$. We take some liberties with the notation: if $G$ is an algebraic group we write $\Inv(G) = \Inv(F)$ where $F = H^1(*,G)$.
\end{definition*}

The set $\Inv(F)$ is a commutative unital $\ZZ$-graded $H(k)$-algebra:
	$
	\Inv(F) = \bigoplus_{i \ge 0} \Inv^i(F)
	$
	where $\Inv^i(F)$ is the submodule of invariants taking values in $H^i(*,\ZZ/2\ZZ)$. Note that $\Inv(*)$ is a contravariant functor $\mathsf{Sets}^{\mathsf{Fields}_{/k}} \to \mathsf{Commutative~Rings}$, so  morphisms of functors $j:F_1 \to F_2$ come with homomorphisms $j^*:\Inv(F_2) \to \Inv(F_1)$, denoted $j^*(y) = y \circ j$. In particular, a homomorphism of algebraic groups $G_1 \to G_2$ comes with a homomorphism $\Inv(G_2) \to \Inv(G_1)$.
An invariant $x\in \Inv(G)$ is called \emph{normalised} if it takes the value 0 at the trivial torsor in $H^1(k,G)$, and it is called \emph{nontrivial} if $x \ne 1$ and for all field extensions $L/k$ there exists an extension $L'/L$ and some $\alpha \in X(L')$ such that $x(\alpha) \ne 0$.

\subsection{Quadratic forms and the Witt ring} \label{sec:quadratic-forms}
We denote by $\mathbb{H} = \langle 1, -1 \rangle$ the hyperbolic quadratic form. For $n \in \mathbb{N}$, we write $nq = q \perp \dots \perp q$ ($n$ times). For $n$-Pfister forms, we use the notation $\llangle c_1, \dots, c_n \rrangle = \bigotimes_{i = 1}^n\langle 1, -c_i \rangle$. If $\phi$ is a Pfister form, we write $\phi'$ for its pure part, i.e.\ the unique quadratic form such that $\phi \simeq \langle 1 \rangle \perp \phi'$. Two quadratic forms are called similar if one is isometric to a scalar multiple of the other.

The fundamental ideal in the Witt ring $W(k)$ is the ideal $I(k)$ consisting of the Witt classes of all even-dimensional quadratic forms. We write $I^n(k) = I(k)^n$; this ideal is generated as a subgroup of $W(k)$ by the classes of $n$-Pfister forms. 
 The ideal $I^2(k)$ is the set of classes of forms with trivial discriminant. 
According to Merkurjev's Theorem of 1981 (see \cite[\S V.3]{lam}), $I^3(k)$ is the set of  classes of forms whose Clifford algebra is split. We define $I^n_j(k)$ to be the set of isometry classes of $j$-dimensional forms whose Witt class is in the ideal $I^n(k)\subset W(k)$, and we define $PI^n_j(k)$ to be the  set $I^n_j(k)$ modulo the equivalence relation of similitude.

We write $\widehat{W}(k)$ for the Grothendieck--Witt ring of $k$, which comes equipped with a natural homomorphism $\pi: \widehat{W}(k) \to W(k)$ whose kernel is generated by~$\mathbb{H}$.

\subsection{Additive transfers} \label{sec:transfers-and-squares} Let $E/k$ be a quadratic \'etale extension, and let $(V,q)$ be an $n$-dimensional quadratic space over $E$. Let $s: E \to k$ be a linear functional. The \emph{additive  transfer}, or Scharlau transfer, of $(V,q)$ along $s$ is the $2n$-dimensional $k$-quadratic space $(V,s_*(q))$ where $s_*(q)(v) = s(q(v))$ for all $v \in V$. We write $T_{E/k}$ for the operation $(\tr_{E/k})_*$.
%No generality is lost if we focus only on $T_{E/k}$, for if $s:E\to k$ is any linear functional, there is an $e \in E$ such that $s(x) = \tr_{E/k}(ex)$ for all $x \in E$ \cite[proof of Corollary~20.7]{elman2008algebraic}, in which case $s_*(q) = T_{E/k}(\langle e \rangle q)$.

In the split case where $E = k\times k$, maintaining the notation from before, $V$ is a direct sum of $n$-dimensional $k$-vector spaces $V = V_1 \oplus V_2$ with the $E$-module structure $(e_1, e_2)\cdot (v_1, v_2) = (e_1 v_1, e_2 v_2)$, and there are quadratic spaces $(V_i, q_i)$ over $k$ such that $q: V \to k$ is the function $q((v_1, v_2)) = (q_1(v_1), q_2(v_2))$. In this case, $T_{E/k}(q) \simeq q_1 \perp q_2$.

%\begin{theorem} \cite[Exact Triangle Theorem 3.5]{lam} \label{thm:exact-triangle} Let $E = k (\sqrt{d})$ be a quadratic field extension, let $r^*: W(k) \to W(E)$ be the natural extension-of-scalars map, let $s_*: W(E) \to W(k)$ be the additive transfer corresponding to the linear functional $s:E \to k$ where $s(1) = 0$ and $s(\sqrt{d}) = 1$, and let $m_{\llangle d \rrangle}: W(k) \to W(k)$ be multiplication by $\llangle d \rrangle$. The following triangle is exact:
%\begin{equation*} \label{diag:exact-triangle}
%\begin{tikzcd}[column sep=small]
%	W(k) \ar[rr, "r^*"] & & W(E)\ar[ld,"s_*"]\\
%	& W(k) \ar[lu,"m_{\llangle d \rrangle}"]
%\end{tikzcd}
%\end{equation*}
%\end{theorem}

 \subsection{Algebras} \label{sec.algebras}
 
 All algebras in this paper are finite-dimensional and not necessarily associative.  Scalar extension of a $k$-algebra $A$ is denoted by $A_R = A \otimes_k R$. (The subscript notation is also used for scalar extension of quadratic forms, algebraic groups, etc.)  
 
 We write commutators and associators as
 $[x,y] = xy - yx$ and $[x,y,z] = (xy)z-x(yz)$. We write $\Nuc(A) = \{a \in A \mid [a,A,A]=[A,a,A]=[A,A,a] = 0\}$ for the nucleus and $Z(A) = \{a \in \Nuc(A) \mid [a,A] = 0\}$ for the centre of $A$. If $n \in \Nuc(A)$, the inner automorphism implemented by $n$ is denoted by $\Int(n) \in \Aut(A)$, $\Int(n)(x) = nxn^{-1}$. The centroid of $A$ is the centre of the subalgebra of $\End(A)$ generated by left- and right-multiplications by elements of $A$.
  We say that $A$ is \emph{central} if its centroid is $k\id$, and \emph{simple} if $A^2 \ne 0$ and $A$ has no two-sided ideals besides $\{0\}$ and $A$.
  
  Tensor products of algebras are usually taken over $k$ unless specified otherwise or clear from context. The internal tensor product of subalgebras is characterised as follows: for subalgebras $B, C \subset A$, $A = B \otimes C$ if and only if $A = BC$, $[B,C]=0$, and $\dim_k A = (\dim_k B)(\dim_k C)$; see \cite[Proposition 4.8]{jacobson2012basicII}. One can easily check that $Z(A\otimes B) = Z(A) \otimes Z(B)$ and $\Nuc(A\otimes B) = \Nuc(A)\otimes \Nuc(B)$ if $A$ and $B$ are unital algebras.

\subsection{Involutions}

An involution on a $k$-algebra is an anti-automorphism of order 2. If $(A,-)$ is an algebra with involution, since $\Char(k) \ne 2$, there is a decomposition into subspaces $A = \Skew(A, -) \oplus \operatorname{Herm}(A, -)$, where $\Skew(A, -) = \{a \in A \mid \bar a = -a\}$ and $\operatorname{Herm}(A, -) = \{a \in A \mid \bar a = a \}$. It is plain to see that $\Nuc(A)$ and $Z(A)$ are stabilised by the involution. We define $\Nuc(A,-) = \Nuc(A)\cap \operatorname{Herm}(A,-)$ and $Z(A, -) = Z(A) \cap \operatorname{Herm}(A, -)$.
 A unital $k$-algebra with involution $(A, -)$ is \emph{central} if $Z(A, -) = k1$ and \emph{simple} if the only two-sided ideals stabilised by the involution are $\{0\}$ and~$A$.
 
 There is a group scheme $\mathbf{Aut}(A, -)$ consisting of automorphisms that commute with the involution, and a Lie algebra $\Der(A, -) = \Lie(\mathbf{Aut}(A,-))$ consisting of derivations that commute with the involution.  If $(A,-)$ is a unital central simple associative algebra with involution, we write $\Sim(A,-) = \{x \in A \mid \bar x x = x \bar x \in k^\times 1 \}$ and $\Iso(A,-) = \{x \in A \mid \bar x x = x \bar x =1\}$. For the corresponding definitions of $\mathbf{Sim}(A,-)$ and $\mathbf{Iso}(A,-)$ as group schemes, we refer to \cite[23.A]{knus1998book}.

\subsection{Composition algebras}

We refer to \cite{jacobson1958composition} and \cite{springer2000exceptional} for the definition and main properties of composition algebras, and this includes the assumption that composition algebras are unital. If $C$ is a composition algebra then we write $C_0$ for the kernel of the trace map; equivalently, $C_0= \Skew(C)$ relative to the standard involution. This subspace equipped with the commutator product  is a central simple Malcev algebra denoted by $C_0^-$.

\subsection{Structurable algebras}
Let $R$ be a commutative ring in which $2, 3 \in R^\times$. Given a unital $R$-algebra with involution $(A,-)$, we define linear endomorphisms $V_{x,y} \in \End(A)$:
	\begin{align*}
	V_{x,y} z = \{x,y,z\} = (x\bar y)z + (z\bar y)x - (z \bar x) y && \text{for } x,y,z \in A.
	\end{align*}
	The algebra $(A,-)$ is called a \emph{structurable algebra} if the identity
	\begin{align*}
		[V_{x,y}, V_{z,w}] = V_{\{x,y,z\},w} - V_{z,\{y,x,w\}}
	\end{align*}
	holds in $\End_R(A)$, for all $x,y,z,w \in A$.
	We refer to \cite{allison1978class,boelaert2019moufang} for the main properties of {structurable algebras} and various facts about the key operators $L_x, R_x, T_x, U_{x,y}, D_{x,y} \in \End_R(A)$, defined as follows:
	\begin{align*}  &L_x(z) = xz, & & T_x(z) = V_{x,1}(z), && D_{x,y}(z) = \tfrac{1}{3}\big[[x,y] +[\bar x, \bar y],z\big] + [z,y,x] - [z,\bar x, \bar y], \\ \notag
	&R_x(z) = zx,	& & U_{x,y}(z) = V_{x,z}(y)
	\end{align*}
	for all $x,y,z \in A$.	 We also write $U_{x,x} = U_{x}$ for short.
% The scalar $\frac{1}{3}$ in the $D$-operators is not necessary for this paper, but it does play a role in~\cite{allison1978class}.
Given a subspace $B \subset A$, we write $V_{B,B}$ for the span of $\{V_{x,y} \mid x, y \in B\}$, $L_BL_B$ for the span of $\{L_x L_y \mid x, y \in B \}$, and similarly for $T_B$, $D_{B,B}$, and so on.  We further define the antisymmetric bilinear mapping
 \begin{align} \label{eq:psi}
 	\psi: A \times A \to S, && \psi(x,y) = x \bar y - y \bar x
 \end{align}
 which plays an important role in the definition of the TKK Lie algebra of $(A,-)$ (see \ref{sec:TKK-construction}).

\subsection{Examples of structurable algebras}
 The class of structurable algebras is vast and varied. Any alternative algebra with any involution whatsoever is a structurable algebra \cite[p.\! 411]{schafer1985structurable}. All Jordan algebras equipped with trivial involutions are structurable algebras. If $(A,\sigma)$ is any associative algebra with involution and $\mu \in Z(A,\sigma)$ is invertible, then the algebra $(\operatorname{KD}(A,\mu),*) = (A \oplus A,*)$ obtained by the \emph{classical} Cayley--Dickson construction \cite[\S2.5]{mccrimmon2006taste} is a structurable algebra if one considers it with the nonstandard involution $*: (a,a') \mapsto (\sigma(a),a')$.

Of course, simple structurable algebras are of interest for the same reason that the simple objects are interesting in any variety of algebras, especially a variety that admits a Wedderburn-type theorem like \cite[Theorem~7]{schafer1985structurable}. But simple structurable algebras are interesting also because they are linked to simple Lie algebras and simple algebraic groups.

An erroneous classification of simple structurable algebras over fields of characteristic 0 appeared in Allison's original paper from 1978 \cite[Theorem 25]{allison1978class}. This classification was later corrected and extended to fields of characteristic 0 or $p > 5$ by Smirnov~\cite[Theorem~3.8]{smirnov1990simple}.

 \subsection{Conjugate inverses and isotopes} \label{sec.conj.inverse} \label{sec.isotop}
 Suppose $x \in A$ for a structurable $R$-algebra $(A,-)$. Then $x$ is called \emph{conjugate-invertible} if there exists a \emph{conjugate inverse} $y \in A$ such that $V_{x,y} = \id$. The conjugate inverse of any element is unique if it exists, so we write it as $y = \hat x$. In fact, if $x$ is conjugate-invertible then $U_x$ is invertible and $\hat x = U_x^{-1}(x)$. It follows that if $(B,-)$ is a subalgebra of $(A,-)$ and $b \in B$, then $b$ is conjugate-invertible in $(A,-)$ if and only if it is conjugate-invertible in $(B,-)$. For more details on conjugate-invertibility, we refer to  \cite{allison1981isotopes}. We say that $(A,-)$ is a structurable division algebra if  every nonzero element is conjugate-invertible.
 
 A linear bijection $\alpha: A \to B$ between two structurable $R$-algebras is called an \emph{isotopy} if $\alpha V_{x,y} \alpha^{-1} = V_{\alpha(x), \beta(y)}$ for some linear bijection $\beta: A \to B$. The map $\beta$ is uniquely determined if it exists, so we write it as $\beta = \hat \alpha$. Two structurable algebras are called \emph{isotopic} if there exists an isotopy between them.

 \subsection{Norms of structurable algebras} \label{sec:traces-and-norms}
 In most structurable $k$-algebras $(A,-)$, there exists a homogeneous polynomial function $N_A: A \to k$ that detects invertibility in the sense that $N_A(x) \ne 0$ if and only if $x$ is conjugate-invertible. If $N_A$ is of minimal degree and is normalised so that $N_A(1) = 1$, then it is unique and called \emph{the norm} of $A$. The norm certainly exists and equals the generic norm if $(A,-)$ is Jordan or alternative, and it also exists as long as $(A,-)$ is simple or $\Char(k)=0$.  In some familiar examples, particularly when $N_A$ coincides with the generic norm, it depends only on the algebra structure of $A$ and not on its involution. However, we caution that in general $N_A$ does depend on the involution as well as the algebra structure. Further details about norms on structurable algebras can be found in \cite{allison1992norms}.

\subsection{Algebraic groups acting on a structurable algebra}

Let $(A, -)$ be a structurable algebra over $k$ with norm $N_A$.  There are several important algebraic groups acting on $A$:
\[
\mathbf{Aut}(A, -)\quad \subset\quad  \mathbf{Str}(A, - )\quad  \subset \quad \mathbf{GO}(N_A)\quad \subset \quad \mathbf{GL}(A).
\]

The \emph{structure group} $\mathbf{Str}(A, - )$ is an algebraic group that can be defined in any of the following ways (see \cite{allison1981isotopes}):
\begin{enumerate}[(i)] 
	\item The group whose $R$-points are isotopies from $(A_R,-)$ to itself.
	\item The group whose $R$-points are  involution-preserving $R$-algebra isomorphisms from $(A_R,-)$ to an isotope $(A_R^{\langle u\rangle},-^{\langle u\rangle})$ for some $u \in {A_R}^*$. (For the definition of the $u$-isotope, we refer to \cite{allison1981isotopes} or \cite[\S2.5]{boelaert2019moufang}.)
	\item The image of  the homomorphism $\pi_1: \mathbf{Aut}_{\rm gr}(K(A, -)) \to \mathbf{GL}(A)$ that sends a graded automorphism to its $+1$-component.
	\item The image of  the homomorphism $\pi_{-1}: \mathbf{Aut}_{\rm gr}(K(A, -)) \to \mathbf{GL}(A)$ that sends a graded automorphism to its $-1$-component.
\end{enumerate}

By \cite[Proposition 12.3]{allison1981isotopes}, every graded automorphism of $K(A, - )$ is uniquely determined by its restriction to the $+1$-component or the $-1$-component, so  $\pi_1$ and $\pi_{-1}$ are in fact isomorphisms onto their images. The map $\pi_{-1} \circ \pi_{1}^{-1} = \pi_{1} \circ \pi_{-1}^{-1}: \alpha \mapsto \hat \alpha$ is an order~2 automorphism of $\mathbf{Str}(A, -)$ that fixes $\mathbf{Aut}(A, - )$ pointwise; it is the same as the map $\alpha \mapsto \hat \alpha$ from \ref{sec.isotop}. We can identify $\mathbf{Aut}(A, -)$ as the subgroup of $\mathbf{Str}(A, - )$ that fixes $1 \in A$ \cite[Corollary~8.6]{allison1981isotopes}.

%The following proposition is a variant of \cite[Ch IV. Lemma 4.2]{seligman1976rational}.
%
%\begin{proposition} \label{isotopy-prop}
%	Let $(A, - )$ be a central simple structurable algebra and let $\alpha \in \GL(A)$. If $\alpha V_{A,A} \alpha^{-1} = V_{A,A}$ and there exists a nonzero associative bilinear form $b$ on $K(A, -)$ such that $b(X_0,Y_0) = b(\alpha X_0 \alpha^{-1}, \alpha Y_0 \alpha^{-1})$ for all $X_0, Y_0 \in V_{A,A}$, then $\alpha \in \Str(A, -)$.
%\end{proposition}
%
%\begin{proof}
%	Note that  $b$ is nondegenerate, since it is nonzero and $K(A, -)$ is simple. Since $\Char(k) \ge 5$, the subspaces $K_i$ and $K_j$ of $K(A, -)$ are orthogonal with respect to $b$, except when $i = -j$.  We may therefore define $\hat \alpha \in \GL(A)$ so that $b(x_{+},y_{-}) = b( \alpha(x)_+, \hat \alpha(y)_-)$ for all $x_{+} \in K_{ 1}$, $y_- \in K_{-1}$. Then for all $X_0 \in K_0 = V_{A,A}$, we have \begin{align*}
%			b\big(\alpha X_0 \alpha^{-1},[  \alpha(x)_+, \hat \alpha(y)_-]\big) &= b\big([\alpha X_0 \alpha^{-1}, \alpha(x)_+],\hat \alpha(y)_-\big) = b \big(\alpha X_0(x)_+,\hat \alpha(y)_-\big) \\
%			&= b\big(X_0(x)_+,  y_-\big) = b\big(X_0,[x_+,y_-]\big) = b\big(\alpha X_0 \alpha^{-1}, \alpha[x_+,y_-]\alpha^{-1} \big).
%		\end{align*}
%		This implies that $V_{\alpha(x), \hat \alpha(y)} =[  \alpha(x)_+, \hat \alpha(y)_-] = \alpha[x_+,y_-]\alpha^{-1} = \alpha V_{x,y} \alpha^{-1}$ for all $x, y \in A$, which is the definition of being an isotopy.
%\end{proof}
%

The group $\mathbf{GO}(N_A)$ of \emph{norm-similitudes} of $A$ is the group whose $R$-points are all $\beta \in \GL(A_R)$ such that there exists $\mu_R(\beta) \in R^\times$ with $N_A(\beta(x)) = \mu_R(\beta) N_A(x)$ for all $x \in A_R$.  It is proved in \cite[Proposition 4.7]{allison1992norms}  that $\mathbf{Str}(A,-)$ is a subgroup of $\mathbf{GO}(N_A)$. The scalar $\mu_R(\beta)$ is called the {multiplier} of $\beta$ and the map $\mu: \mathbf{GO}(N_A) \to \mathbf{G}_m$ is a homomorphism.
The kernel of $\mu$ is the {norm-preserving group} of $A$, denoted by $\mathbf{Iso}(N_A)$.

\subsection{Tensor products of two composition algebras}

Let $C(m)$ be the split composition algebra of dimension $m$ over $k^s$. We assume that $m_1, m_2 \in \{1,2,4,8\}$ but $(m_1, m_2) \ne (2,2)$ throughout the paper. The algebra $C(m_1) \otimes_{k^s} C(m_2)$ equipped with the canonical involution (obtained by tensoring the standard involutions on the two factors) is a central simple structurable algebra \cite[\S8]{allison1978class}.

Our reason for excluding $(m_1, m_2) = (2,2)$ is that $C(2)\otimes C(2) \simeq (k^s)^4$ is not a simple algebra with involution, so it will not produce simple Lie algebras through the TKK construction. However, these algebras are by no means uninteresting: the automorphism group of $C(2) \otimes C(2)$ as an algebra with involution is the dihedral group of order 8. Insights on  forms of $C(2)\otimes C(2)$ and their cohomological invariants can be found in \cite{knus2003quartic},  \cite[\S5.1]{hirsch2020decomposability}, and \cite[Exercise 3.9]{garibaldi2009cohomological}.

\begin{definition*}
	An algebra with involution $(A, -)$ is called an \emph{$(m_1, m_2)$-product algebra} if $(A, -)_{k^s} \simeq C(m_1) \otimes_{k^s} C(m_2)$ as $k^s$-algebras with involution. We say that $(A, -)$ is \emph{decomposable} if there are composition subalgebras $C_1$, $C_2 \subset A$ such that $\overline{C_i} = C_i$ and $A = C_1 \otimes_k C_2$.
\end{definition*}

%If $(A,-)$ is decomposable, the definition implies that the involution on $(A,-)$ restricts to the canonical involutions on $C_1$ and $C_2$. 

%Note that this terminology differs from  some of the references, for example \cite[pp.\! 670--671]{allison1988tensor}, because we find it more natural to make a distinction between decomposable and indecomposable algebras than between non-twisted and twisted algebras. Another subtle difference is that we have defined a decomposable product algebra as one that is the \emph{internal} tensor product of two composition subalgebras (see \ref{sec.algebras}). This makes it unnecessary to ``identify $C_1$ and $C_2$ with the subalgebras $C_1 \otimes 1$ and $1\otimes C_2$" as it was done in  \cite[p.\! 671]{allison1988tensor}.

By a theorem of Albert \cite[Theorem 16.1]{knus1998book}, every associative central simple algebra with involution of degree 4  is a tensor product of two quaternion subalgebras. These are called {biquaternion algebras}, so a $(4,4)$-product algebra is the same thing as a \emph{biquaternion algebra with orthogonal involution}. 
%If $(A,-)$ is one of these, a factorisation $A = C_1 \otimes C_2$ is not unique and it is not always possible to find a factorisation where $\overline{C_i} = C_i$. When this is possible, i.e.\! when $(A,-)$ is decomposable, then the factors $C_1, C_2$ are in fact unique (see Theorem \ref{thm.aut2} for a precise statement). 
%A $(4,4)$-algebra is usually known as a \emph{biquaternion algebra with orthogonal involution}, and every central simple algebra of degree 4 with orthogonal involution is one of these.
We sometimes call $(8,8)$-algebras \emph{bi-octonion algebras}. An $(m,1)$-product algebra is just a composition algebra with its standard involution, and an $(m,2)$-product algebra may be thought of as a composition algebra with an involution of the second kind.

An $(m_1,m_2)$-product algebra is associative if and only if $m_1, m_2 \le 4$, and it is alternative if and only if $(m_1, m_2) \ne (4,8)$ or $(8,8)$ \cite[\S1~Lemma~2]{jacobson1954kronecker}. Not only do $(4,8)$- and $(8,8)$-product algebras fail to be alternative, they also fail to be power-associative \cite[Corollary~1]{bremner2004identities}.
%Not every central simple alternative algebra with involution is an $(m_1, m_2)$-product algebra, the exceptions being associative algebras where the involution is not orthogonal or the degree is $>4$, and octonion algebras with non-standard involutions.
%The latter belong to a class of nonassociative structurable algebras constructed using hermitian forms over quaternion algebras.

\section{Automorphism groups} 

The purpose of this section is to determine the groups $\mathbf{Aut}(A,-)$ for all $(m_1, m_2)$-product algebras.

\subsection{Automorphisms and derivations of composition algebras}
\label{sec:aut-der-comp}

Let $(C, -)$ be composition algebra with its canonical involution, where $m = \dim C = 1$, $2$, $4$, or 8. Recall that $\mathbf{Aut}(C)= \mathbf{Aut}(C, -)$ \cite[p.\ 62]{jacobson1958composition} and this group is trivial if $m = 1$, the constant group~$S_2$ if $m = 2$, $\mathbf{PGL}_1(C)$ if $m = 4$ \cite[\S23]{knus1998book}, and a group of type $G_2$ if $m = 8$ \cite[Theorem 2.3.5]{springer2000exceptional}. In all cases, $\mathbf{Aut}(C)$ is smooth, and it is connected and absolutely simple except when $m = 2$. Moreover, $\Der(C) = 0$ if $m = 1$ or $2$, and $\Der(C)$ is simple and either 3- or 14-dimensional according as $m = 4$ or~$8$  \cite[Lemma 2.4.4]{springer2000exceptional}.

\subsection{Malcev algebras}
	A Malcev algebra over a field of characteristic not 2 is an anticommutative algebra $S$ that satisfies the identity
	\[
	(xy)(xz) = ((xy)z)x + ((yz)x)x + ((zx)x)y
	\]
	for all $x,y,z \in S$. Malcev algebras are a natural generalisation of Lie algebras, just like alternative algebras are a natural generalisation of associative algebras. The main results in the subject over the last half-century are neatly summarised in the editors' comments at the end of \cite{kuzmin2014structure}.
	We say that a $k$-algebra $S$ is an \emph{exceptional simple Malcev algebra} if it is a central simple  Malcev algebra that is not a Lie algebra.
%	Such algebras are always 7-dimensional (see Theorem~\ref{thm.malc7}).

\begin{theorem}[Kuzmin] \label{thm.malc7}
	Every exceptional simple Malcev algebra is $7$-dimensional and isomorphic to $C_0^-$ for some octonion algebra $C$,  unique up to isomorphism, and $\Aut(C_0^-) \simeq \Aut(C)$.
\end{theorem}

\begin{proof}
Kuzmin proved that if $C$ is an octonion algebra over $k$, then $C_0^-$ is an exceptional simple Malcev algebra over $k$, and if $S$ is an exceptional simple Malcev algebra over $k$ then there is a unique octonion structure on $k \oplus S$ such that $(1,0)$ is the unit and $(k\oplus S)_0^- = S$ (see \cite[Theorems 11--13]{kuzmin1968maltsev} or \cite[Theorems 3.11 \& 3.12]{kuzmin2014structure}).  Moreover, if $C$ and $D$ are octonion algebras then every  isomorphism $C \overset{\sim}{\to} D$ restricts to an isomorphism  $C_0^- \overset{\sim}{\to} D_0^-$, and every isomorphism  $C_0^- \overset{\sim}{\to} D_0^-$ is the restriction of a unique isomorphism $C \overset{\sim}{\to} D$ (see \cite[Remark~3.5]{morandi2001tensor} or \cite[Theorem 3.12]{kuzmin2014structure}).
\end{proof}

\subsection{Some equivalent categories of algebras} Consider the following categories, for an arbitrary field $k$ of characteristic not 2 or 3:
\begin{itemize}
\item[--]
	$\mathsf{Prod}_{m_1,m_2}(k)$ is the groupoid of $(m_1, m_2)$-product algebras over $k$, where the morphisms are involution-preserving $k$-algebra isomorphisms;
 \item[--] $\mathsf{Comp}_m(k)$ is the groupoid of $m$-dimensional composition algebras over $k$, where the morphisms are $k$-algebra isomorphisms;
 \item[--] $\mathsf{Comp}_m\mathsf{\acute{E}t}_2(k)$ is the groupoid of $m$-dimensional composition algebras over quadratic \'etale extensions of $k$, where the morphisms are $k$-algebra isomorphisms. That is, the objects are $k$-algebras either of the form $C$ for an $m$-dimensional composition algebra $C$  over a quadratic field extension $E/k$, or of the form $C_1 \times C_2$ where $C_1, C_2$ are $m$-dimensional composition algebras over $k$ (we view $C_1 \times C_2$ as a composition algebra over $E = k \times k$);
 \item[--] $\mathsf{Malc}_7(k)$ is the groupoid of exceptional simple Malcev algebras over $k$, where the morphisms are $k$-algebra isomorphisms. 
\end{itemize}

Clearly $\mathsf{Prod}_{1,m}(k)$ is equivalent to $\mathsf{Comp}_m(k)$ and $\mathsf{Prod}_{m_1, m_2}(k)$ is equivalent to $\mathsf{Prod}_{m_2, m_1}(k)$. By Kuzmin's Theorem \ref{thm.malc7} and \cite[Proposition 12.37]{knus1998book}, there is an equivalence
	\begin{align*}\mathsf{Comp}_8(k) \to \mathsf{Malc}_7(k), &&  C \mapsto C_0^-.
	\end{align*}
 We proceed to describe some more equivalences between these categories.

\begin{theorem} \label{thm.equiv1}
If $m_1 > m_2$, the functor \begin{align*}\mathsf{Comp}_{m_1}(k) \times \mathsf{Comp}_{m_2}(k) \to \mathsf{Prod}_{m_1, m_2}(k), && (C_1, C_2)\mapsto (C_1 \otimes C_2, \gamma_1 \otimes \gamma_2),\end{align*} where $\gamma_i$ is the canonical involution on $C_i$, defines an equivalence of categories. In particular, every $(m_1, m_2)$-product algebra is decomposable. \end{theorem}

\begin{proof}
For $m_2 = 1$, the statement is trivial. Let $(m_1, m_2) = (4,2)$ or $(8,2)$, and let $(A,-)$ be in $\mathsf{Prod}_{m_1,2}(k)$. It is clear that $Z(A) = Z(A_{k^s})^{\Gamma_k}$ because $Z(A_{k^s})$ is stabilised by the action of $\Gamma_k$ on $A_{k^s}$. Since $Z(A_{k^s}) = k^s \times k^s$, Galois descent \cite[Lemma III.8.21]{berhuy} implies that $Z(A_{k^s})^{\Gamma_k}$ is a quadratic \'etale extension of $k$. Of course, $\overline{Z(A)}=Z(A)$. Furthermore, there is a unique involution-invariant $m_1$-dimensional composition $k$-subalgebra $C_{A}$ of $A$ such that $A = C_A \otimes Z(A)$  (see \cite[Proposition 2.22]{knus1998book} for $m_1 = 4$ and \cite[Theorem 3.2]{pumplun2003involutions} for $m_1 = 8$). In fact, $C_A$ is the $k$-subalgebra generated by $[\Skew(A,-),\Skew(A,-)]$. Every isomorphism $(A,-) \overset{\sim}{\to} (A',-)$ in $\mathsf{Prod}_{m_1, 2}(k)$ clearly restricts to isomorphisms $Z(A) \overset{\sim}{\to} Z(A')$ and $C_A \overset{\sim}{\to} C_{A'}$ and is uniquely determined by these restrictions, so the functor $(A,-) \mapsto (C_A, Z(A))$ is inverse to the one displayed above.

Let $(m_1, m_2) = (8,4)$, and let $(A,-)$ be in $\mathsf{Prod}_{8,4}(k)$. By definition, $A_{k^s}\simeq C(8)\otimes C(4)$ as $k^s$-algebras. It is clear that $\Nuc\big(C(8)\otimes C(4)\big) = C(4)$, and the centraliser of $C(4)$ in $C(4) \otimes C(8)$ is $C(8)$. Since both $\Nuc(A_{k^s})$ and its centraliser are stabilised by $\Gamma_k$, $\Nuc(A)$ is a quaternion algebra and its centraliser $C_A(\Nuc(A))$ is an octonion algebra  such that $A = C_A(\Nuc(A))\otimes \Nuc(A)$. The involution on $A$ stabilises both factors because this is true after extending to $k^s$. Any isomorphism $(A,-)\overset{\sim}{\to} (A',-)$ in $\mathsf{Prod}_{8,4}(k)$ is, in particular, an isomorphism of algebras $A \overset{\sim}{\to} A'$ so it restricts to unique isomorphisms on the nuclei and their centralisers. The functor $(A,-) \mapsto \big(C_A(\Nuc(A)),\Nuc(A)\big)$ is inverse to the one displayed in the statement of the theorem. \end{proof}

%An interesting implication of the  proof is that every $k$-algebra isomorphism between $(4,8)$-product algebras preserves the involution. The same is true for $(8,8)$-algebras; see Theorem~\ref{thm.aut2}. In fact, for all $(m_1,m_2)\ne (4,4)$ we could have omitted ``involution-preserving'' when we defined morphisms in $\mathsf{Prod}_{m_1,m_2}(k)$ and we would have defined the same category.

\subsection{Corestriction of composition algebras} \label{sec:corestriction}

Now let $m = 4$ or $8$. Not all $(m,m)$-product algebras are decomposable, so we need to introduce a construction called the corestriction. This construction can be found 
 in \cite[(3.12)]{knus1998book} or \cite[\S2]{allison1988tensor}.  Let $C$ be an $m$-dimensional composition algebra over a quadratic \'etale extension $E/k$, with its canonical $E$-linear involution $\tau$. Let $\iota$ be the unique nontrivial $k$-automorphism of $E$. We define a set of symbols $^\iota C = \{^\iota x \mid x \in C\}$ and give it the structure of an $E$-algebra with involution, as follows:
 \begin{align*}
 	^\iota x + {^\iota y} = {^\iota(x+y)}, && ^\iota x ^\iota y = {^\iota (xy)}, && ^\iota(ex) = \iota(e)^\iota(x), && {^\iota \tau}({^\iota x}) = {^\iota \tau(x)}
 \end{align*}
 	for $x, y \in C$ and $e \in E$.  If $n: C \to E$ is the canonical norm of $C$ then $\iota.n: {^\iota C} \to E$, defined by $\iota. n({^\iota x}) = \iota(n(x))$ for all $x \in C$, is the canonical norm of $^\iota C$.
 	% $\llangle   \iota(a),\iota(b), \iota(c)  \rrangle$ is the norm form of $^\iota M$, and similarly for the quaternion case and 2-Pfister forms.

 	Now, $(^\iota C \otimes_E C,{^\iota \tau} \otimes \tau)$ is an $E$-algebra with involution.
 	The map $s: {^\iota C \otimes_E C} \to {^\iota C \otimes_E C}$, $s(^\iota x \otimes y) = {^\iota y} \otimes x$, is an $E/k$-semilinear automorphism. The set of points in $^\iota C \otimes_E C$ fixed by $s$ is a $k$-algebra, which we denote by $\cor_{E/k}(C)$. We give $\cor_{E/k}(C)$ the canonical involution, namely the restriction of $^\iota\tau \otimes \tau$.

 	If $E = k \times k$, the nontrivial $k$-automorphism is $\iota: (x,y)\mapsto (y,x)$. In this case $C = C_1 \times C_2$ for some composition algebras $C_1, C_2$ over $k$. Building $\cor_{E/k}(C)$ from the definition leads to some heavy notation, but nevertheless there are copies of $C_1$ and $C_2$ inside $\cor_{E/k}(C)$ such that $\overline{C_i} = C_i$ and $\cor_{E/k}(C) = C_1 \otimes C_2$.

 \subsection{Malcev structure on the skew subspace}  If $(A,-)$ is any structurable algebra, the commutator equips $S = \Skew(A,-)$ with the structure of a Malcev algebra \cite[Proposition~18]{allison1978class}.	If $(A,-) = C_1 \otimes C_2$ is a decomposable $(m_1, m_2)$-product algebra then it is clear that
\begin{align} \label{eq.S}
\Skew(A,-) = (C_1)_0 \oplus (C_2)_0
\end{align}
and the subspaces $(C_i)^-_0$ are ideals in the Malcev algebra $\Skew(A,-)^-$. On the other hand, if $(A,-) = \cor_{E/k}(C)$ where $E$ is a field, then $\Skew(A,-)^- = \{{^\iota s} \otimes 1 + 1 \otimes s \mid s \in C_0\}$ is a simple (but not central simple) Malcev algebra, because \begin{align} \label{eq.S2}
 	\Skew(A,-)^- \overset{\sim}{\longrightarrow} C_0^-, && {^\iota s} \otimes 1 + 1 \otimes s \longmapsto s
 	\end{align}
 	is an isomorphism.

 \begin{theorem} \label{thm.equiv2}
	If $m = 4$ or $8$, the functor
	\begin{align*}
	F_k: \mathsf{Comp}_m\mathsf{\acute{E}t}_2(k) \to \mathsf{Prod}_{m,m}(k), && C \mapsto \cor_{Z(C)/k}(C)
	\end{align*}
	defines an equivalence of categories. In particular, if  $C$ in $\mathsf{Comp}_m\mathsf{\acute{E}t}_2(k)$ and $(A, -)$ in $\mathsf{Prod}_{m,m}(k)$ correspond to each other under this equivalence, then:
\begin{enumerate}[\rm (i)]
	\item $\Aut_k(C) \simeq \Aut_k(A,-)$.
	\item The centre of $C$ is isomorphic to the centroid of the Malcev algebra $\Skew(A,-)^-$.
	\item $(A,-)$ is decomposable as $(A,-) \simeq C_1\otimes C_2$ if and only if $C\simeq C_1 \times C_2$. 
\end{enumerate}
\end{theorem}

\begin{proof}
Most of the proof concerns the main statement about the equivalence of categories. It is clear that (i) and (iii) follow from the main statement, and (ii) already follows from~\eqref{eq.S2}.

For $m = 4$, the theorem is proved completely in \cite[Theorem 15.7]{knus1998book}, so we focus on $m = 8$. Owing to the conceptual and notational differences between the decomposable and non-decomposable cases, it is convenient to follow the strategy of \cite[Proposition 12.37]{knus1998book}, instead of formally writing down an inverse to~$F_k$. This involves showing that $F_k$ induces a bijection between the isomorphism classes of objects in  $\mathsf{Comp}_m\mathsf{\acute{E}t}_2(k)$ and in $\mathsf{Prod}_{m,m}(k)$, and a bijection between the automorphism groups of $C$ and $F_k(C)$ for all $C$ in $\mathsf{Comp}_m\mathsf{\acute{E}t}_2(k)$.

Let $(A,-)$ be an $(8,8)$-product algebra, and let $S = \Skew(A,-)$. If $(A,-) = C_1 \otimes C_2$ is decomposable, then $S^-$ is a product of two exceptional simple Malcev algebras over $k$ and the centroid of $S^-$ is $k \times k$; see (\ref{eq.S}). We claim the converse: if the centroid of $S^-$ is isomorphic to $k \times k$, then $(A,-)$ is decomposable. Indeed, $S^-= S_1 \times S_2$ for a pair of simple Malcev subalgebras $S_i$. Since $[S_1, S_2] = 0$, the subspaces $k1 \oplus S_1$ and $k1 \oplus S_2$ centralise each other in $A$, and $A = (k1 \oplus S_1)\otimes(k1 \oplus S_2)$. There is an isomorphism $(A_{k^s},-) \overset{\sim}{\to} C(8)\otimes C(8)$ that can be arranged to map $k1 \oplus S_1$ into $C(8)\otimes 1$ and $k1 \oplus S_2$ into $1 \otimes C(8)$, so the subalgebras $k1\oplus S_i\subset A$ are octonion algebras over $k$. By Theorem \ref{thm.malc7}, there is only one octonion algebra structure on $k1 \oplus S_i$ compatible with $(k\oplus S_i)_0^- = S_i$. Therefore $(A,-) = F_k(C_1\times C_2)$, where $C_1 = k1 \oplus S_1$ and $C_2 = k1 \oplus S_2$ (and these factors are unique up to re-ordering!).

Any $k$-automorphism $\alpha$ of $(A,-)$ restricts to a $k$-automorphism $\alpha'$ of $S^- = S_1 \times S_2$. This $\alpha'$ must map $S_i$ isomorphically to $S_{\sigma(i)}$ for some permutation $\sigma in S_2$, because these are the unique maximal ideals of $S$.  Theorem \ref{thm.malc7} now implies that $\alpha'$ is the restriction of a unique $k$-automorphism $\alpha''$ of $C_1 \times C_2$. Then $F_k(\alpha'') = \alpha$. Moreover, if $\beta$ is any $k$-automorphism of $C_1 \times C_2$, then $F_k(\beta)'' = \beta$. Therefore $F_k$ puts the automorphism group of $C_1\times C_2$ in bijection with that of $(A,-)$.

If $(A,-)$ is a non-decomposable $(8,8)$-product algebra, the centroid of $S^-$ is a quadratic field extension $E/k$ and $S$ is an exceptional simple Malcev algebra over $E$. The extension $(S^-)_E$ has centroid $E \otimes_k E \simeq E \times E$ so it is a product $S_1 \times S_2$ of two simple Malcev $E$-algebras.  Following the reasoning from the previous paragraph, there is a unique decomposition $A_E = C_1  \otimes_E C_2$ where $C_i = E1 \oplus S_i$ is an octonion subalgebra of $A_E$.  The generator $\iota \in \mathcal{G}al(E/k)$ acts on $A_E$ by a bijection that is $k$-linear but not $E$-linear. Since the centroid of $S$ is~$E$, $\iota$ swaps the subspaces $S_1,S_2\subset A$, and it provides a bijection $C_1 \overset{\sim}{\to} C_2$ that is $k$-linear and multiplicative, but not $E$-linear. Composing this with the $k$-isomorphism $C_2 \to {^\iota C_2}$, $x \mapsto {^\iota x}$, we can naturally identify $C_1$ with $^\iota C_2$. Then $A_E = {^\iota C_2} \otimes_E C_2$, $\iota$ acts on $A_E$ by the map $^\iota x \otimes y \mapsto {^\iota y} \otimes x$, and consequently $A = \cor_{E/k}(C_2)$. The pair of $E$-algebras $C_1, C_2$ is uniquely determined by $(A,-)$, and of course $C_1 \simeq C_2$ as $k$-algebras, so the only octonion $E$-algebras $C$ such that $(A,-) \simeq \cor_{E/k}(C)$ are the ones that are $k$-isomorphic to~$C_2$.

If $\alpha$ is an automorphism of $(A,-) = F_k(C) = \cor_{E/k}(C)$, then by (\ref{eq.S2}) it induces a $k$-automorphism $\alpha'$ of $C_0^-$. If $\alpha'$ is $E$-linear then it extends to a unique $E$-automorphism $\alpha''$ of $C$ (directly applying Theorem \ref{thm.malc7}). If $\alpha'$ is not $E$-linear, then the map $C_0^- \to {(^\iota C)^-_0}$, $x \mapsto {^\iota}(\alpha'(x))$, is an $E$-isomorphism and so it extends to an $E$-isomorphism $C \to {^\iota C}$. Composing this with the $k$-isomorphism ${^\iota C \to C}$, ${^\iota x \mapsto x}$ yields a unique extension $\alpha''$ of $\alpha'$ that is actually a $k$-isomorphism of $C$. It is clear that $F_k(\alpha'') = \alpha$ and also that any $k$-automorphism $\beta$ of $C$ satisfies $F_k(\beta)'' = \beta$.
\end{proof}

\begin{proposition} \label{prop.deriv}
	 If $(A, -) = C_1 \otimes C_2$ is a decomposable $(m_1,m_2)$-product algebra, then \[\Der(A, -) = D_{A,A}  = D_{S,S} = D_{C_1,C_1}\times D_{C_2,C_2} \simeq \Der(C_1) \times \Der(C_2).\]
\end{proposition}

\begin{proof}
	Allison \cite[p.\ 148]{allison1978class} shows that $\Der(A,-)\supset D_{A,A} = D_{S,S} = D_{C_1,C_1}\times D_{C_2,C_2}$. But $D_{C_i,C_i}$ embeds as an ideal of $\Der(C_i)$ via the restriction map $D_{x,y} \mapsto D_{x,y}|_{C_i}$, and $\Der(C_i)$ is either $0$ or simple, so $D_{C_i,C_i} \simeq \Der(C_i)$. 	Let $S = \Skew(A,-)$ and $S_i= S\cap C_i$.  The rest of the proof is based on \cite[Proposition 3.6]{morandi2001tensor}. Every $d \in \Der(A,-)$ kills $1 \in A$, satisfies $d([x,y]) = [x,d(y)]+[d(x),y]$, and maps $S$ to itself, so
	\[
		d(C_i) = d(S_i) = d([S_i,S_i]) \subset [S_i,d(S_i)]+[d(S_i),S_i] \subset [S_i, S]+[S,S_i] = S_i \subset C_i.
	\]
	The map $\Der(A,-) \to \Der(C_1)\times\Der(C_2)$, $d \mapsto (d|_{C_1}, d|_{C_2})$, is evidently an isomorphism because $C_1 + C_2$ generates $A$ and $C_1 \cap C_2 = k1$ is annihilated by $\Der(A,-)$. This proves $\Der(A,-) \subset D_{C_1,C_1}\times D_{C_2,C_2}$. 
\end{proof}

\begin{lemma} \label{lem:smoothness}
If $(A,-)$ is an $(m_1, m_2)$-product algebra over $k$, then $\mathbf{Aut}(A,-)$ is smooth.
\end{lemma}

\begin{proof}
We may extend scalars if necessary and assume for the proof that ${(A,-) = C_1 \otimes C_2}$ is decomposable. Consider the canonical homomorphism $\varphi: \mathbf{Aut}(C_1)\times \mathbf{Aut}(C_2) \to  \mathbf{Aut}(A,-)$. For all $k$-algebras~$R$, $\varphi_R: \Aut_R((C_1)_R) \times \Aut_R((C_2)_R) \to \Aut_R(A_R,-)$ is visibly injective so we may consider $\mathbf{Aut}(C_1)\times \mathbf{Aut}(C_2)$ as a subgroup of $\mathbf{Aut}(A, -)$. 
It follows that 
	\begin{align*}
	\dim(\Lie(\mathbf{Aut}(A,-)))  &= \dim(\Der(A,-))\\
	&= \dim(\Der(C_1)) + \dim(\Der(C_2)) & & \text{(from \ref{prop.deriv})}\\
	&= \dim(\mathbf{Aut}(C_1)) + \dim(\mathbf{Aut}(C_2)) & & \text{(from \ref{sec:aut-der-comp})} \\
	&= \dim(\mathbf{Aut}(C_1)\times \mathbf{Aut}(C_2)) \\
	&\le \dim(\mathbf{Aut}(A,-)) \le \dim(\Lie(\mathbf{Aut}(A,-))).  & & \hfill \qedhere
	\end{align*}
\end{proof}

\begin{theorem} \label{thm.aut1}
Let $(A,-) = C_1 \otimes C_2$ be an $(m_1,m_2)$-product algebra, with $m_1 \ne m_2$. The canonical homomorphism $\varphi: \mathbf{Aut}(C_1)\times \mathbf{Aut}(C_2) \to  \mathbf{Aut}(A,-)$ is an isomorphism.
\end{theorem}

\begin{proof}
	It is clear that $\varphi$ is injective, and Theorem \ref{thm.equiv1}  proves that $\varphi_{k^a}$ is surjective. The conclusion follows from the fact that $\mathbf{Aut}(A,-)$ is smooth  (see \ref{sec.exact}).
\end{proof}

\begin{theorem} \label{thm.aut2}
Let $m = 4$ or $8$, and let $(A,-) = \cor_{E/k}(C)$ be an $(m,m)$-product algebra, where $C$ is an $m$-dimensional composition algebra over a quadratic \'etale extension $E/k$. Let $C_{/k}$ be the $k$-algebra with the same underlying set, multiplication, and $k$-vector space structure as $C$. Then $\mathbf{Aut}(A,-) \simeq \mathbf{Aut}(C_{/k})$. Consequently,
\begin{enumerate}[\rm (i)]
	\item $\mathbf{Aut}(A,-)^\circ \simeq R_{E/k}(\mathbf{Aut}(C))$, where $R_{E/k}$ is the Weil restriction.
	\item $\mathbf{Aut}(A,-)$ has two connected components, and the non-identity component has $k$-points if and only if $C \simeq {^\iota C}$ as $E$-algebras.
	\item  If $A = C_1 \otimes C_2$ is decomposable then $\mathbf{Aut}(A,-) \simeq \mathbf{Aut}(C_1 \times C_2)$ and $\mathbf{Aut}(A,-)^\circ \simeq \mathbf{Aut}(C_1)\times \mathbf{Aut}(C_2)$. \end{enumerate}
\end{theorem}

\begin{proof}
	We define 
	\begin{align*}
		\varphi_R: \Aut_R((C_{/k})_R) &\to \Aut_R(A_R,-)\\ \textstyle
		\varphi_R(\alpha)\big(\sum({^\iota x_i}\otimes y_i)\otimes r_i\big) &= \textstyle \sum \big({^\iota \alpha(x_i)}\otimes \alpha(y_i)\big)\otimes r_i
	\end{align*}
	 for all $R$-automorphisms $\alpha$ of $(C_{/k})_R = C\otimes_k R$, and all $x_i, y_i \in C$, $r_i \in R$.	Then $\varphi_R$ is injective because, for instance, $\varphi_R(\alpha)((^\iota s \otimes 1 + 1 \otimes s)\otimes 1)$ is sent to $\alpha(s\otimes 1)$ by the isomorphism $\Skew(A,-)_R \to (C_0)_R$. Theorem \ref{thm.equiv2} proves that $\varphi_{k^a}$ is surjective, and the conclusion that $\varphi$ is an isomorphism follows from the smoothness of $\mathbf{Aut}(A,-)$.
	 
	 The kernel of the natural homomorphism $\mathbf{Aut}(C_{/k}) \to \mathbf{Aut}(E) = S_2$ is isomorphic to $R_{E/k}(\mathbf{Aut}(C))$, which is connected. Now (i) follows from the uniqueness of the connected-\'etale sequence \cite[Proposition 5.58]{milne}. For (ii), it is clear that $k$-automorphisms of $C$ acting nontrivially on $E$ are the same as $E$-algebra isomorphisms $C \simeq {^\iota C}$. Finally, (iii) is just a specialisation of the main statement and (i).
	 \end{proof}

Item (iii) of the above theorem was also proved in \cite[Theorem~3.6]{aranda2020gradings}.

\begin{theorem} \label{thm:aut-schemes-involution}
If $(A,-)$ is an $(m_1, m_2)$-product algebra with $m_1 \ge m_2$, then $\mathbf{Aut}(A,-) = \mathbf{Aut}(A)$ if and only if $(m_1, m_2) = (1,1)$, $(2,1)$, $(4,1)$, $(8,1)$, $(8,4)$ or $(8,8)$.	
\end{theorem}
\begin{proof}
A theorem of Bre\v{s}ar \cite[Theorem 3.1]{brevsar2017derivations} says that \[\Der(C_1 \otimes C_2) = L_{Z(C_1)}\otimes \Der(C_2) + \Der(C_1)\otimes L_{Z(C_2)} + \ad(\Nuc(C_1 \otimes C_2)).\]
By  considering each case individually and comparing with Proposition \ref{prop.deriv}, one can show that $\Der(A,-) = \Der(A)$ if and only if $(m_1, m_2) = (8,4)$, $(8,8)$, or $(m,1)$ for some $m$. Since $\Der(A) = \Lie(\mathbf{Aut}(A))$ and $\Der(A,-) = \Lie(\mathbf{Aut}(A,-))$, this shows $\mathbf{Aut}(A,-) \neq \mathbf{Aut}(A)$ if $(m_1, m_2)$ is not in this list.

Every automorphism of a composition algebra preserves the standard involution (see \ref{sec:aut-der-comp}), which settles the case of $m_2 = 1$. In the $(8,4)$ case, we have shown in Theorem \ref{thm.aut1} that $\mathbf{Aut}(C_1 \otimes C_2) \simeq \mathbf{Aut}(C_1)\times \mathbf{Aut}(C_2)$. Since $\overline C_i = C_i$ and $\mathbf{Aut}(C_i) = \mathbf{Aut}(C_i,-)$, it follows that $\mathbf{Aut}(C_1 \otimes C_2,-) = \mathbf{Aut}(C_1 \otimes C_2)$.

For decomposable bi-octonion algebras, \cite[Lemma~3.5]{aranda2020gradings} shows that $\mathbf{Aut}(C_1 \otimes C_2,-) = \mathbf{Aut}(C_1 \otimes C_2)$. The key step is from  \cite[\S3]{morandi2001tensor}, where is it proven that   $C_1 + C_2 = k \oplus \Skew(A,-)$ has the following first-order definition in the language of algebras without involution:
\[
C_1 + C_2 = \{a \in A \mid (a,x,y) = -(x,a,y) = (x,y,a) \ \forall x, y \in A \}.
\]
 If $(A,-)$ is a non-decomposable bi-octonion algebra, then one can either use the same proof or argue that $\mathbf{Aut}(A,-) = \mathbf{Aut}(A)$ because $\mathbf{Aut}(A,-)$ is smooth and $\mathbf{Aut}(A_F,-) = \mathbf{Aut}(A_F)$ for some field extension $F/k$ which decomposes $(A,-)$. \end{proof}
 
 By Galois descent, we obtain the following corollary:
 
 \begin{corollary} \label{cor:involution-dependence}
 Let $(A,-)$ and $(A',-)$ be $(m_1, m_2)$-product algebras where $(m_1, m_2) = (1,1)$, $(2,1)$, $(4,1)$, $(8,1)$, $(8,4)$ or $(8,8)$. Then $(A,-) \simeq (A',-)$ as algebras with involution if and only if $A \simeq A'$ as algebras.	
 \end{corollary}

\section{TKK Lie algebras} \label{sec:TKK Lie algebras}

In this section we determine the type of 5-graded Lie algebra that one gets by applying the Tits--Kantor--Koecher (TKK) construction to an $(m_1, m_2)$-product algebra. 

\subsection{$\ZZ$-graded algebras} \label{sec:Z-graded_algebras}

 A $k$-algebra $L$ is {$\ZZ^n$-graded} if it is a direct sum of subspaces $L = \bigoplus_{i \in \ZZ^n} L_i$ and $L_i L_j \subset L_{i+j}$ for all $i,j \in \ZZ^n$. We write $\mathbf{Aut}_{\rm gr}(L)$ for the subgroup of $\mathbf{Aut}(L)$ consisting of graded automorphisms: \[\mathbf{Aut}_{\rm gr}(L)(R) = \{\alpha \in \mathbf{Aut}(L)(R) \mid \alpha ((L_i)_R) \subset (L_i)_R \text{ for all } i \in \mathbb{Z}\}.\] If $L$ is $\ZZ$-graded and  $L_i = 0$ for all $|i| > n$, then we say that $L$ is {$(2n+1)$-graded}. We say $L$ is {strictly} $(2n+1)$-graded if it is not also $(2n-1)$-graded. The {trivial grading} on $L$ is the one where $L = L_0$.

Let $L = \bigoplus_{i \in \mathbb{Z}} L_i$ be a  $\mathbb{Z}$-graded algebra. Define the \emph{grading cocharacter} $\lambda: \mathbf{G}_m \to \mathbf{Aut}(L)$ by $\lambda_R(c)(x_i) = c^ix_i$ for all  $c \in R^\times$ and $x_i \in (L_i)_R$, $i \in \mathbb{Z}$. The one-parameter subgroup ${T}_{\rm gr}= \lambda(\mathbf{G}_m)$  of $\mathbf{Aut}(L)$ is called the {grading torus}. Any cocharacter $\lambda:\mathbf{G}_m \to \mathbf{Aut}(L)$ is the grading cocharacter of a unique $\ZZ$-grading on $L$ \cite[Proposition 1.28]{elduque2013gradings}.

\begin{lemma} \label{lem.centraliser} Suppose $L = \bigoplus_{i \in \mathbb{Z}} L_i$ is a $\mathbb{Z}$-graded algebra over an arbitrary field $K$, and let $G = \mathbf{Aut}(L)$. Then $\mathbf{Aut}_{\rm gr}(L) = C_G({T}_{\rm gr})$ and  $\mathbf{Aut}_{\mathrm {gr}} (L)^\circ = C_G({T}_{\rm gr})^\circ = C_{G^\circ}({T}_{\rm gr})$.
\end{lemma}

\begin{proof}
	It is straightforward to show from the definition (see \cite[\S 1.k]{milne}) that $\mathbf{Aut}_{\rm gr}(L)(R) \subset {C}_{G}({T}_{\rm gr})(R)$.	On the other hand, if $\beta \in {C}_{G}({T}_{\rm gr})(R)$ and $ x_i \in (L_i)_R$, then for all $R$-algebras $S$ and all $c \in S^\times$,
	\begin{align} \label{eq.centraliser}\lambda_{S}(c) \circ \beta(x_i) = \beta \circ \lambda_{S}(c)(x_i) = \beta( c^i x_i) = c^i \beta(x_i).
	\end{align}
	Writing $\beta(x_i) = y$ and decomposing it as $y = \sum_{j \in \mathbb{Z}} y_j$ where $y_j \in (L_j)_R$, we also have $\lambda_{S}(c) (y) = \sum_{j \in \mathbb{Z}} \lambda_{S}(c)(y_j) = \sum_{j \in \mathbb{Z}} c^j y_j = \sum_{j \in \mathbb{Z}}c^iy_j$ (comparing with (\ref{eq.centraliser})). Since the homogeneous components of $L_R$ are linearly independent, this implies $(c^j-c^i) y_j = 0$  for all $j \in \mathbb{Z}$ and all $c \in S^\times$. Letting $S = R[t]$ and $c = t$, this implies $y_j = 0$ whenever $i \ne j$. Therefore $y = y_i \in L_i$ and $\beta\in \mathbf{Aut}_{\rm gr}(L)(R)$.
	
	For the final claim,  $\mathbf{Aut}_{\rm gr}(L)^\circ = C_{G}({T}_{\mathrm{gr}})^\circ$ is a ${T}_{\rm gr}$-centralising connected algebraic subgroup of $G$, which implies $C_{G}({T}_{\mathrm{gr}})^\circ \le C_{G^\circ}(T_{\rm gr})$. On the other hand, $C_{G^\circ}({T}_{\rm gr})$ is a connected algebraic subgroup of $C_{G}({T}_{\rm gr})$ \cite[Theorem 17.38, Remark 17.40~(b)]{milne},  and therefore $C_{G^\circ}({T}_{\rm gr}) \le C_{G}({T}_{\rm gr})^\circ$.
	\end{proof}

\subsection{From adjoint simple groups to central simple Lie algebras and back again} \label{sec.adj.simple}

	Let $G$ be an adjoint simple algebraic group over $k$ (recalling that $\Char(k) \ne 2,3$) and let $L = \Lie(G)$. As a rule,  $L$ is central simple, with the only exceptions being when $\Char(k) =p>0$ and  $G$ is of type $A_{mp-1}$ for some $m \ge 1$ \cite[Table~1]{hogeweij1982almost}. In these unusual cases, the ideal $L' = [L,L]$ is  central simple and $\dim L' = \dim L - 1$ \cite[Lemma 4.1.6~(i)]{boelaert2019moufang}.
% In the few cases where $L \ne L'$ it is nevertheless true that $L'$ is a graded ideal with respect to any grading on $L$, and that every grading of $L'$ comes from a grading of $L$. Given any grading on $L$, $L'_i = L_i$ for all $i \ne 0$ \cite[Lemma~4.3.2]{boelaert2019moufang}.

Moreover, the adjoint homomorphism $\Ad: G\to \mathbf{Aut}(L)^\circ$ and the restriction homomorphism $\cdot|_{L'}: \mathbf{Aut}(L)^\circ \to \mathbf{Aut}(L')^\circ$  are isomorphisms \cite[Lemma 4.1.6~(ii)]{boelaert2019moufang}. These facts imply that there  are one-to-one correspondences:

\medskip

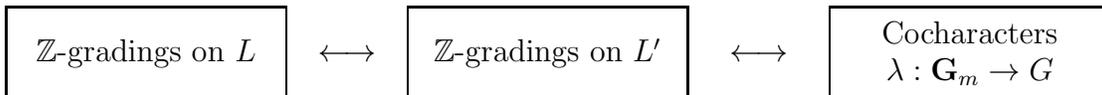
\begin{center}
\fbox{\begin{minipage}[c][1cm]{3.5cm} \centering
$\ZZ$-gradings on $L$
\end{minipage}}\quad  $\longleftrightarrow$\quad\fbox{\begin{minipage}[c][1cm]{3.5cm} \centering
$\ZZ$-gradings on $L'$
\end{minipage}} \quad $\longleftrightarrow $ \quad  \fbox{\begin{minipage}[c][1cm]{3.5cm} \centering
Cocharacters $\lambda: \mathbf{G}_m \to G$
\end{minipage}}
\end{center}

\medskip

If $L = \bigoplus_{i \in \ZZ}L_i$ is a $\ZZ$-grading on $L$, then $L'$ is a graded ideal and $L_i' = L_i$ for all $i \ne 0$ \cite[Lemma 4.3.2]{boelaert2019moufang}. Therefore $(2n+1)$-gradings on $L$ are in one-to-one correspondence with $(2n+1)$-gradings on $L'$.

\begin{lemma} \cite[Proposition 10.34, Corollary 17.59]{milne} \label{lem.L_0}
	Let $G$ be an adjoint simple group with a cocharacter $\lambda: \mathbf{G}_m \to G$, and let $L = \Lie(G)$. The subgroup $H = C_G(\lambda)$ is reductive and if  $L = \bigoplus_{i \in \ZZ}L_i$ is the $\ZZ$-grading determined by $\lambda$ then $\Lie(H) = L_0$.
	%The derived subgroup $H^{\rm der} = [H,H]$ is semisimple and  if none of its almost-simple factors are of type $A_n$, then $L_0 = Z(L_0) \times [L_0, L_0]$, $Z(L_0) = \Lie(Z(H))$, and $[L_0,L_0] = \Lie(H^{\rm der})$.
\end{lemma}

 \subsection{Labelled Dynkin diagrams} \label{sec.labelled.dynkin}

Let $G$ and $L = \Lie(G)$ be as in \ref{sec.adj.simple}. We can attach a combinatorial invariant to a $\ZZ$-grading of $L$ or, equivalently, to a pair $(G, \lambda)$ where $\lambda$ is a $G$-valued cocharacter. This invariant is the  \emph{labelled Dynkin diagram}, and it owes its origins to Dynkin himself:  he called it the \emph{characteristic} in~\cite[Ch. III]{dynkin2000semisimple}.  More background on this can be found in \cite{kac1980some,liebeck2012unipotent,stavrova2018classification}.

To define the labelled Dynkin diagram of $(G, \lambda)$, we temporarily extend scalars until $G$ has a split maximal torus. Suppose $T$ is a maximal torus in $G$ containing the image of $\lambda$. For each root $\alpha \in \Phi(G,T)$ there is a unique integer $\ell(\alpha)$ such that $\alpha\circ \lambda(c) = c^{\ell(\alpha)}$ for all $c \in (k^a)^\times$. Equivalently, the {root space} $L_\alpha$ is contained in the homogeneous component $L_{\ell(\alpha)}$ of the $\ZZ$-grading associated to $\lambda$. We can choose a Weyl chamber for $\Phi(G,T)$ containing all the roots $\alpha \in \Phi(G,T)$ with $\ell(\alpha) > 0$, thus furnishing the root system  with a base $\Pi$ such that $\ell(\beta) \ge 0$ for every $\beta \in \Pi$. The Dynkin diagram of $\Phi(G,T)$ together with the labels $\{\ell(\beta)\}_{\beta \in \Pi}$ is called the labelled Dynkin diagram of $(G,\lambda)$.

Different choices for $T$ or $\Pi$ lead to isomorphic labelled diagrams if $G$ is an exceptional group \cite[Lemma 10.1]{liebeck2012unipotent} -- and probably this is true in general, although we lack a reference. Every possible labelling of the Dynkin diagram of $G$ with non-negative integers is the labelled Dynkin diagram of some cocharacter defined over $k^a$ \cite[Lemma 10.2]{liebeck2012unipotent}.

\subsection{How to read a labelled Dynkin diagram} \label{sec.roots.gradings}

Suppose $G$ has a split maximal torus $T$ of rank $n$.  Let $X^*(T) \simeq \ZZ^n$ be the character group of  $T$ in $G$. The {root space decomposition} 
	\begin{equation*} 
	L = \bigoplus_{\omega \in X^*(T)} L_\omega = L_{\underline{0}} \oplus \Big( \bigoplus_{\alpha \in \Phi(G,T)} L_\alpha\Big)
	\end{equation*}
	 is a fine $\ZZ^n$-grading called the \emph{Cartan grading} \cite{elduque2013gradings}. Here, $L_\omega$ is the $\omega$-weight space, which is either $n$- dimensional, $1$-dimensional, or $0$-dimensional according as $\omega$ is zero, a root, or neither. The 0-weight space $L_{\underline{0}} = \Lie(T)$ is an $n$-dimensional Cartan subalgebra.	 	 
The grading induced on $L$ by a cocharacter $\lambda: \mathbf{G}_m \to T$  is the following coarsening of the Cartan grading:
\begin{equation}\label{eq.lambda-grad}
L = \bigoplus_{i \in \ZZ} L_i
\end{equation}
where \begin{align*} L_i = \bigoplus_{\substack{\alpha \in \Phi(G,T)\\ \ell(\alpha) = i}} L_\alpha \quad \text{ for }i \ne 0, && L_{0} = L_{\underline{0}} \oplus \Big(\bigoplus_{\substack{\alpha \in \Phi(G,T)\\ \ell(\alpha) = 0}} L_\alpha\Big).\end{align*}
%\note{Note: something like (\ref{eq.component-counting}) appears in \cite[Lemma 4.2.2(iii)]{boelaert2019moufang} but the expression for $L_0$ is not accurate.}

We can use the labelled Dynkin diagram to extract some information about the grading. Calculating $\dim L_i$ boils down to counting the number of roots $\alpha = \sum_{\beta \in \Pi} m_\beta(\alpha) \beta \in \Phi(G,T)$ such that $\ell(\alpha) = \sum_{\beta \in \Pi} m_\beta(\alpha)\ell(\beta) = i$. This clearly depends only on the labelled Dynkin diagram of $(G,\lambda)$, since it determines $\ell(\alpha)$ by linearity, for all roots $\alpha \in \Phi(G,T)$.
%If all the labels on the diagram are divisible by $n$, then $L_i = 0$  unless $i \equiv 0 \pmod n$. 
One can  calculate the support of the grading very easily using the coefficients of the highest root. %The highest roots of root systems can be looked up in~\cite{springer1966some} or \cite{bourbaki}.
If $\tilde \alpha = \sum_{\beta \in \Pi} m_\beta(\tilde \alpha) \beta \in \Phi(G,T)$ is the highest root with respect to $\Pi$ then let $n = \ell(\tilde \alpha) = \sum_{\beta \in \Pi} m_\beta(\tilde \alpha) \ell(\beta)$. Since $\tilde \alpha$ is the highest root, we have $n \ge \sum_{\beta \in \Pi} m_\beta( \alpha) \ell(\beta) =  \ell(\alpha)$ for all $\alpha \in \Phi(G,T)$, so $\lambda$ induces a strict $(2n+1)$-grading.

If $G$ does not have a split maximal torus $T$, the $\ZZ$-grading (\ref{eq.lambda-grad}) is still defined over $k$ provided that $\lambda$ is defined over $k$. The dimensions of the $L_i$ are unaffected by scalar extension, so there is no harm in going to an algebraic closure and doing those calculations there.

\subsection{The TKK construction} \label{sec:TKK-construction}

The TKK construction takes a structurable algebra $(A, - )$ as input and gives a $5$-graded Lie algebra as output: \[K(A, - ) =  K_{-2} \oplus K_{-1} \oplus K_0 \oplus K_1 \oplus K_2\]
where $K_{\pm 1} = \{a_\pm \mid a \in A\}$ are copies of $A$, $K_{\pm 2} = \{s_\pm \mid s \in \Skew(A, -)\}$ are copies of $\Skew(A, -)$, and $K_0 = V_{A,A}=  T_A \oplus D_{A,A}$ \cite[p.~139]{allison1978class}.

The Lie bracket is defined on $K(A,-)$ by linearly extending certain bilinear maps $K_i \times K_j\to K_{i+j}$. For example, the bracket $K_1 \times K_1 \to K_2$ is defined as
\[
[x_+, y_+] = \psi(x,y)_+
\]
for all $x, y \in A$, where $\psi$ is defined as in \eqref{eq:psi}. We refer to 
\cite[Definition 2.4.1]{boelaert2019moufang} for the remaining details. Note that $K_0$ is a Lie subalgebra of $\mathfrak{gl}(A)$ and $K(A, -)$ is a $K_0$-module. The linear map $A \to T_A$, $a \mapsto T_a$, is bijective. The subspace $ D_{A,A}$ is an ideal of  $\operatorname{Der}(A, -)$ \cite[p.\! 139]{allison1979models}, and the subspace $L_S L_S$ is an ideal of $K_0$ \cite[Corollary 5~(vii)]{allison1978class}.

If $(A, -)$ is a central simple structurable algebra then $K(A, -)$ is a central simple Lie algebra, and conversely~\cite[\S5]{allison1979models}. Another important fact from \cite[Proposition 12.3]{allison1981isotopes} is that for any pair of structurable algebras $(A,-)$ and $(B,-)$, restriction onto the 1-component puts the set of graded isomorphisms $K(A,-) \overset{\sim}{\to} K(B,-)$ in natural bijection with the set of isotopies $(A,-) \to (B,-)$.

If $L$ is a simple Lie algebra over a field of characteristic 0 or $p > 3$, we have a classification, due to Stavrova \cite[Theorem~5.10]{stavrova2018classification}, of all possible gradings on $L$ such that $L$ is graded-isomorphic to $K(A,-)$ for some central simple structurable algebra $(A,-)$. The labelled Dynkin diagrams of these gradings are symmetric and made of 0's and either one or two 1's. In particular, it is an interesting discovery (long known in characteristic 0, but quite new in characteristic $p > 3$) that every simple Lie algebra $L$ admitting a 5-grading is isomorphic to $K(A,-)$ for some central simple structurable algebra $K(A,-)$ \cite[Theorem~1.1]{stavrova2018classification}.

By Proposition \ref{prop.deriv}, if $(A,-)$ is an $(m_1,m_2)$-product algebra then
\begin{align} \label{eq.dim.formula}
\dim K(A, -) &= 2\dim S + 3 \dim A + \dim \Der(A, -)\\ \notag & = 2(m_1 + m_2-2) + 3m_1 m_2 + \dim \Der(C_1) + \dim \Der(C_2).
\end{align}

%These data are collected in the following table. 
%\[
%\begin{array}{c|ccc|c}
% (m_1,m_2) & \dim S & \dim A & \dim \Der(A, -) & \dim K(A, -) \\ \hline
%	(1,1) & 0 & 1 & 0 & 3 \\
%	(1,2) & 1 & 2 & 0 & 8 \\
%	(1,4) & 3 & 4 & 3 & 21 \\
%	(2,4) & 4 & 8 & 3 & 35 \\
%	(4,4) & 6 & 16 & 6 & 66\\
%	(1,8) & 7 & 8 & 14 & 52 \\
%	(2,8) & 8 & 16 & 14 & 78\\
%	(4,8) & 10 & 32 & 17 & 133\\
%	(8,8) & 14 & 64 & 28 & 248 
%\end{array}
%\]

\begin{theorem} \label{thm.types}
	If $(A, -)$ is an $(m_1, m_2)$-product algebra, the  labelled Dynkin diagram of the $\ZZ$-graded central simple Lie algebra $K(A, -)$ is given by Table \ref{table.main}.
\end{theorem}

\begin{table}[hbt]
\begin{center}
\begin{tabular}{c|x{0.5cm}x{1.5cm}x{2cm}x{3cm}x{3cm}}
                   $(m_1, m_2)$    &     & $1$ & $2$  & $4$   & $8$  \\ \hline
  $1$ &&
% (1,1)
\begin{pspicture}(0,0)(1.5,1.25)
	\psset{arcangle=15,nodesep=2pt}
	\rput(0,-0.5){
	\scalebox{0.9}{
	\rput(0,0.75){$\bullet$}\rput(0,0.5){\tiny $1$}
	}}
\end{pspicture} & 
% (2,1)
\begin{pspicture}(0,0)(1.5,1.25)
	\psset{arcangle=15,nodesep=2pt}
	\rput(0,-0.5){
	\scalebox{0.9}{
	\rput(0,0.75){$\bullet$}\rput(0,0.5){\tiny $1$}
	\rput(0.5,0.75){$\bullet$}\rput(0.5,0.5){\tiny $1$}	
	\psline(0,0.75)(0.5,0.75)
	}
}
\end{pspicture}
  & 
% (4,1)
\begin{pspicture}(0,0)(2.75,0.75)
	\psset{arcangle=15,nodesep=2pt}
	\rput(0,-0.5){
	\scalebox{0.9}{
	\rput(0,0.75){$\bullet$}\rput(0,0.5){\tiny $0$}
	\rput(0.5,0.75){$\bullet$}\rput(0.5,0.5){\tiny $1$}	
	\rput(1,0.75){$\bullet$}\rput(1,0.5){\tiny $0$}	
	\psline(0,0.75)(0.5,0.75)
	\psline(0.5,0.80)(1,0.80)
	\psline(0.5,0.70)(1,0.70)
	\rput(0.75,0.75){\small     $<$}
	}
}
\end{pspicture} 
  &
% (8,1)
\begin{pspicture}(0,0)(2.75,0.75)
	\psset{arcangle=15,nodesep=2pt}
	\rput(0,-0.5){
	\scalebox{0.9}{
	\rput(0,0.75){$\bullet$}\rput(0,0.5){\tiny $0$}
	\rput(0.5,0.75){$\bullet$}\rput(0.5,0.5){\tiny $0$}	
	\rput(1,0.75){$\bullet$}\rput(1,0.5){\tiny $0$}	
	\rput(1.5,0.75){$\bullet$}\rput(1.5,0.5){\tiny $1$}	
	\psline(0,0.75)(0.5,0.75)
	\psline(0.5,0.80)(1,0.80)
	\psline(0.5,0.70)(1,0.70)
	\rput(0.75,0.75){\small     $>$}
	\psline(1,0.75)(1.5,0.75)
	}
}
\end{pspicture}  \\
                        $2$ & &     &      & 
% (4,2)
\begin{pspicture}(0,0)(2.75,1.25)
	\psset{arcangle=15,nodesep=2pt}
	\rput(0,-0.5){
	\scalebox{0.9}{
	\rput(0,1){$\bullet$}\rput(0,0.75){\tiny $0$}
	\rput(0.5,1){$\bullet$}\rput(0.5,0.75){\tiny $1$}	
	\rput(1,1){$\bullet$}\rput(1,0.75){\tiny $0$}	
	\rput(1.5,1){$\bullet$}\rput(1.5,0.75){\tiny $1$}	
	\rput(2,1){$\bullet$}\rput(2,0.75){\tiny $0$}
	\psline(0,1)(2,1)
	}
}
\end{pspicture}
                         &
% (8,2)
\begin{pspicture}(0,0)(2.75,1.25)
	\psset{arcangle=15,nodesep=2pt}
	\rput(0,-0.5){
	\scalebox{0.9}{
	\rput(0,0.75){$\bullet$}\rput(0,0.5){\tiny $1$}
	\rput(0.5,0.75){$\bullet$}\rput(0.5,0.5){\tiny $0$}	
	\rput(1,0.75){$\bullet$}\rput(1,0.5){\tiny $0$}	
	\rput(1,1.25){$\bullet$}\rput(1,1.5){\tiny $0$}	
	\rput(1.5,0.75){$\bullet$}\rput(1.5,0.5){\tiny $0$}	
	\rput(2,0.75){$\bullet$}\rput(2,0.5){\tiny $1$}			
	\psline(0,0.75)(2,0.75)
	\psline(1,0.75)(1,1.25)
	}
}
\end{pspicture} \\
                        $4$  & &     &      &
% (4,4) 
\begin{pspicture}(0,0)(2.75,1.25)
	\psset{arcangle=15,nodesep=2pt}
	\rput(0,-0.5){
	\scalebox{0.9}{
	\rput(0,1){$\bullet$}\rput(0,0.75){\tiny $0$}
	\rput(0.5,1){$\bullet$}\rput(0.5,0.75){\tiny $0$}	
	\rput(1,1){$\bullet$}\rput(1,0.75){\tiny $0$}	
	\rput(1.5,1){$\bullet$}\rput(1.5,0.75){\tiny $1$}	
	\rput(2,1){$\bullet$}\rput(2,0.75){\tiny $0$}
	\rput(2.433,1.25){$\bullet$}\rput(2.433,1.50){\tiny $0$}			
	\rput(2.433,0.75){$\bullet$}\rput(2.433,0.50){\tiny $0$}					
	\psline(0,1)(2,1)
	\psline(2,1)(2.433,1.25)
	\psline(2,1)(2.433,0.75)
	}
}
\end{pspicture}
  & 
%(8,4)
  \begin{pspicture}(0,0)(2.75,1.25)
	\psset{arcangle=15,nodesep=2pt}
	\rput(0,-0.5){
	\scalebox{0.9}{
	\rput(0,0.75){$\bullet$}\rput(0,0.5){\tiny $0$}
	\rput(0.5,0.75){$\bullet$}\rput(0.5,0.5){\tiny $0$}	
	\rput(1,0.75){$\bullet$}\rput(1,0.5){\tiny $0$}	
	\rput(1,1.25){$\bullet$}\rput(1,1.5){\tiny $0$}	
	\rput(1.5,0.75){$\bullet$}\rput(1.5,0.5){\tiny $0$}	
	\rput(2,0.75){$\bullet$}\rput(2,0.5){\tiny $1$}
	\rput(2.5,0.75){$\bullet$}\rput(2.5,0.5){\tiny $0$}				
	\psline(0,0.75)(2.5,0.75)
	\psline(1,0.75)(1,1.25)
	}
}
\end{pspicture}  \\
                        $8$ & &     &      &       &
%(8,8)
\begin{pspicture}(0,0)(2.75,1.25)
	\psset{arcangle=15,nodesep=2pt}
	\rput(0,-0.5){
	\scalebox{0.9}{
	\rput(0,0.75){$\bullet$}\rput(0,0.5){\tiny $1$}
	\rput(0.5,0.75){$\bullet$}\rput(0.5,0.5){\tiny $0$}	
	\rput(1,0.75){$\bullet$}\rput(1,0.5){\tiny $0$}	
	\rput(1,1.25){$\bullet$}\rput(1,1.5){\tiny $0$}	
	\rput(1.5,0.75){$\bullet$}\rput(1.5,0.5){\tiny $0$}	
	\rput(2,0.75){$\bullet$}\rput(2,0.5){\tiny $0$}
	\rput(2.5,0.75){$\bullet$}\rput(2.5,0.5){\tiny $0$}			
	\rput(3,0.75){$\bullet$}\rput(3,0.5){\tiny $0$}			
	\psline(0,0.75)(3,0.75)
	\psline(1,0.75)(1,1.25)
	}
}
\end{pspicture} \\ \
\end{tabular}
\caption{Labelled Dynkin diagrams of $K(A,-)$ where $(A,-)$ is an $(m_1, m_2)$ product algebra.} \label{table.main}
\end{center}
\end{table}

\begin{proof}
%	Using the Block-Wilson-Premet-Strade classification of simple Lie algebras over algebraically closed fields of characteristic not 2 or 3,
	Stavrova \cite[Theorem 3.3]{stavrova2018classification} has shown that the only nontrivially 5-graded simple Lie algebras over algebraically closed fields of characteristic not 2 or 3 are the ``classical" ones, i.e.\! $L= [\Lie(G),\Lie(G)]$ for some adjoint simple group $G$. We are therefore in the setting of~\ref{sec.adj.simple}.

	The dimension of $K(A, -)$ is determined by \eqref{eq.dim.formula}, and this is equal to either $|\Phi| + \operatorname{rank}(\Phi)$ or $|\Phi| + \operatorname{rank}(\Phi)-1$ for some irreducible root system $\Phi$, with the second possibility only permitted for root systems of type $A_n$, $n \ge 4$. For each $(m_1, m_2)$, there is exactly one possibility for $\Phi$, except when $(m_1, m_2) = (2,4)$ because $C_3$ has the same number of roots and the same rank as $B_3$. But $B_3$ is not the type of $K(A,-)$ when $(m_1, m_2) = (2,4)$ because the corresponding split Lie algebra has no 5-gradings with components of dimension $(3,4,7,4,3)$. This data determines the absolute type (or unlabelled Dynkin diagram) of the Lie algebra $K(A,-)$.
	
	To determine the labels on the diagrams, we checked all the possible labellings corresponding to 5-gradings on the Lie algebras.  There are usually not many strict 5-gradings, and one can single them out using the highest root method from  \ref{sec.roots.gradings}. In each case there is only one grading with components of dimension exactly $(\dim S, \dim A, \dim V_{A,A}, \dim A, \dim S)$. This purely combinatorial task was done with the assistance of the \emph{Root Systems} package in SageMath \cite{sagemath}.
\end{proof}

\subsection{Automorphism group schemes of TKK Lie algebras} \label{sec:autKA} Let $(A,-)$ be a central simple structurable algebra, and let $G = \mathbf{Aut}(K(A,-))$. Then $G^\circ$ is an adjoint absolutely simple algebraic group (see \cite[Theorem 4.1.1]{boelaert2019moufang} and \cite[Theorem 4.7]{stavrova2018classification}) and it has $k$-rank $\ge 1$ because $K(A,-)$ is a $\ZZ$-graded algebra (see \ref{sec:Z-graded_algebras}). If $G$ is split, the quotient of $G(k)$ by the connected component of its identity (in the Zariski topology) is isomorphic to the automorphism group of the Dynkin diagram	 of $K(A,-)$  \cite[4.7]{steinberg1961automorphisms}. Therefore $\pi_0(G)$ is an \'etale group scheme of order $1$, $2$, $3$, or $6$ according to the number of symmetries of the Dynkin diagram.
In particular,  if $(A,-)$ is an $(m_1, m_2)$-product algebra then $G$ is connected if and only if $(m_1, m_2) = (8,8)$, $(8,4)$, $(8,1)$, $(4,1)$, and $(1,1)$. For roots systems  $F_4$ and $E_8$, the root lattice coincides with the weight lattice \cite[24.A]{knus1998book} and this implies $G$ is simply connected when $(m_1, m_2) = (8,8)$ or~$(8,1)$.

\begin{lemma} \label{lem:connected str}
	Let $(A,-)$ be a central simple structurable algebra, let $G = \mathbf{Aut}(K(A,-))$, and let $\lambda: \mathbf{G}_m \to G$ be the grading torus. Then $\pi_0(C_G(\lambda)) \simeq \pi_0(G)$. 	
\end{lemma}

\begin{proof}
The labelled Dynkin diagram of $K(A,-)$ is preserved by all graph automorphisms \cite[Theorem 5.10]{stavrova2018classification}. This implies that each of the connected components in $G$  has nonempty intersection with $C_G(\lambda)$. By Lemma  \ref{lem.centraliser}, the group $C_G(\lambda) \cap G^\circ = C_{G^\circ}(\lambda)$ is connected, so the inclusion map induces an isomorphism $\pi_0(C_G(\lambda)) \simeq \pi_0(G)$. 	
\end{proof}

\subsection{Structure groups: a preview} \label{sec:connected str}  Let $(A,-)$ be an $(m_1, m_2)$-product algebra, and let $H = \mathbf{Str}(A,-)$. Recall that $H \simeq  C_G(\lambda)$, where $\lambda$ is the grading torus in $G$. In particular, by Lemma \ref{lem:connected str}, $H$ is connected if $(m_1, m_2) = (8,8)$, $(8,4)$, $(8,1)$, $(4,1)$, and $(1,1)$, and otherwise it has two connected components.

 The connected structure group $H^\circ$ is a Levi subgroup of a certain parabolic subgroup of $G^\circ$ \cite[Proposition 20.4]{borel}. Given the data from Table \ref{table.main}, it is already possible to determine the isogeny class of the semisimple group $(H^\circ)^{\rm der}$ in the split case -- information that appears later in Table \ref{table.structure-groups}~(E). This task is as easy as deleting the vertices labelled with 1's in Table \ref{table.main}, and writing down the names of the Dynkin diagrams that remain. The validity of this method follows from \cite[Theorem 8.1.5(i)]{springer}, which shows that $(H^\circ)^{\rm der}$ is the semisimple group generated by the root groups $\{U_\alpha \mid \alpha \in \Phi(G^\circ,T) , \ell(\alpha) = 0\}$.

When $G^\circ$ is split, it should be possible in principle to determine $H^\circ$ and $(H^\circ)^{\rm der}$ up to isomorphism (not just isogeny)  from the root datum  of  $G^\circ$. However, we have not done any such calculations with root data because we obtain more general results in the next section.
In case $(m_1, m_2) = (1,8)$ or $(8,8)$, the task is slightly easier because $G$ is simply connected, so $H^\circ = H$ and  $H^{\rm der}$ must be a simply connected semisimple group \cite[Exercise~8.4.6(6)]{springer}. This implies that   $H^{\rm der} \simeq \mathbf{Spin}_{7}$ or $\mathbf{Spin}_{14}$, these being the split simply connected groups of type $B_3$ and $D_7$ respectively.

In the next section we determine $H^\circ$ for an arbitrary $(m_1, m_2)$-product algebra over $k$.  The main technique is to study some homomorphisms between these connected structure groups and other well-known algebraic groups. These homomorphisms often fit into commutative diagrams and the problem reduces to diagram-chasing. The results obtained in \S\ref{sec:structure groups} are completely independent of those from \S\ref{sec:TKK Lie algebras}.

\section{Structure groups} \label{sec:structure groups}

The goal of this section is to  determine precisely the connected structure groups of all $(m_1, m_2)$-product algebras. We deal with the associative cases first, followed by the non-associative cases, in which a quadratic form called the Albert form plays a very important role. Some of the results in this section appear implicitly or explicitly in \cite{allison1988tensor}.  However, we make fewer restrictions on the characteristic of $k$ (only assuming $\Char(k) \ne 2,3$) and we work exclusively with algebraic groups rather than their Lie algebras. 

\subsection{Left and right multiplication operators} \label{sec:L-isotopes}

If $(A,-)$ is a structurable $R$-algebra, we define $A^* \subset A\setminus\{0\}$ to be the set of conjugate-invertible elements in $A$, and $S^* = \Skew(A,-) \cap A^*$. 
We use some important facts from \cite[\S11]{allison1981isotopes} about skew invertible elements, namely that $S^* = \{s \in S \mid L_s \in \GL(A)\}$ and $L_s$ is an isotopy for all $s \in S^*$, with $\hat {L}_{s} = L_{\hat s} = -L_s^{-1}$.
	Similar statements hold for left multiplication by nuclear elements and right multiplication by nuclear similitudes. Lacking a reference, we prove these statements below.
	  
\begin{lemma} \label{lem:nuclear-isotopes}
	If $(A,-)$ is a  structurable algebra over $k$ and $n \in \Nuc(A)$, then the following notions of invertibility are equivalent:
	\begin{enumerate}[\rm (1)]
	\item	$xn = nx = 1$ has a solution in $A$,
	\item $L_n\in \GL(A)$,
	\item $R_n \in \GL(A)$,
	\item $n$ is conjugate-invertible.
	\end{enumerate}
	Assuming these conditions are met, $\hat n = L_{\bar n}^{-1}(1)$ and $L_n$ is an isotopy with $\hat{L}_n = L_{\bar n}^{-1}$. In contrast,
		 $R_n$ is an isotopy if and only if $n\bar n \in Z(A)$, and in this case $\hat{R}_n = R_{\bar n}^{-1}$.
 \end{lemma}
 
 \begin{proof}
 We first show that (1)--(3) are equivalent. Assuming (1), for all $b \in A$ we have a solution to $ny = b$ because $n(xb) = (nx)b = b$, so (3) holds. Assuming (3), let $x = L_n^{-1}(1)$. Then $nx = L_n(x) = 1$ and since $n \in \Nuc(A)$ this implies $\id = L_{nx} = L_n L_x$, so $L_x = L_n^{-1}$. Consequently $xn = L_xL_n(1) = L_{n}^{-1}L_n(1) = 1$, so (1) holds. Symmetrically, we can prove that (1) is equivalent to (3).
  
 Assuming (1)--(3), it is clear that not only $n$ but also $\bar n$ is an invertible nuclear element, and we may write $\bar{n}^{-1} = \overline{n^{-1}}$ without causing confusion. It is straightforward to show that $V_{nx,\bar{n}^{-1}y}(nz) = nV_{x,y}(z)$ and therefore $L_n$ is an isotopy with $\hat{L}_n = L_{\bar{n}^{-1}} = L_{\bar n}^{-1}$. From this it follows that $V_{n,\bar{n}^{-1}} = L_n V_{1,1} L_{n}^{-1}= L_n\id L_n^{-1} = \id$, so (4) holds and $\hat n = \bar{n}^{-1} = L_{\bar n}^{-1}(1)$. Assuming only (4), $U_n \in \GL(A)$ by \cite[\S6]{allison1981isotopes}, so $A = U_n(A) \subset nA \subset A$. Then $L_n$ is surjective and (2) holds because $A$ is finite-dimensional.
 
 Lastly, $R_n$ is an isotopy if and only if $\Int(n) = L_n R_n^{-1}$ is too, and $\Int(n)$ is an isotopy if and only if it preserves the involution  \cite[Corollary 8.6]{allison1981isotopes}, or equivalently if $n\bar n \in Z(A)$. It is easy to show that $V_{xn,y\bar n^{-1}}(zn) = V_{x,y}(z)n$ using the fact that $n \bar n \in Z(A)$, so $\hat R_n = R_{\bar n^{-1}}=R_{\bar n}^{-1}$.
 \end{proof}

\subsection{Structure groups of associative central simple algebras with involution} \label{sec:associative}

Let $(A,-)$ be an associative central simple algebra with involution over $k$ and let  $F = Z(A)$, which is either $k$ or a quadratic \'etale extension of $k$. The most obvious subgroups of $\mathbf{Str}(A,-)$ are $\mathbf{Aut}(A,-)$ and the copy of $\mathbf{GL}_1(A)$ embedded in $\mathbf{Str}(A,-)$ by the invertible left-multiplication operators (see Lemma \ref{lem:nuclear-isotopes}). In \cite[\S10~(1)]{allison1981isotopes} it is shown that the abstract group $\Str(A,-)$ is indeed generated by  $L_{A^\times}$ and $\Aut(A,-)$, which gives a semidirect product decomposition:
\begin{equation} \label{eq:str-semidirect}
	\Str(A,-) = L_{A^\times} \rtimes \Aut(A,-).
\end{equation}
Using the fact that $\mathbf{Str}(A,-)$ and $\mathbf{Aut}(A,-)$ are smooth, one can make this into a statement about algebraic groups, namely that there is a split short exact sequence:
\begin{equation} \label{diag:associative ses}
	\begin{tikzcd}
		1 \ar[r] & \mathbf{GL}_1(A) \ar[r] & \mathbf{Str}(A,-) \ar[r,"e" ,shift left] & \mathbf{Aut}(A,-) \ar[l,"i", shift left] \ar[r] & 1.
	\end{tikzcd}
\end{equation}
However, we prefer a different description so that we can determine the derived subgroup of $\mathbf{Str}(A,-)^\circ$ and express it as an almost-direct product of simple groups. Consider the homomorphism \begin{align*}
	\phi: \mathbf{GL}_1(A) \times \mathbf{Sim}(A,-) \to \mathbf{Str}(A,-)^\circ, &&
	\phi_R(x,y) & =  L_x R_{y^{-1}}
\end{align*}
for all $x \in A_R^\times$ and $y \in \Sim(A_R,-)$. It is easy to show that $\ker(\phi_R) = \{ (c, c^{-1}) \mid c \in F_R^\times \} \simeq F_R^\times$. We have $\mathbf{Aut}(A,-)^\circ = \mathbf{Aut}_F(A,-) =  \Int(\mathbf{Sim}(A,-))$ \cite[(23.3)]{knus1998book}, so by \eqref{eq:str-semidirect} every isotopy in the identity component of $\Str(A,-)$ is of the form $L_w \Int(z) = L_w L_z R_{z^{-1}} = L_{wz}R_{z^{-1}} = \phi(wz,z^{-1})$ for some $w \in A^\times$ and $z \in \Sim(A,-)$. Since all the groups involved are smooth, $\phi$ is a surjection. This yields the exact sequence:
\[
\begin{tikzcd}
	1 \ar[r] & \mathbf{GL}_1(F)	 \ar[r] & \mathbf{GL}_1(A) \times \mathbf{Sim}(A,-) \ar[r, "\phi"] &  \mathbf{Str}(A,-)^\circ \ar[r]& 1.
\end{tikzcd}
\]

Using the notation established in \cite[pp.\!~159,~346--347]{knus1998book},
\[
\mathbf{Sim}(A,-) = \begin{cases}
	\mathbf{GO}(A,-) & \text{if }(m_1, m_2) = (4,4)\\
	\mathbf{GU}(A,-) & \text{if }(m_1, m_2) = (4,2)\\
	\mathbf{GSp}(A,-) & \text{if }(m_1, m_2) = (4,1). 	
 \end{cases}
\]

\subsection{Generalities on similitudes of quadratic spaces} \label{sec:generalities quadratic spaces}

Let $(V,q)$ be an even-dimensional quadratic space over $k$. Every isometry $\nu\in \operatorname{O}(V,q)$ induces an automorphism $\tilde C(\nu)$ of the full Clifford algebra $C(V,q)$, and every similitude $\beta\in \GO (V,q)$ induces an automorphism $C(\beta)$ of the even Clifford algebra $C^+(V,q)$  \cite[Proposition 13.1]{knus1998book}. Concretely,
\begin{align} \label{C-map}
\tilde{C}(\nu)(v_1\dots v_r) &= \nu(v_1) \dots \nu(v_r), \\
C(\beta)(v_1\dots v_{2r}) &= \mu(\beta)^{-r}\beta(v_1) \dots \beta(v_{2r})
\end{align}
for all $v_1, \dots, v_{2r} \in V \subset C(V,q)$, where $\mu(\beta) \in k^\times$ is the multiplier of $\beta$.   These maps are homomorphisms: \begin{align*}\tilde C: \operatorname{O}(V,q) &\to \Aut(C(V,q)) \\ C: \GO(V,q) &\to \Aut(C^+(V,q)).\end{align*}
Of course $\tilde{C}(\nu)$ preserves the $\ZZ/2\ZZ$-grading on $C(V,q)$, and the restriction of $\tilde{C}(\nu)$ to $C^+(V,q)$ is  $C(\nu)$.
 Moreover, $C(\beta)$ fixes the centre $Z = Z(C^+(V,q))$ if and only if $\beta \in \GO^+(V,q)$; i.e., $\beta$ is a proper similitude \cite[Proposition~13.2]{knus1998book}. In other words, $C$ restricts to a homomorphism $C: \GO^+(V,q) \to \Aut_Z(C^+(V,q))$ whose kernel is $k^\times \id$. 
 
 This all works on the level of algebraic groups, so there are canonical homomorphisms $\tilde{C}: \mathbf{O}(V,q) \to \mathbf{Aut}(C(V,q))$ and $C: \mathbf{GO^+}(V,q) \to \mathbf{Aut}_Z(C^+(V,q))$.  Since all $k$-automorphisms of $C(V,q)$ and all $Z$-automorphisms of $C^+(V,q)$ are inner, the homomorphisms
  \begin{align*}
 	\Int: \mathbf{GL}_1(C(V,q)) \to \mathbf{Aut}(C(V,q)), && \Int: \mathbf{GL}_1(C^+(V,q)) \to \mathbf{Aut}_Z(C^+(V,q))
 \end{align*}
 are both surjective.
 
 	Let $\tau: s_1\dots s_{2r} \mapsto s_{2r} \dots s_1$ be the main involution on $C^+(V,q)$.
The following subgroups of $\mathbf{GL}_1(C^+(V,q))$ are well-known and the equalities below can serve as their definitions (see \cite{knus1998book, scharlau1985quadratic}):
	\begin{align*}
 	\mathbf{\Omega}(V,q) &= \Int^{-1}(C(\mathbf{GO}^+(V,q)) & \text{\emph{(the extended Clifford group)}} \\
 	\mathbf{\Gamma^+}(V,q) &= \Int^{-1}(\tilde{C}(\mathbf{O}^+(V,q)))\\ &= \Int^{-1}(\tilde C(\mathbf{O}(V,q)))\cap \mathbf{GL}_1(C^+(V,q)) & \text{\emph{(the even Clifford group)}} \\
 	\mathbf{Spin}(V,q) &= \mathbf{Iso}(C^+(V,q),\tau)\cap \mathbf{\Gamma}^+(V,q) & \text{\emph{(the Spin group)}.} 
 	\end{align*}
  A consequence of the Cartan--Dieudonn\'e Theorem is that the $k$-points of $\mathbf{\Gamma^+}(V,q)$ are products $v_1 \dots v_{2r}$, where $v_1, \dots, v_{2r} \in V$ are vectors such that $q(v_i) \ne 0$ \cite[Theorem 4.15]{jacobson2012basicII}. This is in contrast with $\mathbf{\Omega}(V,Q)$ whose $k$-points may include units of $C^+(V,q)$ that are nontrivial linear combinations of such terms. All three of these groups act on  $(C^+(V,q),\tau)$ by involution-preserving inner automorphisms. 
 In summary, we have inclusions
 	\[
 	\mathbf{Spin}(V,q) \quad \subset \quad \mathbf{\Gamma^+}(V,q)\quad \subset\quad \mathbf{\Omega}(V,q)\quad \subset\quad \mathbf{Sim}(C^+(V,q),\tau)\quad \subset \quad \mathbf{GL}_1(C^+(V,q)).
 	\]
 	
	The \emph{vector representation} of $\mathbf{\Gamma}^+(V,q)$ is the homomorphism $\chi: \mathbf{\Gamma}^+(V,q) \to \mathbf{O}^+(V,q)$ where $\chi(x)(v) = xvx^{-1}$ for all $x \in \Gamma^+(V,q)$ and all $v \in V_R$, i.e.\! $\chi = \Int|_V$. And there is a unique homomorphism $\chi': \mathbf{\Omega}(V,q) \to \mathbf{PGO}^+(V,q)$ such that $\chi'(x) = \bar{\beta}$	where $\beta \in \GO^+(V,q)$ is a similitude such that $\Int(x) = C(\beta)$, and $\bar \beta$ is its image in $\PGO^+(V,q)$  (see \cite[(13.19)]{knus1998book}).
	
	 We have the following commutative diagram of algebraic groups, where the rows are exact, the first two columns are injective, and the third column is surjective \cite[(13.24), p.\!~352]{knus1998book}:
	\begin{equation} \label{diag:omega-exact}
		\begin{tikzcd}
		1 \arrow[r] &  \mathbf{G}_m \arrow[r] \ar[d] & \mathbf{\Gamma^+}(V,q) \arrow[r,"\chi"] \ar[d] & \mathbf{O}^+(V,q) \arrow[r] \ar[d] & 1\\
		1 \arrow[r] &  \mathbf{GL}_1(Z) \arrow[r] & \mathbf{\Omega}(V,q) \arrow[r,"\chi'"] & \mathbf{PGO}^+(V,q) \arrow[r] & 1.
		\end{tikzcd}
	\end{equation}
	Since the groups in the first and third columns are smooth and connected, so are the groups in the middle column \cite[Propositions 1.62~\&~5.59]{milne}.
	
	If $(V,q)$ is an odd-dimensional quadratic form, a modified version of this diagram is still valid but it is somewhat degenerate because the two rows are isomorphic: $\mathbf{GL}_1(Z) \simeq \mathbf{G}_m$, $\mathbf{PGO}(V,q) \simeq \mathbf{O}^+(V,q)$  \cite[Proposition~12.4]{knus1998book}. The groups $\mathbf{PGO}^+(V,q)$ and $\mathbf{\Omega}(V,q)$ are not usually defined if $\dim V$ is odd.%	If $\dim V$ is odd, this diagram is still valid but it is somewhat degenerate because the two rows are isomorphic.
	
	If $(V,q)$ is $2n$-dimensional with a split Clifford algebra, then $C^+(V,q) \simeq M_{2^{n-1}}(k) \times M_{2^{n-1}}(k)$. The \emph{half-spin representations} $\rho_i: \mathbf{Spin}(V,q) \to \mathbf{SL}_{2^{n-1}}(k)$ are the projections onto each component of $C^+(V,q)$; in general these are inequivalent representations.

\subsection{The Albert form $Q$ and the $\natural$-map}\cite[\S3]{allison1988tensor} \label{sec:QandNat}
Let $(A,- )$ be an $(m_1,m_2)$-product algebra where $m_1 \ge m_2$ and ${S} = \Skew(A, - )$.
To begin with, suppose that $(A,-)= C_1 \otimes C_2$ is decomposable and write $S_i = \Skew(C_i)= (C_i)_0$ for the skew subspace of $C_i$, and $n_i$ for its norm.
The \emph{Albert form} is the following nondegenerate quadratic form defined on ${S}= S_1 \oplus S_2$:
\begin{align} \label{eq.albert1}
Q(s_1 + s_2) = n_1(s_1) - n_2(s_2) && \text{for all } s_i \in S_i.
\end{align}
In other words $Q = n_1' \perp \langle -1\rangle n_2'$. The identity $Q = n_1 - n_2 = n_1'-n_2'$ holds in   $W(k)$.
Another important  map  is the isometry $\natural \in  \operatorname{O}(S,Q)$, defined as:
	\begin{align} \label{eq.nat}
	(s_1 + s_2)^\natural =  s_1 - s_2 && \text{for all } s_i \in S_i.
	\end{align}
Note that if $m_2 = 1$ then $(A,-)$ is just a composition algebra, $Q$ is the pure norm, and $\natural$ is the identity map.
 If $m_1 = m_2$, the definitions (\ref{eq.albert1}) and (\ref{eq.nat}) depend on a choice of decomposition for $(A,-)$. By Theorem~\ref{thm.equiv2}, the only possible decompositions are $C_1 \otimes C_2$ and $C_2 \otimes C_1$. Changing the order flips the sign of $Q$, but this shall have no material effect on anything that follows.
 	
 \subsection{$Q$ and $\natural$ for indecomposable algebras} \label{sec:QandNat2}

 If $(A,-)$ is indecomposable then it is of the form $\cor_{E/k}(C)$ for a quadratic field extension $E/k$ and an $m$-dimensional composition algebra $C$ over $E$, where $m = 4$ or $8$.  Write $n$ for the norm of $C$. 	Since $(A_E,-)={^\iota C}\otimes_E C$ is decomposable, we have an Albert form $\iota.n' \perp \langle -1 \rangle n'$ defined on $S_E$, in the sense of \ref{sec:QandNat}. Unfortunately, this quadratic form is never in the image of $\res_{E/k}: W(k)\to W(E)$ and so there is no quadratic form on $S$ whose extension to $E$ is isometric to $\iota.n' \perp \langle -1 \rangle n'$.
Nevertheless, it is possible to nominate a pair of maps $Q$ and $\natural$, which play more or less the same role as the ones from~\ref{sec:QandNat}.

Suppose that $E = k(\sqrt{d})$ for some non-square $d \in k$. Define the following nondegenerate quadratic form on $S$:
\begin{align*} 
	Q({^\iota} s\otimes 1 + 1 \otimes s) = \sqrt{d}\big(\iota(n(s))-n(s)\big) && \text{for all } s \in C_0. 
 \end{align*}
 This is the Scharlau transfer of $n'$ along the linear functional $E \to k$, $a + b \sqrt{d} \mapsto -2bd$ for all $a, b \in k$; in other words, $Q = T_{E/k}(\langle -\sqrt{d} \rangle n')$. Clearly $Q_E$ is similar to the Albert form of $(A_E,-)$; the two forms merely differ by a factor $\sqrt{d}$.
 Likewise, we define a map ${\natural}: S \to S$ as follows:
	\begin{align*}
 	({^\iota s} \otimes 1 + 1 \otimes s)^{\natural} = \sqrt{d}({^\iota s } \otimes 1 - 1 \otimes s) = -{^\iota( \sqrt{d} s)}\otimes 1 - 1\otimes (\sqrt{d}s) && \text{for all } s \in C_0.
	\end{align*}
Again $\natural_E$ is $\sqrt{d}$ times the $\natural$-map on $S_E$. The definitions of $Q$ and $\natural$ depend on a choice of $d$, and $\natural$ is no longer an isometry of $Q$ but rather a similitude  of $Q$ with multiplier~$d$.

\subsection{Generalities on invertible skew elements and isotopies} \label{sec:generalities skews}

Up to and including \ref{eq.theta}, we continue to assume that $(A,-)$ is any $(m_1, m_2)$-product algebra over $k$, $S = \Skew(A,-)$, and $Q$ and $\natural$ are as defined in \ref{sec:QandNat} or \ref{sec:QandNat2}.
Recall from \ref{sec:L-isotopes} that if $R$ is a $k$-algebra and $s \in {S}_R$, then $s$ is conjugate-invertible if and only if $L_s \in \GL(A_R)$. In that case  $L_s^{-1} = -L_{\hat s}$. Therefore, if  $c \in R$ and $s \in {S_R}$, then $c s \in {S_R}^*$ if and only if  $c \in R^\times$ and $s \in {S_R}^*$.
Given $s \in {S}_{R}$, one can calculate directly (as in \cite[Proposition 3.3]{allison1988tensor}) that
\begin{align} \label{Ls-identity}
 L_sL_{s^\natural} = -Q_R(s)\id.
\end{align}  
This implies that $s \in {S_R}^*$ if and only if  $Q_R(s) \in  R^\times$. We can calculate the conjugate inverse of $s$ if it has one: 
	\begin{equation} \label{inversion formula}\hat s =  -L_s^{-1}(1) = Q_R(s)^{-1}L_{s^\natural}(1) = Q_R(s)^{-1}s^\natural.
	\end{equation}

Recall the map $\psi: A_R \times A_R \to S_R$ defined in \eqref{eq:psi}. Given $\alpha \in \Str(A_R, - )$ we define $\alpha_S \in \End_R{(S_R)}$, 
\begin{align} \label{eq.alpha-s}
\alpha_S(s) = \tfrac{1}{2}\psi\big(\alpha(s), \alpha(1)\big) && \text{for all } s \in S.
\end{align}
It turns out (see \cite[Lemma 12.1]{allison1981isotopes}) that $\alpha_S \in \GL(S_R)$, and for all $s \in {S}_{R}$, \begin{equation} \label{eq:alphaL}L_{\alpha_{S}(s)} \hat\alpha = \alpha L_s.\end{equation} Exchanging the roles of $\alpha$ and $\hat \alpha$, it also holds that $L_{\hat \alpha_{{S}}(s)}\alpha = \hat \alpha L_s$.
	If $s, t \in S_R$, then \begin{align} \label{line-one} \alpha L_s L_t \alpha^{-1} &= (\alpha L_s \hat \alpha^{-1})( \hat \alpha L_t \alpha^{-1})  = L_{\alpha_S(s)} L_{\hat \alpha_S(t)}.
	\end{align}
	Furthermore,  if $s$ is conjugate-invertible then setting $t = \hat s$ yields
\begin{equation*} L_{\alpha_S(s)} L_{\hat \alpha_S(\hat s)} = \alpha L_s L_{\hat s} \alpha^{-1}  =  \alpha (-\id) \alpha^{-1} =  -\id.\end{equation*}
It follows that $\alpha_S(s)$ is conjugate-invertible if and only if $s$ is so. We can then derive:
\begin{align}\label{widehats}
\widehat{\alpha_{{S}}(s)} =L_{\widehat{\alpha_{S}(s)}}(1) =  -L_{\alpha_{S}(s)}^{-1}(1) = L_{\hat\alpha_{{S}}(\hat s)}(1) = \hat \alpha_{{S}}(\hat s).
\end{align}

	\begin{proposition} \label{gamma}
There is a homomorphism of algebraic groups $\gamma: \mathbf{Str}(A, -) \to \mathbf{GO}({S},Q)$ with $\gamma(\alpha) = \alpha_S$ for all $\alpha \in \Str(A,-)$, where $\alpha_S$ is the linear map defined in (\ref{eq.alpha-s}). \end{proposition}

\begin{proof}
It is shown in \cite[Lemma 12.1]{allison1981isotopes} that $\alpha_{S} (\psi(x,y)) = \psi(\alpha(x), \alpha(y))$
	for all $\alpha \in \Str(A_R, - )$ and all $s \in {S}_R$.
	It follows that \[(\alpha \beta)_{S}(s) = \frac{1}{2} \psi(\alpha\beta(s), \alpha\beta(1)) = \alpha_{S}(\frac{1}{2}  \psi(\beta(s), \beta(1))) = \alpha_{S} \beta_{S}(s)\] for all $\alpha, \beta \in \Str(A_R,- )$. Clearly $(\id_A)_{S} = \id_{S}$, so $\alpha\mapsto \alpha_S$ is a homomorphism $\Str(A_R,-)\to \GL(S_R)$.  It is also clear that this homomorphism is functorial in $R$, so it defines a morphism of algebraic groups $\mathbf{Str}(A, -) \to \mathbf{GL}({S})$.
	
	By Lemma \ref{lem.centraliser},  $\mathbf{Str}(A, -)$ is the centraliser of a torus in a smooth group, so $\mathbf{Str}(A,-)$ is smooth \cite[Theorem 13.9]{milne}. To show that the morphism of the previous paragraph factors through $\mathbf{GO}(S,Q)$, it suffices (by \ref{sec.exact}) to show that  $\gamma_{k^a}(\Str(A_{k^a},-))\subset \GO(S_{k^a},Q_{k^a})$. This is covered by the remaining part of the proof, where we assume that $k = k^a$.

	We follow the same approach as \cite[Proposition 5.2]{allison1988tensor}, but with slightly more details included, to show that for all $\alpha \in \Str(A, -)$,  $\alpha_S$ is a similitude of $Q$.
	Applying (\ref{inversion formula}) to both sides of (\ref{widehats}) yields that
	\begin{align*}  Q(\alpha_S(s))^{-1}\alpha_S(s)^\natural = \hat \alpha_S(Q(s)^{-1} s^\natural) = Q(s)^{-1} \hat \alpha_S(s^\natural) && \text{for all } s \in S^*.
	\end{align*}
	Define the rational function $\rho(s) = Q(\alpha_S(s)) Q(s)^{-1}$ on $S$, which gives \begin{align}\label{rho}\alpha_S(s)^\natural = \rho(s)\hat\alpha_S(s^\natural) && \text{for all } s \in S^*.\end{align}
	We shall show that $\rho: {S}^* \to k$ is a constant function. Fix a particular $s_0 \in {{S}}^*$, and let $t \in {{S}}^*$ be arbitrary with the intention of showing that $\rho(t) = \rho(s_0)$. Clearly $\rho(s_0) = \rho(\lambda s_0)$ for all $\lambda \in k^\times$, so we may assume that $s_0$ and $t$ are linearly independent. Working with (\ref{rho}), we obtain:
	\begin{align*}
	\rho(s_0)\hat \alpha_S(s_0^\natural) +  \rho(t)\hat \alpha_S(t^\natural) &= \alpha_S(s_0)^\natural + \alpha_S(t)^\natural= \alpha_S(s_0+t)^\natural =   \rho(s_0+t)\hat \alpha_S((s_0+t)^\natural) \\ &= \rho(s_0+t)\hat \alpha_S(s_0^\natural) + \rho(s_0+t)\hat \alpha_S(t^\natural).
	\end{align*}
	Since $s_0$ and $t$ are linearly independent, so are $\hat \alpha_S(s_0^\natural)$ and $\hat \alpha_S(t^\natural)$, and it follows that $\rho(s_0) = \rho(s_0+t) = \rho(t)$. Let $\mu(\alpha) = \rho(s_0)\in k^\times$. We have proved
	\begin{align}\label{mu}
	Q(\alpha_{S}(s)) = \mu(\alpha) Q(s) && \text{for all } s \in S^*.
	\end{align}
	Since we shown in \ref{sec:generalities skews} that $Q(s) = Q(\alpha_S(s)) = 0$ for all $s \in S\setminus S^*$, equation (\ref{mu}) holds for all $s \in S$. The conclusion is that $\alpha_S \in \GO(S, Q)$ and its multiplier is $\mu(\alpha)$.
	\end{proof}

Note that we can use (\ref{rho}) to determine the effect of composing $\wedge$ with $\gamma$:
\begin{align} \label{eq.hat-alpha}
\gamma_R(\hat \alpha)(s) = \hat \alpha_S(s) = \mu(\alpha)^{-1}\alpha_S(s^\natural)^\natural && \text{for all } \alpha \in \Str(A,-) \text{ and } s \in S.
\end{align}

\subsection{A composition property of the Albert form} \label{sec:comp-prop}
The multiplication operators $L_s$, $s \in S^*$, are important examples of isotopies (see \ref{lem:nuclear-isotopes}). Applying the map $\gamma$ to $L_s$ leads us to an interesting observation. It is easy to show from the definition that $\gamma(L_s)(t) = (L_s)_S(t) = -sts$ for all $s \in S^*$, $t \in S$. (Note that we can write $sts$ without brackets because of the skew-alternativity property of structurable algebras \cite[Proposition 1]{allison1978class}.) By (\ref{Ls-identity}), this yields $(L_s)_S(s^\natural) = -ss^\natural s = Q(s)s$ and so $\mu(L_s) = Q(s)^2$ is the relevant multiplier. Since $Q((L_s)_S(t) = Q(-sts) = Q(sts)$, we  have derived the interesting identity \begin{align} \label{eq:composition property} Q(sts) = Q(s)^2Q(t) && \text{for all } s, t \in S^*.\end{align}
Additionally, we have $L_{sts} = L_sL_t L_s$ as a consequence of skew-alternativity \cite[(2.5)]{boelaert2019moufang}, so $Q(sts) =0$ if and only if $Q(s)^2Q(t) = 0$, hence the identity \eqref{eq:composition property} holds for all $s, t \in S$.

Allison showed that the even Clifford algebra $C^+(S,Q)$ has a representation in the multiplication algebra of $A$. For the sake of completeness we give a tiny proof:
	
	\begin{proposition}\cite[Proposition 4.2]{allison1988tensor} \label{eq.theta} There exists a unique algebra homomorphism ${\theta: C^+(S, Q) \to \End_k(A)}$ such that
	\begin{align*}
	\theta(s t) = -L_{s}L_{t^\natural}
	&& \text{for all } s, t \in S.
	\end{align*}
	%The image of $\theta$ is $\End_{\Nuc(A)}(A)$, the centraliser of $R_{\Nuc(A)}$ in $\End_k(A)$.	
	\end{proposition}
	
	\begin{proof}
		By the universal property of the even Clifford algebra \cite[Lemma 8.1]{knus1998book},	it is good enough to show that $-L_s L_{s^{\natural}} = Q(s)\id$ and $(-L_r L_{s^{\natural}})(- L_{s} L_{t^\natural}) = -Q(s) L_r L_{t^\natural}$ for all $r,s,t \in S$. These identities are evident from (\ref{Ls-identity}).
	\end{proof}
		
	\subsection{Structure groups of octonion algebras} \label{sec:octonion} Let $(C,-)$ be an octonion algebra over $k$ with its standard involution and norm $n$. Since $C^+(C_0,n') \simeq M_8(k) \simeq \End_k(C)$, the representation~$\theta$ is an isomorphism. The generators $st$, $Q(s), Q(t) \ne 0$, of $\Gamma^+(C_0,n')$ are mapped to isotopies $-L_sL_{t}\in \Str(C,-)$. Since $\mathbf{\Gamma}^+(C_0,n')$ is connected, $\theta$ induces an injective homomorphism $\theta': \mathbf{\Gamma}^+(C_0,n') \to \mathbf{Str}(C,-)^\circ$. It is easy to calculate from the definition that $\gamma_R(L_s) = \gamma_R(-L_s) = -L_sR_s|_{S_R}$ for all  $s \in (C_0)_R^*$. We calculate that for all $s,t \in (C_0)_R^*$,
\[
	\gamma_R\circ \theta'_R(st) = \gamma_R(-L_s L_t) = \gamma_R(-L_s)\gamma_R(L_t) = L_sR_sL_tR_t|_{(C_0)_R}.
\]
		Let $\rho_s \in \operatorname{O}(C_0,n')$ be the reflection about an anisotropic vector $s \in C_0$. It turns out that $L_s R_s|_{C_0} = n(s)\rho_s$ \cite[p.\! 44]{springer2000exceptional}, so $\gamma_R \circ \theta'_R(st) = n(st)\rho_s\rho_t$. Meanwhile, in the vector representation it is well-known that $\chi_R(s) = \rho_s$ \cite[p.\! 239]{jacobson2012basicII}, so $\chi_R(st) =  \rho_s\rho_t$. This proves that $\gamma_R\circ \theta'_R$ agrees with $\chi_R$ {modulo scalars}. Let $\gamma': \mathbf{Str}(C,-)^\circ \to \mathbf{O}^+(C_0,n)$ be the composition of $\gamma$ and the natural surjection $\mathbf{GO}^+(C_0,n') = \mathbf{G}_m\cdot \mathbf{O}^+(C_0,n') \to \mathbf{O}^+(C_0,n')$. We have proved that the following diagram commutes:
	\begin{equation} \label{diag:octonions}
		\begin{tikzcd}
		1 \arrow[r] &  \mathbf{G}_m \arrow[r] \ar[d, "\simeq"] & \mathbf{\Gamma^+}(C_0,n') \arrow[r,"\chi"] \ar[d,"\theta'"] & \mathbf{O}^+(C_0,n') \arrow[r] \ar[d,equal] & 1\\
		1 \arrow[r] &  \mathbf{G}_m \arrow[r] & \mathbf{Str}(C,-)^\circ \arrow[r,"\gamma'"] & \mathbf{O}^+(C_0,n') \arrow[r] & 1.
		\end{tikzcd}
	\end{equation}
	Since $\chi = \gamma'\circ \theta'$ is surjective, $\gamma'$ must be surjective too. Now suppose $\alpha \in \Str(C_R,-)$ and $\gamma_R(\alpha)   = r \id $ for some $r \in R^\times$. Then $\gamma_R(\hat \alpha) = r^{-1} \id $, according to (\ref{eq.hat-alpha}). By (\ref{line-one}), we have 
	$
		\alpha L_s L_t \alpha^{-1} = L_{rs} L_{r^{-1}t} =  L_s L_t
	$
	for all $s, t \in (C_0)_R$. But $\{L_sL_t \mid s, t \in C_0\}$ generates $\End_{R}(C_R)$ as an $R$-algebra because $\theta$ is surjective. So $\alpha\in Z(\GL(C_R)) = R^\times \id \subset \Str(C_R,-)$. Therefore $\ker(\gamma')$ is the central torus $\mathbf{G}_m$. This shows that the bottom row of \eqref{diag:octonions} is exact. All the groups involved are smooth and a standard diagram chase (the five lemma) proves that $\theta': \mathbf{\Gamma}^+(C_0,n') \to \mathbf{Str}(C,-)^\circ$ is an isomorphism.
	
 	The bottom row of \eqref{diag:octonions} is exactly the same as the second exact sequence in \cite[\S4.13]{petersson2002structure}, although this is not obvious at first sight. If $C$ is an octonion division algebra, $\Str(C,-)$ is the same as the group $X_1$ studied in \cite[(37.23)]{tits2002moufang}.
 	
 	\subsection{Structure groups of $(8,m)$-product algebras} The next few subsections are all about adapting the arguments from \ref{sec:octonion} to $(8, m)$-product algebras where $m \ge 2$, i.e.\! fitting their structure groups into commutative diagrams that are easy to understand. First let us fix some notation. The right-multiplication operators give $\End_k(A)$ the structure of a right $N$-module, where $N = \Nuc(A)$. Define $\End_N(A)$ to be the ring of $N$-endomorphisms. In light of \ref{lem:nuclear-isotopes}, there is an embedding $R: \mathbf{GL}_1(N)  \to \mathbf{Str}(A,-)$. Define $\mathbf{Str}_N(A,-)$ to be the centraliser of the image of  $\mathbf{GL}_1(N)$ in $\mathbf{Str}(A,-)$. We also fix the notation $Z = Z(C^+(S,Q))$.

	\begin{lemma} \label{lem:theta-map} Let $(A,-)$ be an $(8,m)$-product algebra with $m \ge 2$. \begin{enumerate}[\rm (i)] 
	\item \cite[Theorem 4.5]{allison1988tensor} The image of $\theta: C^+(S,Q) \to \End_k(A)$ is $\End_N(A)$.
	\item The algebra homomorphism $\theta$ induces a homomorphism of algebraic groups \[\theta': \mathbf{\Omega}(S,Q) \to \mathbf{Str}_N(A, -)^\circ.\]
		\item If $m=2$ then $\theta'$ is injective. If $m = 4$ or $8$ then \[\ker(\theta'_R) = \{ce_1 + 1e_2 \mid c \in R^\times\} \simeq R^\times\] where $e_1, e_2 \in Z$ are orthogonal idempotents such that $1 = e_1 + e_2$.	
		\item If $k$ is algebraically closed, the $k$-subalgebra of $\End_k(A)$ generated by $\theta'_k(\Spin(S,Q))$ is $\End_N(A)$.
		\end{enumerate}
	\end{lemma}

	\begin{proof}
	(i) Clearly $0 \ne \theta(C^+(S,Q)) \subset \End_N(A)$ because left-multiplications commute with right-multiplications by nuclear elements. If $m = 8$ then $Q$ is a 14-dimensional form with trivial Clifford invariant, so $C^+(S,Q) \simeq M_{64}(k)\times M_{64}(k)$.
	%The two orthogonal idempotents in $Z$ are $e_1 = \frac{1+z}{2}$ and $e_2 = \frac{1-z}{2}$.
	Since $N = k$ we have $\End_N(A) = \End_k(A) \simeq  M_{64}(k)$, so $\theta$ must kill one of the full matrix subalgebras of $C^+(S,Q)$ and map the other one isomorphically onto $\End_k(A)$. If $m = 4$ then $Q$ is a 10-dimensional form whose discriminant is trivial and whose Clifford invariant is the Brauer class of $N = C_2$, so $C^+(S,Q) \simeq M_{8}(C_2)\times M_{8}(C_2)$. Then $\theta(C^+(S,Q)) \subset \End_N(A,-) =  \End_{C_2}(C_1 \otimes C_2) \simeq  M_8(C_2)$ and the conclusion follows as it did before. If $m = 2$ then $Q$ is an 8-dimensional form whose discriminant is the class of $N = Z(A) = C_2$, and $C^+(S,Q) \simeq M_8(C_2)$ \cite[Theorem~4.14]{jacobson2012basicII}. Now $\End_N(A) = \End_{C_2}(C_1 \otimes C_2) \simeq M_8(C_2)$. 	If  $C_2$ is a field, the conclusion is clear. More generally, if $s_1, \dots, s_7 \in  S_1$ and $t \in S_2$ constitute an orthogonal basis for $S$, then the centre of $C^+(S,Q)$ is $Z = k[s_1\dots s_7t] \simeq C_2$ \cite[p.\! 237]{jacobson2012basicII} and it is easy to see that \[\theta(s_1\dots s_7t) = (-L_{s_1} L_{{s_2}^\natural})(-L_{s_3} L_{{s_4}^\natural})(-L_{s_5}L_{{s_6}^\natural})(-L_{s_7}L_{t^\natural}) = -L_{s_1} \dots L_{s_7}L_t \notin k\id.\] 
	because $L_{s_1} \dots L_{s_7}L_t(C_1)\subset tC_1$ and $C_1\cap tC_1 = \{0\}$. Therefore $\theta$ is injective on~$Z$, so it is an isomorphism onto its image $\End_N(A)$.

	(ii) Since $\mathbf{\Omega}(S,Q)$ is smooth and connected, it suffices (see \ref{sec.exact}) to show that \[\theta'_{k^a}\big(\mathbf{\Omega}(S,Q)(k^a)\big) \subset \mathbf{Str}_N(A, -)(k^a).\]
	We assume $k = k^a$ for ease of notation; this implies $A = C_1 \otimes C_2$ is decomposable (and the $C_i$'s are split composition subalgebras of $A$).
	The group  $\Omega(S,Q)$ is generated by its subgroups $\Gamma^+(S,Q)$ and $Z^\times$ \cite[Lemma 13.20]{knus1998book}. Since $\theta'(s_1 s_2) = -L_{s_1}L_{{s_2}^\natural} \in \Str_N(A,-)$ for all $s_1, s_2 \in S^*$ and  $\Gamma^+(S,Q)$ is generated by elements of the form $s_1s_2$ where the $s_i$ are anisotropic, we have $\theta'(\Gamma^+(S,Q)) \subset \Str_N(A,-)$. Now $Z^\times = k^\times e_1 \times k^\times e_2$ where $e_1, e_2 \ne 1$ are a pair of orthogonal idempotents such that $e_1 + e_2 = 1$. From the proof of (i), it is clear that if $m = 4$ or $8$ then $\theta(e_i) \in \{0, \id\}$, which implies $\theta'(Z^\times) = k^\times \id \subset \Str_N(A,-)$. If $m = 2$ then $\theta$ is an isomorphism so $\theta(Z) = Z(\End_{Z(A)}(A)) = R_{Z(A)}$ and  $\theta'(Z^\times) = R_{Z(A)}^\times \subset \Str_N(A,)$ by Lemma \ref{lem:nuclear-isotopes}.

	(iii) For $m = 2$, this is clear because $\theta$ itself is injective. For $m = 4$ or $8$, it follows from the proof of (i) that $\ker(\theta'_R) = \{xe_1 + 1 e_2 \mid x \in C^+(S_R,Q_R) \} \cap \Omega(S_R,Q_R)$. Recall that $\Omega(S_R,Q_R) \subset \Sim(C^+(S_R,Q_R),\tau)$ where $\tau$ is the main involution on $C^+(S,Q)$. In this situation, $\tau$ is an involution of the second kind \cite[Proposition~8.4]{knus1998book}. So if $y = xe_1 +1e_2 \in \Omega(S_R,Q_R)$ then $y \tau(y) = x e_1 + \tau(x)e_2 \in R^\times(e_1 + e_2)$ implies $x = \tau(x) \in R^\times$. This shows that $\ker(\theta'_R) = \{ce_1 + 1e_2 \mid c \in R^\times \} \simeq R^\times$. 
	
	(iv) This is simply because $C^+(S,Q)$ is linearly spanned by $\Gamma^+(S,Q)$, and it is easy to show that $k$ being algebraically closed implies $\Gamma^+(S,Q) = Z^\times.\Spin(S,Q)$. Therefore $Z^\times \cup \Spin(S,Q)$ generates $C^+(S,Q)$ as a $k$-algebra, so the set \[\theta'_k(Z^\times \cup \Spin(S,Q)) = k^\times\id \cup \theta'_k(\Spin(S,Q))\] generates $\theta'_k(C^+(S,Q)) = \End_N(A)$ as a $k$-algebra.
	\end{proof}

Recall the definitions of the following homomorphisms from   \ref{sec:generalities quadratic spaces},   \ref{gamma}, and \ref{lem:theta-map}.
\begin{align*}\chi': \mathbf{\Omega}(V,q) \to \mathbf{PGO}^+(V,q), && \gamma: \mathbf{Str}(A,-) \to \mathbf{GO}(S,Q),&& \theta': \mathbf{\Omega}(S,Q) \to  \mathbf{Str}(A,-)^\circ. \end{align*}
Let $\gamma': \mathbf{Str}(A,-) \to \mathbf{PGO}(S,Q)$ be the composition of $\gamma$ with the natural homomorphism $\mathbf{GO}(S,Q) \to \mathbf{PGO}(S,Q)$, and let $\gamma'': \mathbf{Str}(A,-)^\circ \to \mathbf{PGO}^+(S,Q)$ be the restriction of $\gamma'$ to the identity component.  

\begin{lemma} \label{ker-gamma}
		Let $(A,-)$ be an $(8,m)$-product algebra with $m \ge 2$, and let $N = \Nuc(A)$. \begin{enumerate}[\rm (i)] \item 	
$\ker(\gamma')$ is the group of right multiplications by nuclear elements, i.e.,  the image of $\mathbf{GL}_1(N) \to \mathbf{Str}(A,-)$.
\item  $\gamma''\circ \theta' = \chi'$.
 \end{enumerate}
\end{lemma}

\begin{proof}
	Towards (i), suppose $\alpha \in \ker(\gamma'_R) = \gamma_R^{-1}(R^\times \id)$. This means there exists $r \in R^\times$ with $\alpha_S = r \id$. According to (\ref{eq.hat-alpha}) and (\ref{line-one}), we have 
	$
		\alpha L_s L_t \alpha^{-1} = L_{rs} L_{r^{-1}t} =  L_s L_t
	$
	for all $s, t \in S_R$. Now $\{L_sL_t \mid s, t \in S\}$ generates $\End_{N_R}(A_R)$ as an $R$-algebra because $\theta(C^+(S,Q)) = \End_N(A)$, as we showed in Lemma \ref{lem:theta-map}~(i). So $\alpha$ centralises  $\End_{N_R}(A_R)$ in $\End_{R}(A_R)$. The double centraliser theorem \cite[Theorem 4.10]{jacobson2012basicII} implies $\alpha \in N_R$. So $\ker(\gamma'_R)$ is contained in the image of $\mathbf{GL}_1(N) \to \mathbf{Str}(A,-)$. It is easy to check that this containment is an equality.
		
	For (ii), let $x \in \Omega(S_R, Q_R)$ and let $\alpha = \theta'_R(x)$. By definition, $\chi'_R(x)  = \bar{\beta} \in \PGO^+(S_R, Q_R)$ for some $\beta \in \GO^+(S_R,Q_R)$ such that $\Int(x)(st) = C_R(\beta)(st) =  \mu(\beta)^{-1} \beta(s)\beta(t)$ for all $s, t \in S_R$. Applying $\theta'_R$ to the preceding equation yields
	\[
	\alpha (-L_sL_{t^\natural})\alpha^{-1} = \theta'_R \circ C_R(\beta)(st) = -\mu(\beta)^{-1}L_{\beta(s)}L_{\beta(t)^\natural}.
	\]
 Meanwhile, $\gamma''_R(\alpha) = \overline {\beta'}$ for some $\beta' \in \GO^+(S_R,Q_R)$ such that  $\beta' = \gamma_R(\alpha) = \alpha_S$. This implies that
 \begin{align*}\alpha (-L_s L_{t^\natural}) \alpha^{-1} = -\mu(\beta')^{-1}L_{\beta'(s)}L_{\beta'(t)^\natural} = \theta'_R\circ C_R(\beta')(st) \text{for all } s,t \in S_R,
 \end{align*}
 the first equality being a consequence of (\ref{eq.hat-alpha}) and (\ref{line-one}). We have shown that $\theta_R \circ C_R(\beta) = \theta_R\circ C_R(\beta')$. We would like to show that $C_R(\beta) = C_R(\beta')$. If $m_1 = 2$, then $\theta_R$ is injective so this goal is achieved. If $m_1 = 4$ or~$8$, the main involution $\tau$ on $C^+(S,Q)$ is of the second kind \cite[(8.4)]{knus1998book}, so it swaps the two full matrix subalgebras of $C^+(S,Q)$ and equality in $C^+(S_R,Q_R)$ can be tested by applying $\theta_R$ and $\theta_R\circ \tau$ successively. Any automorphism in the image of $C_R$ commutes with~$\tau$, so $\theta_R \circ \tau \circ C_R(\beta) = \theta_R\circ C_R(\beta)\circ \tau = \theta \circ C_R(\beta') \circ \tau = \theta_R \circ \tau \circ C_R(\beta')$. This proves $C_R(\beta) = C_R(\beta')$. Given that $\ker(C_R) = R^\times \id$, this proves that $\overline \beta = \overline{\beta'} \in \PGO^+(S_R,Q_R)$. 
\end{proof}

\begin{proposition}\label{main1} Let $(A,-)$ be an $(8,8)$-product algebra. The following diagram commutes, the rows are exact, and the vertical arrows are surjective.
	\begin{equation} \label{diag.nice-diagram} 
		\begin{tikzcd}
			1 \arrow[r] &  \mathbf{G}_m \times \mathbf{G}_m \arrow[r] \arrow[d] & \mathbf{\Omega}(S,Q) \arrow[r,"\chi'"] \arrow[d,"\theta'"] & \mathbf{PGO}^+(S,Q) \arrow[r] \arrow[d, equal] & 1\\
			1 \arrow[r] & 1\times \mathbf{G}_m \arrow[r] & \mathbf{Str}(A, -)^\circ \arrow[r, "\gamma''"] & \mathbf{PGO}^+(S,Q) \ar[r] & 1.
		\end{tikzcd}
	\end{equation}
	The middle column is part of a short exact sequence: 
	\begin{equation} \label{diag.nice-short-exact}
	\begin{tikzcd}
		1 \ar[r] & \mathbf{G}_m \times 1 \ar[r] & \mathbf{\Omega}(S,Q) \ar[r, "\theta'"] &\mathbf{Str}(A,-)^\circ \ar[r] & 1.
	\end{tikzcd}
	\end{equation}
	\end{proposition}

\begin{proof}
	The top row of the diagram comes from (\ref{diag:omega-exact}), using the fact that $Z = Z(C^+(S,Q))$ is split. Lemma \ref{ker-gamma} proves that $\ker(\gamma'')$ is the group of scalar matrices
 and that the right square is commutative (so the left square is too, by design). It follows that $\gamma''$ is surjective onto $\mathbf{PGO}^+(S,Q)$, hence the bottom row is exact.
A diagram chase (the five lemma) applied to the $R$-points of (\ref{diag.nice-diagram}) proves that $\theta'_R$ is surjective onto $\Str(A_R,-)$ for all $R$. The kernel of $\theta'$ is one of the copies of $\mathbf{G}_m$ in the centre of $\mathbf{\Omega}(S,Q)$ by Lemma \ref{lem:theta-map}~(iii). Hence \eqref{diag.nice-short-exact} is exact.
\end{proof}

%We can set up diagrams like (\ref{diag.nice-diagram}) for the other $(m_1, m_2)$-product algebras too. Since the proofs are similar to Theorem \ref{main1}, we give sketches and highlight only the main differences.

\begin{proposition} \label{prop:4,8}
	Let $(A,-)$ be an $(8,4)$-product algebra and let $N = \Nuc(A)$ be its quaternion factor. The following diagram commutes and the rows are exact.
	\begin{equation} \label{diag.nice-diagram2} 
		\begin{tikzcd}
			1 \arrow[r] &  \mathbf{G}_m \times \mathbf{G}_m \arrow[r] \arrow[d,"{(c_1, c_2) \mapsto c_2}" left] & \mathbf{\Omega}(S,Q) \arrow[r,"\chi'"] \arrow[d,"\theta'"] & \mathbf{PGO}^+(S,Q) \arrow[r] \arrow[d, equal] & 1\\
			1 \arrow[r] & \mathbf{GL}_1(N) \arrow[r,"R"] & \mathbf{Str}(A,-)^\circ \arrow[r, "\gamma''"] & \mathbf{PGO}^+(S,Q) \ar[r] & 1.
		\end{tikzcd}
	\end{equation}
\end{proposition}

\begin{proof}
	The proof is again very straightforward using Lemma \ref{ker-gamma}.
\end{proof}

	The major difference between  and \eqref{diag.nice-diagram} and \eqref{diag.nice-diagram2} is that  in the latter diagram the left two vertical arrows are neither injective nor surjective.  Rather, it is clear from \ref{lem:theta-map}~(i) that $\theta'(\mathbf{\Omega}(S,Q)) \subset \mathbf{Str}_N(A,-)^\circ$ and by diagram chasing it is easy to deduce that $\theta'(\mathbf{\Omega}(S,Q)) = \mathbf{Str}_N(A,-)^\circ$, which by Lemma \ref{lem:theta-map}~(iii) is isomorphic to $\mathbf{\Omega}(S,Q)/\mathbf{G}_m$. We can apply another diagram chase to (\ref{diag.nice-diagram2}) and find that $\mathbf{Str}(A,-)^\circ$ is generated by the two commuting subgroups $\mathbf{GL}_1(N)$ and $\mathbf{Str}_N(A,-)^\circ \simeq \mathbf{\Omega}(S,Q)/\mathbf{G}_m$. These two subgroups intersect in a central torus $\mathbf{G}_m$, so \[\mathbf{Str}(A,-)^\circ \simeq \frac{ \mathbf{Str}_N(A,-)^\circ \times \mathbf{GL}_1(N)}{\mathbf{G}_m} \simeq \frac{ \mathbf{\Omega}(S,Q) \times \mathbf{GL}_1(N)}{\mathbf{G}_m \times \mathbf{G}_m}.\]

\begin{proposition} \label{prop:2,8}
	Let $(A,-)$ be an $(8,2)$-product algebra and let $Z(A) = F$. The following diagram commutes, the rows are exact, and the vertical arrows are isomorphisms:
	\begin{equation*} 
		\begin{tikzcd}
			1 \ar[r] & \mathbf{GL}_1(Z) \ar[r]\ar[d, "\simeq"] & \mathbf{\Omega}(S,Q)\ar[d, "\theta'"] \ar[r, "\chi'"] & \mathbf{PGO}^+(S,Q) \ar[r] \ar[d,equal] & 1\\
			1 \ar[r] & \mathbf{GL}_1(F) \ar[r] & \mathbf{Str}(A,-)^\circ \ar[r,"\gamma''"] & \mathbf{PGO}^+(S,Q) \ar[r] & 1
		\end{tikzcd}
	\end{equation*}
\end{proposition}
\begin{proof}
	The exactness of the bottom row again follows from Lemma \ref{ker-gamma}. Clearly the leftmost vertical arrow is an isomorphism (using \cite[Theorem~8.2]{knus1998book} or just the fact from Lemma \ref{lem:theta-map} that $\theta'$ is injective). By the five lemma, $\theta'$ is an isomorphism.
\end{proof}

Recall that in \ref{sec:connected str}, we have determined  the group of components $\pi_0(H)$ of the structure group $H = \mathbf{Str}(A,-)$. In the following theorem (which is independent of  \S \ref{sec:TKK Lie algebras}), we summarise our results on the connected component $H^\circ$.

\begin{theorem} \label{main2}
Let $(A,-)$ be an $(m_1, m_2)$-product algebra with an Albert form $Q$ on $S = \Skew(A,-)$, and let $N = \Nuc(A)$ and $F = Z(A)$.  The connected structure group $H^\circ = \mathbf{Str}(A,-)^\circ$ is the reductive group determined up to isomorphism by the data in Table \ref{table.structure-groups}.
\end{theorem}

\begin{table}[hbt]
\begin{center} {\scriptsize 
\begin{tabular}{x{0.45cm}x{0.45cm}|x{3.3cm}x{3.3cm}x{1.4cm}x{1.6cm} x{1.8cm}} 
&& (\textbf{A}) & (\textbf{B}) & (\textbf{C}) & (\textbf{D}) & (\textbf{E})\\ & \\
  $m_1$ & $m_2$ & $H^\circ$ & $(H^\circ)^{\rm der}$ & $Z(H^\circ)$ & $Z((H^\circ)^{\rm der})$ & Type \\ \hline &\\
 $8$ & $8$ & $\mathbf{\Omega}(S,Q)/\mathbf{G}_m$ & $\mathbf{Spin}(S,Q)$ & $\mathbf{G}_m$ & $\bm{\mu}_4$ & $D_7$ \\ & \\
 & $4$ & $\displaystyle \frac{ \mathbf{\Omega}(S,Q) \times \mathbf{GL}_1(N)}{\mathbf{G}_m\times \mathbf{G}_m}$ & $\displaystyle \frac{\mathbf{Spin}(S,Q) \times \mathbf{SL}_1(N)}{\bm{\mu}_2}$ & $\mathbf{G}_m$ & $\bm{\mu}_{4}$ & $D_5 \times A_1$ \\ & \\
 & $2$ & $\mathbf{\Omega}(S,Q)$ & $\mathbf{Spin}(S,Q)$ & $\mathbf{GL}_1(F)$ & $\bm{\mu}_2 \times \bm{\mu}_2$ & $D_4$ \\ & \\
 & $1$ & $\mathbf{\Gamma}^+(S,Q)$ & $\mathbf{Spin}(S,Q)$ & $\mathbf{G}_m$ & $\bm{\mu}_2$ & $B_3$ \\ & \\
  $4$ & $4$ & $\displaystyle \frac{\mathbf{GL}_1(A)\times \mathbf{GO}(A,-)}{\mathbf{G}_m}$ & $\displaystyle \frac{\mathbf{SL}_1(A) \times \mathbf{SO}(A,-)}{\bm{\mu}_2}$ & $\mathbf{G}_m$  & $\bm{\mu}_4$ & $A_3 \times A_1 \times A_1$ \\ & \\
  & $2$ & $\displaystyle \frac{\mathbf{GL}_1(A)\times \mathbf{GU}(A,-)}{\mathbf{GL}_1(F)}$ & $\displaystyle \frac{\mathbf{SL}_1(A) \times \mathbf{SU}(A,-)}{\bm{\mu}_{2}}$ & $\mathbf{GL}_1(F)$ & $\bm{\mu}_{2}$ & $A_1 \times A_1 \times A_1$ \\ & \\
&  $1$ & $\displaystyle \frac{\mathbf{GL}_1(A) \times \mathbf{GL}_1(A)}{\mathbf{G}_m}$ & $\displaystyle \frac{\mathbf{SL}_1(A) \times \mathbf{SL}_1(A)}{\bm{\mu}_2}$ & $\mathbf{G}_m$ &  $\bm{\mu}_2$ & $A_1 \times A_1$ \\ & \\
$2$ & $1$ & $\mathbf{GL}_1(F)$ & $1$ & $\mathbf{GL}_1(F)$ & $1$ \\ & \\
$1$ & $1$ & $\mathbf{G}_m$ & $1$ & $\mathbf{G}_m$ & $1$ 
 \\ \ \\
\end{tabular} }
\caption{If $(A,-)$ is an $(m_1,m_2)$-product algebra, the table displays:  (\textbf{A})~The group $H^\circ = \mathbf{Str}(A,-)^\circ$. (\textbf{B})--(\textbf{D})~The isomorphism classes of $(H^\circ)^{\rm der}$, $Z(H^\circ)$, and $Z((H^\circ)^{\rm der})$.
%, which determine $H^\circ$ in the ``standard form" $H^\circ \simeq \frac{(H^\circ)^{\rm der}\times Z(H^\circ)}{Z((H^\circ)^{\rm der})}$.
(\textbf{E}) The type of the root system of $(H^\circ)^{\rm der}$.} \label{table.structure-groups}
\end{center}
\end{table}

\begin{proof}
	Column (A) comes directly from \ref{sec:associative}, \ref{sec:octonion}, \ref{main1}, \ref{prop:4,8}, and \ref{prop:2,8}. So we begin with Column (B). Since $\mathbf{Spin}(V,q)^{\rm der} = \mathbf{Spin}(V,q)$, it is not difficult to show using the definitions that $\mathbf{\Omega}(V,q)^{\rm der} = \mathbf{\Gamma}^+(V,q)^{\rm der} = \mathbf{Spin}(V,q)$ for any quadratic space $(V,q)$. Similar facts like $\mathbf{GL}_1(A)^{\rm der} = \mathbf{SL}_1(A)$,  $\mathbf{GO}(A,-)^{\rm der} = \mathbf{SO}(A,-)$, and $\mathbf{GU}(A,-)^{\rm der} = \mathbf{SU}(A,-)$ are standard and easy to prove. To complete Column~(B) we also used the fact that $[B/A,B/A] \simeq [B,B]/(A \cap [B,B])$ for abstract groups $A \lhd B$, in combination with the characterisation of $G^{\rm der}$ from \cite[Proposition 6.18]{milne}. The remaining columns are based on direct and easily replicable calculations, using facts from \cite[\S\S23--26]{knus1998book} or \cite[\S24]{milne} about the root systems and centres of semisimple groups.
	\end{proof}

\section{Norm-similitude groups}

Algebraists have been studying the norm-similitude groups of algebras for a very long time. The following Theorem \ref{Frobenius-theorem} is as classical as it gets, proved in an early form by Frobenius \cite{frobenius1897uber} and generalised by Dieudonn{\'e} \cite{dieudonne1948generalisation}, Jacobson \cite{jacobson1976structure},  Waterhouse \cite{waterhouse1987automorphisms}, and many others. 

%\subsection{The norm-similitude group for involutions of the first kind}

Let $(A,\sigma)$ be a central simple associative algebra with involution. Consider the group $(\mathbf{GL}_1(A) \times \mathbf{GL}_1(A)) \rtimes \ZZ/2\ZZ$, where $\ZZ/2\ZZ$ acts by $(x,y) \mapsto (y,x)$, together with the homomorphism 
\begin{align*}
	\phi_{\sigma}: (\mathbf{GL}_1(A) \times \mathbf{GL}_1(A))\rtimes \ZZ/2\ZZ \to \mathbf{GO}(N_A), &&
	\phi_{\sigma}(x,y,\epsilon) & =  L_x R_{\sigma(y)}\sigma^\epsilon.
\end{align*}

\begin{theorem} \label{Frobenius-theorem}	If $(A,\sigma)$ is a central simple associative algebra with involution of the first kind over an arbitrary field $K$, the following sequence is exact:
	\begin{equation*}
	\begin{tikzcd}
	1 \ar[r] & \mathbf{G}_m \ar[r] & (\mathbf{GL}_1(A) \times \mathbf{GL}_1(A))\rtimes \ZZ/2\ZZ \ar[r,"\phi_\sigma"] & \mathbf{GO}(N_A) \ar[r] & 1.
%	&& \mathbf{G}_m \ar[r,equal] & \mathbf{G}_m 
	\end{tikzcd}
	\end{equation*}
\end{theorem}

\begin{proof}
	It is clear that the kernel of $\phi_\sigma$ is the one-parameter subgroup $T \simeq \mathbf{G}_m$ where $T(R) = \{(c\id,c^{-1}\id,0) \mid c \in R^\times \}$. By definition,  $\mathbf{Iso}(N_A)$ is the kernel of the multiplier homomorphism $\mathbf{GO}(N_A) \to \mathbf{G}_m$. Assume first that $A \simeq M_n(K)$ and $\sigma = t$ is the transpose involution. Then $N_A$ is the determinant and \cite[Theorem 4.1]{waterhouse1987automorphisms} shows that $\mathbf{Iso}(N_A)$ is smooth (so $\mathbf{GO}(N_A)$ is smooth too). The same theorem implies quite directly that $\mathbf{GO}(N_A)$ is generated by the image of $\phi_t$. If $A= M_n(K)$ and $\sigma$ is another orthogonal involution or a  symplectic involution, then $\sigma = \Int(u)t$ for some invertible $u \in A$ \cite[Proposition~2.19]{knus1998book}, and it is clear that $\phi_\sigma$ and $\phi_t$ have the same image, namely $\mathbf{GO}(N_A)$. This still holds if $A$ is not split, because one can extend to the algebraic closure (or just until $A$ is split) and use the smoothness of $\mathbf{GO}(N_A)$.
\end{proof}

\subsection{The norm-similitude group for involutions of the second kind} \label{sec:ns-2ndkind}

Suppose $K$ is an arbitrary field and $(A,\sigma)$ is a central simple associative algebra over $K$ with involution of the second kind, and let $F = Z(A)$. It is clear that we have an exact sequence
\begin{equation*}
\begin{tikzcd} 1 \ar[r] & \mathbf{GL}_1(F) \ar[r] & (\mathbf{GL}_1(A) \times \mathbf{GL}_1(A)) \rtimes \ZZ/2\ZZ \ar[r,"\phi_\sigma"] & \mathbf{GO}(N_A)	.
\end{tikzcd}
\end{equation*}
In the rather trivial case where $A = F$, it is easy to work out that $\phi_\sigma$ is surjective. However, if $A \ne F$ then $\phi_\sigma$ is not surjective. Indeed, if $A = M_n(F)$, $n > 1$, then there is a conjugate-transpose involution $* = t\circ \iota= \iota \circ t$ on $A$, where $t \in \End(M_n(F))$ is the transpose map and $\iota$ is the nontrivial automorphism  of $F/K$ applied entrywise. There is a self-adjoint unit $u = u^* \in \mathbf{GL}_n(F)$ such that $\sigma = \Int(u)\circ *$ \cite[Proposition~2.20~(2)]{knus1998book}. By \cite[Th{\'e}or{\`e}me~3]{dieudonne1948generalisation}, \[\mathbf{GO}(N_A)(K) = \frac{(\GL_1(A) \times \GL_1(A))\rtimes \{1, t, \iota, *\}}{F^\times}\]  and this group has no fewer than four Zariski-connected components, unlike the image of $\phi_\sigma$ which has only two connected components.

\subsection{The norm-similitude group for octonion algebras} \label{sec:ns-octonion}
	If $C$ is an octonion algebra over $k$ with norm $n$, the similitude group $\mathbf{GO}(C,n) = \mathbf{G}_m.\mathbf{O}(C,n)$ is the reductive group generated by the group of scalar transformations and the orthogonal group of $(C,n)$. The $k$-points of $\mathbf{GO}(C,n)$ are all of the form $L_z t$ where $z$ is an invertible element of $C$ and $t \in O(C,n)$ is a product of some reflections \cite[p.\! 38]{springer2000exceptional}. If $y \in C$ is invertible then the reflection about $y$ is $\rho_y = -N(y)^{-1}L_y R_y \tau = -N(y)^{-1}L_y \tau L_{\bar y}$ where $\tau$ is the standard involution \cite[p.\! 44]{springer2000exceptional}, so the group of norm-similitudes is generated as an abstract group by left-multiplications and the standard involution.

\begin{theorem} \label{thm:Str-GO}
 \label{thm:ns-8,8 and 8,2}
	If $(A,-)$ is an $(8,8)$- or $(8,2)$-product algebra, then $\mathbf{Str}(A,-) = \mathbf{GO}(N_A)$.
\end{theorem}

\begin{proof}
We have $\mathbf{Str}(A,-) \subset \mathbf{GO}(N_A) \subset \mathbf{GL}(A)$. Assume first that $(m_1, m_2) = (8,8)$.  In light of Lemma \ref{lem:connected str} and Proposition~\ref{main1},  $\mathbf{Str}(A,-)^\circ = \mathbf{Str}(A,-) = X.Z(\mathbf{GL}(A))$ where $X = \mathbf{Str}(A,-)^{\rm der}$ is the image of the (faithful) half-spin representation $\theta'|_{\mathbf{Spin}(S,Q)}: \mathbf{Spin}(S,Q) \to \mathbf{GL}(A)$. Clearly $X \subset \mathbf{Iso}(N_A) \subset \mathbf{SL}(A)$. This half-spin representation is irreducible, restricted (in the sense of \cite[p.\! 3]{garibaldi2015simple}), and tensor indecomposable. Since $\Char(k) \ne 2$ and $N_A$ is an irreducible polynomial of degree 8 \cite[Theorem~9.6]{allison1992norms}, we can apply \cite[Lemma 5.1]{garibaldi2015simple} and conclude that $(\mathbf{Iso}(N_A)^\circ)_{k^a}$ is smooth and equals $X_{k^a}$. Therefore $\mathbf{Iso}(N_A)^\circ$ is smooth, and $\mathbf{Iso}(N_A)^\circ = X$. 

We claim that $\mathbf{Str}(A,-)$ contains the normaliser of $X$ in $\mathbf{GL}(A)$; since $\mathbf{Iso}(N_A)^\circ$ is normal in $\mathbf{GO}(N_A)$, this clearly implies that $\mathbf{GO}(N_A) \subset \mathbf{Str}(A,-)$. If $g \in N_{\mathbf{GL}(A)}(X)(k^a)$, then $\Int(g)|_X$ is an automorphism of $X$. There is only one nontrivial outer automorphism class of $X$ and it acts nontrivially on $Z(X) \subset Z(\mathbf{GL}(A))$, so there is no element of $\mathbf{GL}(A)(k^a)$ whose conjugation action is an outer automorphism of~$X$. Therefore, $\Int(g)|_X$ is an inner automorphism of $X$ and this implies $g \in X.C_{\mathbf{GL}(A)}(X)(k^a)$. But $C_{\mathbf{GL}(A)}(X) = Z(\mathbf{GL}(A))$ because by Lemma \ref{lem:theta-map}~(iv), anything in $\mathbf{GL}(A)(k^a)$ that commutes with $X(k^a)$ must commute with all of $\mathbf{GL}(A)(k^a)$. Therefore $g \in X.Z(\mathbf{GL}(A))(k^a) = \mathbf{Str}(A,-)(k^a)$.

	If $(m_1, m_2) = (8,2)$, then by Proposition \ref{prop:2,8}, $\mathbf{Str}(A,-)^\circ = X.Z$ where the subgroup $X = \mathbf{Str}(A,-)^{
	\rm der} \simeq \mathbf{Spin}(S,Q)$ is the image of the full spin representation and $Z = Z(\mathbf{Str}(A,-)) \simeq \mathbf{GL}_1(F)$. We can use \cite[Lemma 5.1]{garibaldi2015simple} again to conclude that $(\mathbf{Iso}(N_A)^\circ)_{k^a}$ is smooth and equal to $X_{k^a}$, so $\mathbf{Iso}(N_A)^\circ = X$. The centraliser of $X$ in $\mathbf{GL}(A)$ is $Z$ by Lemma \ref{lem:theta-map}~(iv), so the connected normaliser  $N_{\mathbf{GL}(A)}(X)^\circ$ is equal to $X.Z = \mathbf{Str}(A,-)^\circ$. By Lemma~\ref{lem:connected str}, $\pi_0(\mathbf{Str}(A,-))= \ZZ/2\ZZ = \pi_0(\mathbf{Aut}(X)) = \pi_0({N_{\mathbf{GL}(A)}}(X))$, and this implies $\mathbf{GO}(N_A) \subset N_{\mathbf{GL}(A)}(X) = \mathbf{Str}(A,-)$.
\end{proof}

%\begin{lemma} \label{lem:generated by centralisers}
%If a connected reductive group $G$ contains two commuting normal subgroups $X, Y$  that are connected, simple, and of positive dimension, then $G$ is generated by the (scheme-theoretic) centralisers of $X$ and $Y$.\end{lemma}
%
%
%\begin{proof}
%This is a straightforward consequence of the structure theory of reductive groups and \cite[Proposition 14.10~(2)]{borel}.	
%\end{proof}

By Lemma \ref{lem:connected str}, Table \ref{table.structure-groups}, and the results of \ref{Frobenius-theorem}--\ref{thm:ns-8,8 and 8,2} we arrive at the following conclusion.

\begin{theorem}
Let $(A,-)$ be an $(m_1, m_2)$-product algebra. If $(m_1, m_2) = (8,8)$, $(8,2)$, $(4,1)$, $(2,1)$, or $(1,1)$ then $\mathbf{Str}(A,-) = \mathbf{GO}(N_A)$. If $(m_1, m_2) = (8,1)$, $(4,4)$, or $(4,2)$ then $\mathbf{Str}(A,-)$ is a proper subgroup of positive codimension in $\mathbf{GO}(N_A)$.
\end{theorem}

Our attempts to answer this question for $(8,4)$-product algebras are still inconclusive.

\section{Albert forms and criteria for isotopy and division algebras} \label{ch:Gal-com-KP-and-SA}

The goal of this section is to establish some theorems about isotopy of $(m_1, m_2)$-product algebras, and some criteria for being a division algebra. For instance, we prove that isotopic algebras have similar Albert forms, and that the Albert forms classify $(8,m)$-product algebras up to isotopy -- this was previously known only in characteristic~0. Galois cohomology affords us some elegant proofs of these results, which are also new proofs of the known results in characteristic~0.

\subsection{Nonassociative pairs}\label{sec:pairs1}
To interpret the Galois cohomology set $H^1(k, \mathbf{Str}(A, -))$, we follow standard practice and look for an algebraic object whose automorphism group is $\mathbf{Str}(A, -)$.

\begin{definition*}A \emph{(nonassociative) pair} over a $k$-algebra $R$ is a pair of $R$-modules $P = (P_+, P_-)$ equipped with a pair of $R$-bilinear maps $\mathcal{V}: P_\sigma \times P_{-\sigma} \to \End(P_{\sigma})$ for $\sigma = \pm 1$. We deliberately use the same notation $\mathcal{V}: (x,y)\mapsto \mathcal{V}_{x,y}$ for both these maps. 
\end{definition*}

 An isomorphism of pairs $P\to Q$ is a pair of $R$-module isomorphisms $f_\sigma: P_\sigma \to Q_\sigma$ such that $f_\sigma \mathcal{V}_{x,y} = \mathcal{V}_{f_\sigma(x), f_{-\sigma}(y)}f_\sigma$ for all $x, y \in P_\sigma \times P_{-\sigma}$. For an $R$-module $S$, we define the scalar extension of $P$ in the obvious way and denote it by $P_S = (P_{+,S},P_{-,S})$.

\subsection{The automorphism group scheme of a pair}

Let $P$ be a nonassociative pair over $k$, and let 
$W_\sigma = \End(P_\sigma \otimes P_{-\sigma} \otimes P_\sigma, P_{\sigma})$.
Consider the following pair of representations:
\begin{align*}
 \rho_\sigma: \mathbf{GL}(P_+)\times\mathbf{GL}(P_{-}) &\to \mathbf{GL}(W_\sigma)
  \\ \rho_{\sigma,R}(f_+,f_{-})(E)(x \otimes y \otimes z) &= f_{\sigma}(E(f_\sigma^{-1}(x)\otimes f_{-\sigma}^{-1}(y) \otimes f_{\sigma}^{-1}(z)))
 \end{align*}
 for all $f_\sigma \in \GL(P_{\sigma,R})$, $E \in W_{\sigma,R}$, and $(x,y,z) \in P_{\sigma,R}\times P_{-\sigma,R}\times P_{\sigma,R}$.
  The direct sum $\rho = \rho_+ \oplus \rho_-$ is a representation of $\mathbf{GL}(P_+)\times\mathbf{GL}(P_-)$ in the vector space $W = W_+\oplus W_-$.
  
  Note that $W$ can be identified with the set of all pairs with underlying vector spaces $P_+$ and $P_-$. Let $w \in W$ be the representative of $\mathcal{V}$; that is, $w = (w_+,w_-)$ where $w_\sigma(x,y,z) = \mathcal{V}_{x,y}(z)$. Define the group functor  $\mathbf{S}_w\subset \mathbf{GL}(W)$ as the stabiliser of $w$ in $\mathbf{GL}(W)$; that is,
  \[
  	\mathbf{S}_w(R) = \{\alpha \in \GL(W_R) \mid \alpha(w) = w \text{ for all } w \in W_R\}.
  \]
  Finally, define the group functor $\mathbf{Aut}(P) = \rho^{-1}(\mathbf{S}_w)\subset \mathbf{GL}(P_+)\times \mathbf{GL}(P_-)$, meaning:
	\[
	\mathbf{Aut}(P)(R) = \{(f_+, f_-) \in \GL(P_{+,R})\times \GL(P_{-,R}) \mid \rho_R(f_+,f_-) \in \mathbf{S}_w(R)\}.
	\]
  The functor $\mathbf{Aut}(P)$ is representable \cite[Examples~20.3~(2) \& 20.4~(2)]{knus1998book}, and we call it the \emph{automorphism group scheme of $P$}.  The group of $R$-points $\mathbf{Aut}(P)(R)$ is precisely the automorphism group of the pair $P_R$.
  
 \subsection{Galois cohomology of Kantor pairs} \label{sec:cohomology of pairs} Since $H^1(k, \mathbf{GL}(W)) = 1$ by Hilbert 90, \cite[Proposition~29.1]{knus1998book} implies that there is a one-to-one correspondence
  \begin{center}
\fbox{\begin{minipage}[c][1cm]{4cm} \centering
$H^1\big(k,\mathbf{Aut}(P)\big)$
\end{minipage}}\quad  $\longleftrightarrow$\quad\fbox{\begin{minipage}[c][1cm]{9cm} \centering
$k$-isomorphism classes of pairs $Q$ such that the $k^s$-pairs $Q_{k^s}$ and $P_{k^s}$ are isomorphic.
\end{minipage}}
\end{center}
 
\begin{definition*} \label{def:kantor-pairs}  \cite[\S3]{allison1999elementary} 
A pair $P = (P_+,P_-)$ is called a \emph{Kantor pair} if it satisfies the identities:
\begin{enumerate}[(KP1)]
	\item $[\mathcal{V}_{x,y},\mathcal{V}_{z,w}] = \mathcal{V}_{\mathcal{V}_{x,y}(z),w} - \mathcal{V}_{z, \mathcal{V}_{y,x}(w)}$
	\item $\mathcal{K}_{a,b}\mathcal{V}_{x,y} + \mathcal{V}_{y,x}\mathcal{K}_{a,b} = \mathcal{K}_{\mathcal{K}_{a,b}(x),y}$
\end{enumerate}
for all $(x,y), (z,w) \in P_\sigma \times P_{-\sigma}$ and $(a,b) \in P_\sigma \times P_{\sigma}$, where $\mathcal{K}_{z,w}(x) = \mathcal{V}_{z,x}(w) - \mathcal{V}_{w,x}(z)$.
\end{definition*}

\subsection{Kantor pairs from structurable algebras} \label{sec:kantor-pairs}

If $(A,-)$ is a structurable algebra, then the pair $\KP(A,-) = (A,A)$ equipped with the maps $\mathcal{V}_{x,y} = 2V_{x,y}$ is a Kantor pair \cite[p.\! 533]{allison1999elementary}.
Let  $(A,-)$ and $(B,-)$ be structurable $R$-algebras. An $R$-module isomorphism $f:A \to B$ is an isotopy if and only if there exists a linear map $\hat f: A \to B$ such that $(f,\hat f)$ is an isomorphism of the Kantor pairs $\KP(A,-) \to \KP(B,-)$. The map $\hat f$, if it exists, is uniquely determined by $f$.  Consequently, $(A,-)$ is isotopic to $(B,-)$ if and only if $\KP(A,-)$ is isomorphic to $\KP(B,-)$, and there is an isomorphism of algebraic groups $\mathbf{Str}(A,-) \overset{\sim}{\to} \mathbf{Aut}(\KP(A,-))$ defined by $\alpha \mapsto (\alpha, \hat \alpha)$	for all $\alpha \in \Str(A_R,-)$.

Of course, not every Kantor pair is of the form $\KP(A,-)$ for a structurable algebra $(A,-)$, so the following lemma is very important; it implies among other things that the class of Kantor pairs associated to structurable algebras is closed under Galois descent. 
\begin{lemma} \label{lem:dense-orbit}
	Let $(A,-)$ be a central simple structurable algebra over $k$.
	\begin{enumerate}[\rm (i)]
	\item 	If $k$ is algebraically closed, the action of $\mathbf{Str}(A,-)^\circ$ on $A$ has a dense open orbit.
	\item The map $i_*: H^1(k,\mathbf{Aut}(A,-)) \to H^1(k,\mathbf{Str}(A,-))$ induced by the inclusion $i: \mathbf{Aut}(A,-) \to \mathbf{Str}(A,-)$ is surjective.
	\end{enumerate}
\end{lemma}

\begin{proof}
	Part (i) is well-known fact of unclear provenance: see \cite{azad1990structure,garibaldi2009cohomological,stavrova2018classification} and the references therein. We shall prove it first by appealing to the most convenient reference, and then later sketch an alternative proof using the classification of structurable algebras. Let $L = K(A,-) = \bigoplus_{i = -2}^2 K_i$, let  $G = \mathbf{Aut}(L)^\circ$, let $\nu: \mathbf{G}_m \to G$ be the grading cocharacter, and let $\lambda = 2\nu$. Then by \cite[Theorem~5.7]{stavrova2018classification}, the pair $(G,  \lambda)$ satisfies the conditions of \cite[Theorem~5.5~(1)]{stavrova2018classification} and it follows that $C_G(\lambda) = C_G(\nu)$ has a unique dense orbit in $K_1$ (since the 2-weight space of $\lambda$ is the same as the 1-weight space of~$\nu$). The diagram $C_G(\nu) \to \mathbf{GL}(K_1)$ is isomorphic to the diagram $\mathbf{Str}(A,-)^\circ \to \mathbf{GL}(A)$, so $\mathbf{Str}(A,-)^\circ$ has a dense orbit in $A$.

	To prove (ii), firstly assume  $k$ is infinite. Then Rost's Theorem on prehomogeneous vector spaces \cite[Theorem 9.3 \& Context, p.\ 29]{garibaldi2009cohomological} implies that $i_*$ is surjective. Secondly, if $k$ is a finite field, then $i_*$ is surjective provided that $\pi_0(i): \pi_0(\mathbf{Aut}(A,-)) \to \pi_0(\mathbf{Str}(A,-))$ is surjective \cite[III.~\S2.4 Corollaries 2 \& 3]{serre1997galois}, and $\pi_0(i)$ is indeed surjective by Lemma~\ref{lem:connected str}.	\end{proof}

We quickly sketch an alternative proof of Lemma \ref{lem:dense-orbit}~(i). One can use the classification of central simple structurable algebras to show that two central simple structurable algebras over an algebraically (even separably) closed field are isotopic if and only if they are isomorphic. Then for all $u \in A^*$, $(A,-) \simeq (A^{\langle u \rangle}, -^{\langle u \rangle})$, which means that $\mathbf{Str}(A,-)$ acts transitively on the invertible elements of $A$. It is known that $A^*$ is dense and open in $A$ \cite[Theorem~10.5]{allison1992norms}, so this gives an explicit description of the dense orbit.

\subsection{Galois cohomology of structurable algebras} \label{sec:galois cohomology of structurable algebras}
If $(A,-)$ is any algebra with involution, there is a one-to-one correspondence:
 \begin{center}
\fbox{\begin{minipage}[c][1.6cm]{4cm} \centering
$H^1\big(k,\mathbf{Aut}(A,-)\big)$
\end{minipage}}\quad  $\longleftrightarrow$\quad\fbox{\begin{minipage}[c][1.6cm]{9cm} \centering
$k$-isomorphism classes of algebras with involution $(A',-)$ such that  $(A'_{k^s},-)$ and $(A_{k^s},-)$ are isomorphic.
\end{minipage}}
\end{center}
Assume $(A,-)$ is a structurable algebra. By \ref{sec:cohomology of pairs} and \ref{sec:kantor-pairs}, we can identify $H^1(k, \mathbf{Str}(A,-))$ with the set of isomorphism classes of Kantor pairs that become isomorphic to $\KP(A,-)$ after extending scalars to $k^s$. The map $i_*: H^1\big(k,\mathbf{Aut}(A,-)\big) \to H^1\big(k,\mathbf{Str}(A,-)\big)$ sends the isomorphism class of $(A',-)$ to the isomorphism class of $\KP(A',-)$. Two $k$-forms of $(A,-)$ have the same image under $i_*$ if and only if they are isotopic over $k$. We can therefore identify the image of $i_*$ with the set of $k$-isotopy classes of structurable algebras that become isotopic to $(A,-)$ over $k^s$. If $(A,-)$ is central simple, then $i_*$ is surjective by Lemma~\ref{lem:dense-orbit}~(ii), hence there is a one-to-one correspondence:\begin{center} 
\fbox{\begin{minipage}[c][1cm]{4cm} \centering
$H^1\big(k,\mathbf{Str}(A,-)\big)$
\end{minipage}}\quad  $\longleftrightarrow$\quad\fbox{\begin{minipage}[c][1cm]{9cm} \centering
$k$-isotopy classes of structurable algebras $(A',-)$ such that $(A'_{k^s},-)$ and $(A_{k^s},-)$ are isotopic.
\end{minipage}}
\end{center}

\subsection{Galois cohomology of orthogonal similitude groups} \label{sec:gal-com-GO}

Let $(A,-) = (\End(V),\tau_q)$ be a matrix algebra with orthogonal involution adjoint to a quadratic form $q$ on a vector space~$V$ of even dimension $n$. Then $\mathbf{Aut}(A,-)=\mathbf{PGO}(V,q)$ and so $H^1(k, \mathbf{PGO}(V,q))$ is in one-to-one correspondence with the isomorphism classes of associative central simple algebras with orthogonal involution of degree $n$ \cite[29.F]{knus1998book}.  By Hilbert 90 and twisting, the natural map $H^1(k, \mathbf{O}(V,q)) \to H^1(k, \mathbf{GO}(V,q))$ is surjective and the natural map $H^1(k, \mathbf{GO}(V,q)) \to H^1(k, \mathbf{PGO}(V,q))$ is injective. As such,  $H^1(k, \mathbf{GO}(V,q))$ is in natural one-to-one correspondence with the similitude classes of $n$-dimensional quadratic spaces (see also \cite[III, Exercise~2 \& Lemma~VIII.21.21]{berhuy} for another perspective).

The identity component of $\mathbf{PGO}(V,q)$ is $\mathbf{PGO}^+(V,q)$, and $H^1(k, \mathbf{PGO}^+(V,q))$ is in one-to-one correspondence with isomorphism classes of triples $[(A,\sigma),\varphi]$ where $(A,\sigma)$ is an orthogonal involution of degree $n$ and $\varphi:Z(C(A,\sigma))\to Z(C^+(V,q))$ is an isomorphism \cite[29.F]{knus1998book}, $C(A,\sigma)$ being the Clifford algebra of $(A,\sigma)$ as defined in \cite[8.B]{knus1998book}. In the event that $Z(C^+(V,q)) = k \times k$, making a choice of an isomorphism $\varphi:Z(C(A,\sigma))\to k\times k$ is the same as ordering the two simple subalgebras of $C(A,\sigma)$ and labelling them as $C_+$ and $C_-$ \cite[Remark~29.31]{knus1998book}.

While the map $H^1(k, \mathbf{PGO}^+(V,q))\to H^1(k,\mathbf{PGO}(V,q))$ is not generally injective, it does have trivial fibres over the subset corresponding to quadratic forms; that is, the composition $H^1(k, \mathbf{GO}^+(V,q)) \to H^1(k, \mathbf{PGO}^+(V,q)) \to H^1(k, \mathbf{PGO}(V,q))$ is injective. One can show this by an argument similar to the one on \cite[p.\!~407]{knus1991quadratic}.
 
 \begin{proposition} \label{prop:isotopic-similar}
	If $(A,-)$ and $(A',-)$ are isotopic $(m_1, m_2)$-product algebras, then their Albert forms are similar.	
\end{proposition}

\begin{proof}
	The  map $\gamma: \mathbf{Str}(A,-) \to \mathbf{GO}(S,Q)$ from Lemma \ref{gamma}, when restricted to $\mathbf{Aut}(A,-)$, becomes  $\gamma(f) = f|_{S_R}$ for all $f \in \mathbf{Aut}(A,-)(R)$. This induces a commutative triangle:
	\begin{equation*} \label{diag:aut-str}
	\begin{tikzcd}
		H^1(k, \mathbf{Aut}(A,-)) \ar[rd, "(\gamma|_{\mathbf{Aut(A,-)}})_*"] \ar[d,"i_*"]\\
		H^1(k,\mathbf{Str}(A,-)) \ar[r,"\gamma_*"] & H^1(k,\mathbf{GO}(S,Q))
	\end{tikzcd}
	\end{equation*}
It is clear that the arrow  $(\gamma|_{\mathbf{Aut}(A,-)})_*$ sends the isomorphism class of an $(m_1, m_2)$-product algebra $(A',-)$ to the similitude class of its Albert form $Q'$. The fact that this factors through $H^1(k,\mathbf{Str}(A,-))$ gives us the lemma. 
\end{proof}

\subsection{A word on octonion algebras with nonstandard involutions}
There exist alternative structurable algebras that are neither associative nor $(8,1)$- or $(8,2)$-product algebras. In fact, there is just one such class of examples: these are called octonion algebras with nonstandard involutions (see \cite[p.\ 376]{allison1986conjugate} or \cite[Proposition~2.5]{pumplun2003involutions}). They are characterised among central simple structurable algebras by $(\dim A, \dim \Skew(A,-)) = (8,3)$. In the Allison--Smirnov classification scheme of central simple structurable algebras, these algebras fall into the class of structurable algebras constructed from Hermitian forms over quaternion algebras. The (abstract) automorphism group of such an algebra is of the form $\Aut(A,-) = \PGL(Q) \ltimes \SL(Q)$ for some quaternion algebra $Q \subset A$  \cite[\S2.1]{springer2000exceptional}. It is not clear without further work what the structure group of such an algebra is, although heuristically (using the ideas from \S\ref{sec:TKK Lie algebras}) it ought to be a connected reductive group with root datum of type $A_1 \times C_2$.

The following theorem was proved by Allison in \cite[Corollary~7.6]{allison1986conjugate}, who showed the statement is \emph{also true} for octonion algebras with nonstandard involutions. We give an alternative proof of the theorem but exclude the nonstandard octonionic involutions for expedience and in order to stay within our scope.

\begin{theorem} \cite[Corollary~7.6]{allison1986conjugate} \label{thm:assoc-isotopy}
	Let $(A,-)$ and $(A',-)$ be central simple algebras with involution over $k$ such that $A$ is alternative and $(\dim A, \dim \Skew(A,-)) \ne (8,3)$. Then $(A,-)$ and $(A',-)$ are isotopic if and only if they are isomorphic.
\end{theorem}

\begin{proof}
	Under the assumptions, $(A,-)$ is either associative or it is an $(8, m)$-product algebra where $m \le 2$ \cite[Theorem 5.1]{allison1986conjugate}. The claim to be proven is clearly equivalent to the statement that $i_*: H^1(k,\mathbf{Aut}(A,-)) \to H^1(k,\mathbf{Str}(A,-))$ is injective (and therefore an isomorphism, according to Lemma \ref{lem:dense-orbit}~(ii)).

	If $(A,-)$ is associative, applying $H^1(k,*)$ to the split short exact sequence \eqref{diag:associative ses} makes it clear that $i_*$ is  injective.
	If $(A,-)$ is not associative, then it is either an octonion algebra over $k$ or an $(8,2)$-product algebra over $k$. If $(A,-)$ is an octonion algebra then so are its isotopes, and one can deduce from \eqref{diag:octonions} using Hilbert 90 and twisting that $(A,-)$ and $(A',-)$ are isotopic if and only if they have the same image in $H^1(k,\mathbf{O}^+(n'))$; that is, their pure norms are isometric. This of course implies $(A,-) \simeq (A',-)$.
	
	If $(A,-)$ is an $(8,2)$-product algebra then so are its isotopes. By Proposition \ref{prop:isotopic-similar}, the isotopes $(A,-)$ and $(A',-)$ have similar Albert forms, say $Q \simeq \langle c \rangle Q'$. In the Witt ring we have $Q = n - \llangle d \rrangle$ and $Q' = m - \llangle e \rrangle$ for some $d, e \in k^\times$ and $3$-Pfister forms $n, m \in W(k)$. Then $n-\langle c \rangle m = \llangle d \rrangle - \langle c \rangle \llangle e \rrangle = 0$ by the Arason--Pfister Hauptsatz, so $n = \langle c \rangle m$ and $\llangle d \rrangle = \langle c \rangle \llangle e \rrangle$. It follows that $n = m$ and $\llangle d \rrangle = \llangle e \rrangle$ in $W(k)$ \cite[Ch.\! X. Corollary~5.4]{lam}, so $(A,-)$ and $(A',-)$ have isomorphic octonion and quadratic \'etale factors, and consequently $(A,-) \simeq (A',-)$.
	 \end{proof}
	 
	To proceed further towards an isotopy criterion for $(8,m)$-product algebras, we summarise some information from \S\ref{sec:structure groups}, now taking into account the whole structure group and not just its connected component.

\begin{lemma} \label{lem:horizontal-summary}
	If $(A,-)$ is an $(8,m)$-product algebra, $m \ge 2$, and $N = \Nuc(A)$, then the following sequence is exact:
	\begin{equation*}  \begin{tikzcd}1 \arrow[r] & \mathbf{GL}_1(N) \arrow[r,"R"] & \mathbf{Str}(A,-) \arrow[r, "\gamma'"] & \mathbf{PGO}(S,Q).
	\end{tikzcd}
	\end{equation*}
	If $m = 2$, then $\gamma'$ is surjective; otherwise its image is $\mathbf{PGO}^+(S,Q)$.
	The map \[\gamma_*':H^1(k,\mathbf{Str}(A,-))\to H^1(k,\mathbf{PGO}(S,Q))\] is injective, and sends the isotopy class of $(A',-)$ to the isomorphism class of $(\End S',\tau_{Q'})$ where $S' = \Skew(A',-)$ and $\tau_{Q'}$ is the adjoint involution to an Albert form $Q'$ of $(A',-)$. \end{lemma}

\begin{proof}
	The exactness of the sequence is proved in Lemma \ref{ker-gamma}~(i). As per the discussion in~\ref{sec:connected str}, $\mathbf{Str}(A,-)$ is connected if $m = 8$ or $4$, and has two connected components if $m = 2$. Therefore, by Propositions \ref{main1}, \ref{prop:4,8}, and \ref{prop:2,8}, the image of $\gamma'$ is $\mathbf{PGO}^+(S,Q)$ if $m = 8$ or~$4$ and all of $\mathbf{PGO}(S,Q)$ if $m = 2$. If $m = 2$ then Hilbert 90 and a twisting argument can be used to show that $\gamma_*': H^1(k,\mathbf{Str}(A,-)) \to H^1(k,\mathbf{PGO}(S,Q))$ is injective. Similarly, if $m= 8$ or $4$ then $\gamma_*'': H^1(k,\mathbf{Str}(A,-)) \to H^1(k,\mathbf{PGO}^+(S,Q))$ is injective. As mentioned in~\ref{sec:gal-com-GO}, the image of $\gamma_*''$ maps injectively into $H^1(k,\mathbf{PGO}(S,Q))$, hence $\gamma_*'$ is injective too. The interpretation of $\gamma_*'$ is clear from the definition of $\gamma'$.
\end{proof}

\begin{corollary} \label{cor:isotopic-similar}
	Let $(A,-)$ and $(A',-)$ be $(8,m)$-product algebras over $k$, where $m = 1,2,4$, or $8$. Then $(A,-)$ and $(A',-)$ are isotopic if and only if they have similar Albert forms.
\end{corollary}

\begin{proof}
	This is implied by Theorem \ref{thm:assoc-isotopy} for $m = 1$ and by Lemma~\ref{lem:horizontal-summary} for  $m \ge 2$. \end{proof}

The following is an elementary property of Pfister forms. Its proof is an easy exercise (following, say,  \cite[Theorem~VII.3.1 \& Theorem~X.1.8]{lam}).

\begin{lemma} \label{lem:pfister-splitting}
	Let $q = \langle 1 \rangle \perp q'$ be an $m$-Pfister form ($m \ge 1$) over $k$, and let $F= k(\sqrt{a})$ be a quadratic field extension. Then $q_F$ is isotropic if and only if $q'$ represents $-a$.
\end{lemma}

The final theorem of this section could have been proved much earlier, but it fits in well at this point, now that the isotopy criteria and the role of the Albert form are established.

\begin{theorem} \label{thm:albert-form-division-algebra}
Let $(A,-)$ be an $(m_1, m_2)$-product algebra. Then $(A,-)$ is a structurable division algebra if and only if   	its Albert form $Q$ is anisotropic and $Z(A)$ is a field.
\end{theorem}

\begin{proof}
	$(\Rightarrow)$ If the Albert form $Q$ is isotropic,  there exists a nonzero element $s \in S = \Skew(A,-)$ such that $Q(s) \ne 0$, and by \eqref{Ls-identity} this implies that $s$ is not invertible. If $Z(A)$ is not a field, then clearly $(A,-)$ fails to be a structurable division algebra.
	
	$(\Leftarrow)$ If $A = Z(A)$ is a field, then $A$ is obviously a structurable division algebra; see Lemma~\ref{lem:nuclear-isotopes}. If $m_2 = 1$ and $m_1 \ge 4$, then $A$ is a composition algebra and $Q$ is the pure norm of $A$. If $Q$ is anisotropic then the standard norm $\langle 1 \rangle \perp Q$ is anisotropic too because $\dim Q > \frac{1}{2}(\dim Q + 1)$, so $A$ is an alternative division algebra, and therefore also a structurable division algebra \cite[Corollary~3.6]{allison1986conjugate}. Suppose $m_2 = 2$, $m_1 \ge 4$, $Z(A) = k(\sqrt{a})$ is a field, and $Q$ is anisotropic. Let $C$ be the $m_1$-dimensional composition algebra in the (unique) decomposition $A = C\otimes Z(A)$, and let $n$ be the norm of $C$. Then $Q \simeq n' \perp \langle a \rangle$, so $n'$ does not represent $-a$, and $n_{Z(A)}$ is anisotropic by Lemma \ref{lem:pfister-splitting}. This implies that $A$ is an alternative division algebra, because $n_{Z(A)}$ is its generic norm, hence $(A,-)$ is a structurable division algebra \cite[Corollary~3.6]{allison1986conjugate}. If $(m_1, m_2) = (4,4)$, then the statement is essentially Albert's Theorem -- see \cite[Theorem~16.5 \& Corollary~16.28]{knus1998book}.
		
	This leaves two remaining cases. Suppose $(m_1, m_2) = (8,4)$, let $N = \Nuc(A)$ be the quaternion  factor of $A$, and assume $Q$ is anisotropic. Say the norm of $N$ is $n = \langle 1 \rangle \perp n'$; then $n'$ is anisotropic because it is a subform of $Q$, and therefore $n$ is anisotropic too. This implies $N$ is an associative division algebra, so $\mathbf{SL}_1(N)$ is anisotropic.  By Theorem~\ref{main2}, $\mathbf{Aut}(K(A,-))^\circ$ contains a semisimple subgroup isomorphic to $\mathbf{Spin}(S,Q)\times \mathbf{SL}_1(N)$, which is anisotropic of absolute rank 6. Therefore $\mathbf{Aut}(K(A,-))^\circ$, being of absolute rank 7, has $k$-rank equal to 1. Now the proof of \cite[Theorem~4.3.1]{boelaert2019moufang} implies $(A,-)$ is a structurable division algebra. If $(m_1,m_2) = (8,8)$ and $Q$ is anisotropic, then $\mathbf{Aut}(K(A,-))^\circ$ has a subgroup isomorphic to $\mathbf{Spin}(S,Q)$, which is anisotropic of absolute rank 7. Therefore $\mathbf{Aut}(K(A,-))^\circ$, being of absolute rank $8$, has $k$-rank equal to 1, which implies $(A,-)$ is a division algebra.
\end{proof}

The statement of Theorem \ref{thm:albert-form-division-algebra} for biquaternion algebras (Albert's Theorem) has been proved many times, and a discussion of its history can be found in the notes in \cite[p.\! 275]{knus1998book} and \cite[p.\! 71]{lam}. The statement for $(8,2)$- and $(8,4)$-product algebras are also provable in elementary ways; see \cite[Lemma~3.16]{boelaertthesis}. However, an elementary proof of Theorem \ref{thm:albert-form-division-algebra} for $(8,8)$-product algebras seems far out of reach because of the difficulty of working with the elements of such algebras.

The following is an alternative and somewhat useful reformulation of the above theorem for $(8,8)$-product algebras (for a similar theorem on biquaternion algebras in all characteristics, see \cite[Theorem~1.1]{becher2018transfer}).

\begin{corollary} \label{cor:cor}
	Let $E/k$ be a quadratic \'etale extension and let $C$ be an octonion  algebra over $E$. The following are equivalent:
	\begin{enumerate}[\rm (1)]
	\item $\cor_{E/k}(C)$ is not a structurable division algebra.
	\item The pure norm $n'$ of $C$ represents an element of $k \subset E$.
	\item $C$ contains a quadratic \'etale extension $K/k$ which is linearly disjoint from $E/k$ (in the sense that $EK$ is a $4$-dimensional $k$-vector space).
	\end{enumerate}
\end{corollary}

\begin{proof}
	An Albert form of $\cor_{E/k}(C)$ is $Q = T_{E/k}(\langle \delta \rangle n')$ where $n'$ is the pure norm of $C$ and $\delta \in E^\times$ is an element of trace zero. By definition $Q$ is isotropic if and only if there is a nonzero element $z \in C_0$ such that $\tr_{E/k}(\delta n(z)) = 0$, which is equivalent to $n(z) \in k$. Together with Theorem~\ref{thm:albert-form-division-algebra}, this observation yields (1) $\Leftrightarrow$ (2). If (2) holds and $z \in C_0$ has $n(z) \in k$, let $K = k(z)$. Then $\{z, \delta z, 1, \delta\}$ is a $k$-basis for $EK$ because $Ez \cap E = \{0\}$, implying (3). Conversely, suppose (3) holds. There is a generator $y$ of $K/k$ such that $y^2 \in k \subset E$. It is a basic property of composition algebras over fields of characteristic not 2 that an element whose square is a central scalar is either in the centre or has trace zero: so either $y \in E$ or $y\in C_0$. The former is impossible since $y$ generates $K$ and $EK \ne E$. Hence $y^2 = -n(y) \in k$, which gives us (2). \end{proof}

\begin{theorem} \label{thm:divisionness}
Let $(A,-)$ be an $(8,8)$- or $(8,4)$-product algebra over $k$. The following are equivalent:
\begin{enumerate}[\rm (1)]
\item \label{tfae-1}	$(A,-)$ is not a structurable division algebra.
\item \label{tfae-albert_form} $(A,-)$ has an isotropic Albert form.
\item \label{tfae-non-invertible_skew} $(A,-)$ has a non-invertible skew element.
\item \label{tfae-biquaternion_subalgebra_with_involution} $(A,-)$ has a non-division biquaternion subalgebra stabilised by the involution.
\item \label{tfae-subalgebra with involution} $(A,-)$ has a non-division associative subalgebra stabilised by the involution.
\end{enumerate}
\end{theorem}

\begin{proof}
$(\ref{tfae-1}) \Leftrightarrow (\ref{tfae-albert_form})$ is given by Theorem \ref{thm:albert-form-division-algebra}, and $(\ref{tfae-albert_form}) \Leftrightarrow (\ref{tfae-non-invertible_skew})$ is a direct consequence of \eqref{Ls-identity}.

To show that $(\ref{tfae-albert_form}) \Rightarrow (\ref{tfae-biquaternion_subalgebra_with_involution})$, assume first that  $(A,-)$ is an $(8,8)$-product algebra with an isotropic Albert form. Then $(A,-) = \cor_{E/k}(C)$ for some quadratic \'etale extension $E/k$ and an octonion algebra  $C$ over $E$. By Corollary \ref{cor:cor} there is a $z \in C_0$ such that  $n(z) \in k$. There is a quaternion subalgebra $Q \subset C$ containing $z \subset Q_0$  \cite[Proposition~1.6.4]{springer2000exceptional}. Then $\cor_{E/k}(Q)$ is a biquaternion subalgebra of $(A,-)$ stabilised by the involution whose Albert form is isotropic, and which is therefore not a division algebra.

Secondly, if  $(A,-)$ is an $(8,4)$-product algebra with an isotropic Albert form, then $(A,-) = C \otimes Q$ where $C$ is an octonion algebra and $Q$ is a quaternion algebra, such that either $C$ or $Q$ is split, or the pure norms $n_C'$ and $n_Q'$ represent a common element $a \in k^\times$. If $C$ or $Q$ is split, then it is easy to find a non-division biquaternion subalgebra stabilised by the involution. Otherwise, $n_C = \llangle -a, b, c \rrangle$ and $n_Q = \llangle -a, d \rrangle$ for some $b,c,d \in k^\times$ \cite[Pure Subform Theorem~1.5]{lam}. Let $Q'\subset C$ be a subspace on which $n_C$ restricts to $\llangle -a, b \rrangle$; then $Q'$ is a quaternion algebra  \cite[Propositions~1.2.3~\&~1.5.1]{springer2000exceptional} and $Q'\otimes Q$ is a biquaternion subalgebra stabilised by the involution. Since $Q' \otimes Q$ has an isotropic Albert form $\llangle -a,b \rrangle' - \llangle -a, d \rrangle'$, it is not a division algebra. Therefore we have shown that $(\ref{tfae-albert_form})$ implies~$(\ref{tfae-biquaternion_subalgebra_with_involution})$.

Clearly $(\ref{tfae-biquaternion_subalgebra_with_involution})\Rightarrow (\ref{tfae-subalgebra with involution})$. If $(A,-)$ has an associative subalgebra $(B,-)$ with involution that is not a division algebra, then $(B,-)$ is not a structurable division algebra by Lemma~\ref{lem:nuclear-isotopes}, so neither is $(A,-)$. This settles $(\ref{tfae-subalgebra with involution}) \Rightarrow (\ref{tfae-1})$.
\end{proof}

It would be interesting to know if there exists a characterisation of structurable division algebras of this type, where the characterisation does not make reference to the involution. For example, if $A$ has a nontrivial idempotent $e = e^2$, does this imply $(A,-)$ is not a structurable division algebra?

\section{14-dimensional quadratic forms in $I^3$ and octic polynomials}

Recall from \ref{sec:gal-com-GO} that if $(V,q)$ is a quadratic space with trivial discriminant, the cohomology set $H^1(k,\mathbf{PGO}^+(V,q))$ can be seen as the set of isomorphism classes of quadruples $(B,\sigma, C_+, C_-)$, where $(B,\sigma)$ is an orthogonal involution of degree~$n = \dim V$ and $C_+, C_-$ are $k$-algebras such that the Clifford algebra $C(B,\sigma)$ is $k$-isomorphic to $C_+ \times C_-$. Note that the fundamental relation \cite[(9.16)]{knus1998book} says $[C_+] [C_-] = 1$ in $\Br(k)$; that is, $C_+ \simeq (C_-)^{\rm op}$.

\begin{proposition} \label{prop:connecting-map}
	Let $(A,-)$ be a bi-octonion algebra. The connecting map \[\Delta: H^1(k, \mathbf{PGO}^+(S,Q)) \to H^2(k,\mathbf{G}_m) = \Br(k)\] induced by the second row of \eqref{diag.nice-diagram} is $[(B,\sigma, C_+, C_-)] \mapsto [C_-] \in \Br(k)$.
\end{proposition}

\begin{proof}
	The diagram \eqref{diag.nice-diagram} induces a pair of exact sequences with a vertical arrow between them:	\[
	\begin{tikzcd}[column sep = small]
	\dots \ar[r] & H^1(k, \mathbf{\Omega}(S,Q)) \ar[r,"\chi'_*"] \ar[d]	& H^1(k, \mathbf{PGO}^+(S,Q) )	\ar[r,"\Delta'"] \ar[d,equal] & H^2(k, \mathbf{G}_m\times \mathbf{G}_m) = \Br(k\times k) \ar[d] \\
	\dots \ar[r] & H^1(k, \mathbf{Str}(A,-)) \ar[r, "\gamma''_*"] 	& H^1(k, \mathbf{PGO}^+(S,Q))	\ar[r,"\Delta"] & H^2(k, \mathbf{G}_m) = \Br(k).
	\end{tikzcd}
	\]
	%Note that the fundamental relation \cite[(9.16)]{knus1998book} says $[C_+] + [C_-] = 0$ in $\Br(k)$; that is, $C_+ \simeq (C_-)^{\rm op}$.
	
	The connecting map $\Delta'$ is described in \cite[VII. Exercise~15]{knus1998book}: since $C^+(S,Q) \simeq M_{64}(k) \times M_{64}(k)$, the class $[(B,\sigma,C_+, C_-)]$ is mapped to the class $[C_+ \times C_-]$ in $\Br(k\times k)$.  The vertical arrow $\Br(k\times k) \to \Br(k)$ is induced by the projection $\mathbf{G}_m \times \mathbf{G}_m \to 1 \times \mathbf{G}_m \simeq \mathbf{G}_m$, so it sends $[C_+\times C_-]\mapsto [C_-]$. Therefore the connecting map $\Delta$ sends $[(B,\sigma, C_+, C_-)] \mapsto [C_-]$.
\end{proof}

Note that there is no canonical way of doing this, and had we made different conventions back in \S\ref{sec:structure groups} or elsewhere, then Proposition~\ref{prop:connecting-map} might have said that $\Delta$ maps $[(B,\sigma, C_+, C_-)]\mapsto C_+$. In more ambidextrous terms, the proposition would say: ``the connecting map $\Delta$ associated to $\gamma''$ sends an element of $H^1(k, \mathbf{PGO}^+(S,Q))$ to one of the components of its Clifford algebra, and if $\nabla$ is the connecting map associated to $\gamma''\circ f$, where $f$ is an outer automorphism of $\mathbf{Str}(A,-)$, then $\nabla$ sends an element of $H^1(k,\mathbf{PGO}^+(S,Q))$ to the other component of its Clifford algebra."

\begin{corollary} \label{cor:Str-PI}
	Let $(A,-) = C(8) \otimes C(8)$ be the split bi-octonion algebra. The map 	$H^1(k,\mathbf{Str}(A,-)) \to PI_{14}^3(k)$, sending the isotopy class of a bi-octonion algebra to the similitude class of its Albert form, is an isomorphism of pointed sets.
\end{corollary}

\begin{proof}
	Suppose $(V,q)$ is a 14-dimensional quadratic space whose Clifford algebra is isomorphic to $M_{128}(k)$; for short, $q \in I^3_{14}(k)$. Let $(B,\sigma) = (M_{14}(k), \tau_q)$ be the orthogonal involution adjoint to~$q$. Then $C(B,\sigma) = C_+\times C_-$ where $C_+ \simeq C_- \simeq M_{64}(k)$ \cite[Theorem V.2.5~(3)]{lam}, and hence $[(B,\sigma, C_+,C_-)] \in \ker(\Delta) = \im(\gamma_*'')$. Therefore $(B,\sigma) \simeq (M_{14}(k),\tau_Q)$ for some $Q$ which is an Albert form of a bi-octonion algebra $(A,-)$, and hence $q$ and $Q$ are similar.  This shows $H^1(k,\mathbf{Str}(A,-)) \to PI_{14}^3(k)$ is surjective. The injectivity was established in \ref{cor:isotopic-similar}.
%	
%
%	The image of $\gamma''_*: H^1(k,\mathbf{Str}(A,-)) \to H^1(k, \mathbf{PGO}^+_{14})$ is the kernel of $\Delta'$, namely the set of all $[(B,\sigma, C_+,C_-)]$ such that $C_-$ is trivial in $\Br(k)$. If $C_-$ is trivial then $C_- \simeq (C_-)^{\rm op}\simeq C_+$, so the fundamental relation \cite[(9.15)]{knus1998book} implies $[C_+ \times C_-]^2 = [B\times B] = 0$ in $\Br(k \times k)$, and therefore $B$ is split. Therefore the image of $\gamma''_*$ is the set of isomorphism classes of orthogonal involutions on $M_{14}(k)$ with trivial Clifford algebras, and this is in one-to-one correspondence with the set of similitude classes of 14-dimensional quadratic spaces $(V,q)$ whose even Clifford algebra is isomorphic to $M_{64}(k)\times M_{64}(k)$ \cite[Proposition~8.11]{knus1998book}. But $C^+(V,q) \simeq M_{64}(k) \times M_{64}(k)$ if and only if $C(V,q) \simeq M_{128}(k)$ \cite[V.2.5~(3)]{lam}, if and only if $q \in I^3(k)$. Therefore the map $H^1(k,\mathbf{Str}(A,-)) \to {PI}_{14}^3(k)$ is surjective. 
\end{proof}

Note that since $\im(\gamma_*'') = \ker(\Delta)$, Proposition~\ref{prop:connecting-map} also implies that for an orthogonal involution $(A,\sigma)$ of degree $14$, if one of the components of $C(A,\sigma)$ is split, then $A$ is split and $\sigma$ is adjoint to a quadratic form in $I^3$. But this fact is actually easy to prove  using the fundamental relations \cite[Theorem~19.12]{knus1998book} for any degree $n \equiv 2 \mod 4$ \cite[Lemma~1.5]{garibaldi2007orthogonal}.

\begin{corollary}[Rost] \label{cor:Rost}
	Let $Q \in I^3_{14}(k)$. Then either:
	\begin{enumerate}[\rm (1)]
	\item There exist $3$-Pfister forms $\phi_1$ and $\phi_2$ over $k$ and a scalar $c \in k^\times$ such that \[Q \simeq \langle c \rangle (\phi_1' \perp \langle -1 \rangle \phi_2').\]
	\item There exists a quadratic field extension $E/k$, a $3$-Pfister form $\phi$ over $E$, and an element $\delta \in E^\times$ of trace zero such that \[ Q \simeq  T_{E/k}(\langle \delta \rangle \phi').
	\]
	\end{enumerate}
\end{corollary}
\begin{proof}
	By Corollary \ref{cor:Str-PI}, $Q$ is similar to the Albert form of a bi-octonion algebra, and these are precisely the two possibilities for a form that is similar to the Albert form of a bi-octonion algebra.
\end{proof}

The above proof of Rost's Theorem is different from those that have previously been sketched in \cite{rost14dim} and \cite[Theorem~21.3]{garibaldi2009cohomological}. 
Corollary~\ref{cor:Rost} is an example of a ``rational parameterisation". Similar parameterisation theorems exist for quadratic forms in $I_n^3(k)$ for all even $n \le 14$  \cite[Theorem~2.1]{hoffmann1998}. The case of $n = 12$ will be discussed further in \ref{sec:I12}. Notably, there is no known rational parameterisation of quadratic forms in $I_n^3(k)$ for $n > 14$, and it is conjectured by Merkurjev that no rational parameterisation exists (see \cite[Conjecture~4.5]{merkurjev2017invariants} and \cite[Corollary~11]{merkurjev2020versal}).

\subsection{Some existence results on quadratic forms in $I_{14}^3$.} The two possibilities in Rost's Theorem  are not mutually exclusive.
%For example, if $k$ has a quadratic field extension $E = k(\sqrt{d})$, then the hyperbolic form $7\mathbb{H}$ is similar to both $\llangle 1,1,1 \rrangle' \perp \langle -1 \rangle \llangle 1,1,1 \rrangle'$ and $T_{E/k}(\langle \sqrt{d} \rangle \llangle 1,1,1 \rrangle')$ and so it satisfies both (1) and (2).
In fact,  any isotropic form  $q$ in $I_{14}^3(k)$ satisfies (1) -- more on this in \ref{sec:I12}. Quadratic forms satisfying (2) but not (1) must therefore be anisotropic. They exist, and examples have been found by Izhboldin and Karpenko \cite[Corollary~17.4]{izhboldin2000} and Hoffmann and Tignol \cite[p.~\&~211]{hoffmann1998}; probably the most accessible example is over the field $\mathbb{Q}(t_1, t_2, t_3, t_4)$. 
Quadratic forms satisfying (1) but not (2) also exist: for example over the field $k_n = \mathbb{R}((t_1, \dots, t_n))$ where $n \ge 3$,  every anisotropic form in $I_{14}^3(k_n)$ satisfies (1) but not (2) \cite[Theorems~7.3 \&~7.13]{allison1988tensor}.

\subsection{Pfister's parameterisation of  quadratic forms in $I_{12}^3$} \label{sec:I12}

By a theorem of Pfister (proved in \cite[pp.~123--124]{pfister1966quadratische} and differently in \cite[Theorem~17.13]{garibaldi2009cohomological}), for every quadratic form $q \in I^{3}_{12}(k)$ there exists a $6$-dimensional quadratic form $r \in I_6^2(k)$ and a scalar $c \in k^\times$ such that $q \simeq \llangle c \rrangle  r$. Moreover, for any form $r \in I_{6}^2(k)$, there exist scalars $d, x_1, x_2, y_1, y_2 \in k^\times$ such that $r = \langle d \rangle (\psi_1' \perp \langle -1 \rangle \psi_2')$ where $\psi_i = \llangle x_i, y_i \rrangle$. That is to say, 
\begin{equation} \label{eq:q in I12}
q \simeq  \langle d \rangle \llangle c \rrangle (\psi_1'  \perp \langle -1 \rangle \psi_2').
\end{equation}
This also means that every isotropic quadratic form $Q \in I_{14}^3(k)$ is of the form
\[Q = q + \mathbb{H} = \langle d \rangle (\llangle c, x_1, y_1  \rrangle' \perp \langle -1 \rangle \llangle c, x_2, y_2 \rrangle').\]
Combining this result with Corollary \ref{cor:isotopic-similar} and Theorem \ref{thm:albert-form-division-algebra} yields the following theorem:

\begin{theorem} \label{thm:divisionness0}
	Let $(A,-)$ be a bi-octonion algebra. The following are equivalent:
	\begin{enumerate}[\rm (1)] \item $(A,-)$ is not a structurable division algebra.
	\item $(A,-)$ is isotopic to a decomposable bi-octonion algebra $C_1 \otimes C_2$ where $C_1$ and $C_2$ are octonion algebras whose norms have a common slot.	
	\item The Albert form of $(A,-)$ is similar to $\phi_1' \perp \langle -1 \rangle \phi_2'$ where $\phi_i$ are $3$-Pfister forms with a common slot.
	\end{enumerate}
\end{theorem}

\subsection{The octic norm of a bi-octonion algebra}
	The norm $N_A$ of a structurable algebra $(A,-)$ is preserved up to similitude under the action of $\mathbf{Str}(A,-)$,  hence $N_A$ is preserved up to isomorphism under the action of $\mathbf{Str}(A,-)^{\rm der}$. If $(A,-)$ is a bi-octonion algebra, this means that its norm  is an invariant polynomial for $\mathbf{Spin}(S,Q)$, where $Q$ is an Albert form for $(A,-)$. When $(A,-)$ is split, the polynomial $N_A$ is traditionally denoted by $J$. It follows from Lemma \ref{lem:dense-orbit}~(i) and then \cite[Proposition~6.1]{bermudez2014linear} applied to $\mathbf{Spin}_{14}$ that $J$ in fact generates the ring of $\mathbf{Spin}_{14}$-invariant polynomials; that is:
	\[k[A]^{\mathbf{Spin}_{14}} = k[J].\]
	It is well-known in invariant theory that $\deg J = 8$: in characteristic zero this was proved in \cite[Proposition~40]{sato1977classification} and \cite[Proposition~13]{popov1980classification}.  Also including characteristic $p > 3$,  Allison and Faulkner  proved that $\deg J = 8$ and  even gave a formula for $N_A$ for arbitrary bi-octonion algebras \cite[Theorem~9.6]{allison1992norms}. Their formula is reproduced in \eqref{bi-oct norm}. Another concrete but unwieldy expression for $J$ (over the complex numbers) was calculated by Gyoja in \cite[Theorem~5]{gyoja1990construction}.
	
	In this subsection and the next one, we assemble some curious and remarkable facts about this octic form $J$ and build a picture of how bi-octonion algebras, quadratic forms, octic forms, and algebraic groups fit in with one another.  	  From Theorem~\ref{thm:ns-8,8 and 8,2}  and  Corollary~\ref{cor:Str-PI}, we have  $PI_{14}^3(k)  \simeq H^1(k,\mathbf{Str}(C(8)\otimes C(8))) = H^1(k, \mathbf{GO}(J)) $. Since $H^1(k, \mathbf{Str}(C(8) \otimes C(8))$ classifies bi-octonion algebras up to isotopy (see \ref{sec:galois cohomology of structurable algebras}), and $H^1(k,\mathbf{GO}(J))$ classifies $k$-forms of $J$ up to similitude \cite[III, Exercise~2]{berhuy}, there are one-to-one correspondences:
	\begin{center}
\fbox{\begin{minipage}[c][2.5cm]{3.9cm} \centering
Similitude classes of 14-dimensional quadratic forms in $I^3(k)$
\end{minipage}}\quad  $\longleftrightarrow$\quad\fbox{\begin{minipage}[c][2.5cm]{3.9cm} \centering
Isotopy classes of bi-octonion $k$-algebras
\end{minipage}} \quad  $\longleftrightarrow$\quad\fbox{\begin{minipage}[c][2.5cm]{3.9cm} \centering
Similitude classes of 64-dimensional octic forms over $k$ that become isomorphic to $J$ over $k^{\rm sep}$.
\end{minipage}}
\end{center}
	The leftward arrow sends a bi-octonion algebra to the class of its Albert quadratic form, and the rightward arrow sends the algebra to the class of its norm. Moreover, Theorem~\ref{thm:albert-form-division-algebra} implies that a quadratic form in $I_{14}^3(k)$ is anisotropic if and only if its corresponding octic form is anisotropic.
	
	 By Theorem \ref{thm.types}, Theorem \ref{main2}, and classification results on simple algebraic groups \cite{tits1966classification}, the three-way correspondence above can also be extended to a five-way correspondence including:
	\begin{center}
	\fbox{\begin{minipage}[c][1.5cm]{6.5cm} \centering
Isomorphism classes of strongly inner simply connected simple  algebraic $k$-groups of type $D_7$
\end{minipage}} \quad  $\longleftrightarrow$ \quad\fbox{\begin{minipage}[c][1.5cm]{6.5cm} \centering
Isomorphism classes of simple algebraic $k$-groups with Tits index

$E_{8,1}^{91}$, $E_{8,2}^{66}$, $E_{8,4}^{28}$, or $E_{8,8}^0$. 
\end{minipage}}
\end{center}
Concretely, these last two correspondences are defined by sending a bi-octonion algebra to the derived subgroup of its structure group (a simply connected group of type $D_7$) or to the  automorphism group of its TKK Lie algebra (an isotropic group of type $E_8$). Equivalently, the simply connected group of type $D_7$ contains the semisimple anisotropic kernel of the corresponding group of type $E_8$.

	\subsection{Matrix factorisations}

Let $P \in k[x_1, \dots, x_n]$ be a nonzero polynomial in $n$ variables. A \emph{matrix factorisation} of $P$ is a pair $(N, M)$ where $N, M$ are square $t\times t$ matrices  ($t \ge 1$) with entries in $k[x_1, \dots, x_n]$, such that $NM = P.\id_{t\times t}$. What is interesting about this concept is that many irreducible polynomials (i.e., having no factorisation by a pair of nonconstant polynomials) do have matrix factorisations. If $P$ is homogeneous, say of degree~$d$, and the entries of $N$ and $M$ are all of degree $j \ge 1$ and $d-j$ respectively, then the number of coefficients appearing in $N$ and $M$ is potentially much smaller than the number of coefficients appearing in $P$. This is especially true if $j$ is close to $\frac{1}{2}d$, and even more so if $N = M$. In other words, a matrix factorisation is often an incredibly efficient way to represent $P$.

Recently it was discovered that over the complex numbers the $\mathbf{Spin}_{14}$-invariant octic polynomial $J$ admits a factorisation as the square of a certain $14 \times 14$ matrix \cite[Theorem~2.3.2]{abuaf2021gradings}. We are able to describe this matrix factorisation explicitly in terms of bi-octonion algebras. Surprisingly, non-split forms of $J$ are also matrix-factorisable over the base field. The matrices in the factorisation can be calculated in reasonable time using computer algebra software.

\subsection{$P$-operators}
In any structurable algebra $(A,-)$, there is a one-parameter family of operators $P_x \in \End A$, $x \in A^*$, defined as
\[
P_x(a) = \tfrac{1}{3}U_x(5a-2V_{a,x}\hat x)
\]
for all $a \in A$. Despite the exotic definition, these operators have a number of nice properties: for example $P_x \in \Str(A,-)$ and \begin{equation} \label{eq:P-properties}\hat{P}_x = P_{\hat x} = P_x^{-1}\end{equation} for all $x \in A^*$ \cite[Theorem~8.3]{allison1981isotopes}. Since $P_x \in \Str(A,-)$, we can define the map $(P_x)_S \in \End S$ as in \eqref{eq.alpha-s}.

\begin{lemma} \label{lem:P} Let $(A,-)$ be a structurable algebra. For all $x \in A^*$ and $s \in   \Skew(A,-)$, \begin{equation} \label{eq:mat-factorisation} (P_x)_S(s) = \tfrac{1}{6}\psi(x, U_x(sx)).\end{equation} \end{lemma}

\begin{proof}
	The proof is loosely based on \cite[Lemma~3.3.4]{boelaert2019moufang}. Using \eqref{eq:alphaL}, we have
	\begin{equation} \label{eq:f0}
	L_{(P_x)_S(s)}\hat P_x = P_x L_s.
	\end{equation}
	An expression for $\hat x$ derived by Allison (see \cite[Proposition~2.6]{allison1986conjugate}, or the more accessible sources \cite[eq.\ (1.7)]{allison1992norms} and \cite[eq.\ (2.14)]{boelaert2019moufang}),  combined with the fact that $L_{s}^{-1} = -L_{\hat s}$, yields: 
	\begin{equation} \label{eq:f1}
		U_x(sx) = - \tfrac{1}{2}\psi(x, U_x(sx))\hat x.
	\end{equation}
	Two more necessary identities appear on \cite[p.\ 33]{boelaert2019moufang}; these are (reasonably direct) consequences of the definition of $P_x$: \begin{align} \label{eq:f2}
 U_x(sx) &= -3P_x(sx),\\ \label{eq:f3} P_x(\hat x) &= x.
 \end{align}
 	Applying \eqref{eq:f1}, \eqref{eq:f2}, \eqref{eq:f0}, and \eqref{eq:f3} in that order, it follows that
  	\begin{align*}
	\tfrac{1}{6}\psi(x, U_x(sx))\hat x &= -\tfrac{1}{3}U_x(sx) =   P_x(sx)  =   P_x L_s (x)  \\&= L_{(P_x)_S(s)} \hat P_x (x) = L_{(P_x)_S(s)} \hat x = (P_x)_S(s)\hat x.
	\end{align*}
	The linear map $R_{\hat x}|_S: S \to S\hat x$, $s \mapsto s\hat x$, is bijective because $S \hat x= \Skew(A^{\langle \hat x \rangle}, -^{\langle \hat x\rangle })$ and $\dim(\Skew(A,-)) = \dim\Skew(A^{\langle \hat x \rangle}, -^{\langle \hat x\rangle })$ \cite[Corollary~12.2]{allison1981isotopes}, hence \eqref{eq:mat-factorisation}.
	\end{proof}

An important role of Lemma \ref{lem:P} is to show that $x \mapsto (P_x)_S$ extends to a globally defined rational mapping $A \to \End(S)$ such that $A^*$ is mapped into $\GL(S)$. That is, even though $P_x$ might not be defined for all $x \in A$, it is clear that $\tfrac{1}{6}\psi(x, U_x(sx))$ is defined for all $x \in A$ and $s \in \Skew(A,-)$.

\subsection{A matrix factorisation of the octic norm}

Now assume that $(A,-)$ is a bi-octonion algebra. The composition
\[
\begin{tikzcd} A^* \ar[r,"P"] & \mathbf{Str}(A,-) \ar[r, "\gamma"] & \mathbf{GO}(S,Q) \ar[r, "\mu"] & \mathbf{G}_m	
\end{tikzcd}
\]
is a well-defined map of varieties. By Lemma \ref{lem:P}, this composition is the restriction of a genuine polynomial function $A \to k$, which we can define by picking an arbitrary conjugate-invertible basepoint $s_0 \in \Skew(A,-)$ and sending
\begin{equation} \label{bi-oct norm}
x \mapsto   \frac{1}{36Q(s_0)}Q(\psi(x, U_x(s_0 x))).
\end{equation}
This polynomial function is none other than the norm of $(A,-)$  \cite[Theorem~9.6]{allison1992norms}, a claim which can be justified by the uniqueness property of the norm (namely that it is the unique normalised invertibility-detecting polynomial function of minimal degree on $A$). We repeat for emphasis that $N_A(x) = \mu((P_x)_S)$ if $x \in A^*$. 

\begin{theorem}
Let $(A,-)$ be a bi-octonion algebra with $\Skew(A,-) = S$. Define for all $x \in A$ the linear map $M_x \in \End S$,
\begin{align*}
M_x(s) = \frac{1}{6}\psi(x,U_x(s^\sharp x)).
\end{align*}
Then
\[M_x^2 = N_A(x).\id	_S.\]
\end{theorem}

\begin{proof}
	Suppose $x \in A^*$. Note that $M_x(s) = (P_x)_S(s^\sharp) = \gamma(P_x)(s^\sharp)$ by Lemma \ref{lem:P}. We have by \eqref{eq.hat-alpha} and \eqref{eq:P-properties} that \begin{align*}\id_S &= \gamma(P_x P_x^{-1}) = \gamma( P_x \hat P_x) = \gamma( P_x) \gamma(\hat P_x) = (P_x)_S (\hat P_x)_S = \frac{1}{\mu((P_x)_S)}(P_x)_S \circ \sharp  \circ (P_x)_S \circ \sharp  \end{align*}
	and hence
	\[
	N_A(x).\id_S = \mu((P_x)_S).\id_S =  M_x^2.
	\]
	Extending to an infinite field if necessary, $M_x$ and $N_A(x)$ remain polynomial and $A^*$ is Zariski-dense in $A$. Therefore $N_A(x).\id_S = M_x^2$ for all $x \in A$. 
\end{proof}

\section{Cohomological invariants of bi-octonion algebras} \label{sec:cohom_inv_of_bi-oct_alg}

\subsection{Degree $n$ invariants of $I^n$} \label{sec:en}
Since the Milnor Conjectures are true, there exists an infinite sequence of mod 2 cohomological invariants
\begin{align*}
e_n: I^n \to H^n(*,\ZZ/2\ZZ) && (n \ge 0)
\end{align*}
such that for all fields $L/k$, $e_n: I^n(L) \to H^n(L,\ZZ/2\ZZ)$ is the unique additive group homomorphism with
$	e_n(\llangle a_1, \dots,  a_n \rrangle) = (a_1){\cdot} \cdots {\cdot} (a_n)$  for all $ a_1, \dots, a_n \in L^\times$.
For all $n \ge 0$, the kernel of $e_n$ is $I^{n+1}$, and this implies 
\begin{align*}
e_n(\langle c \rangle q) = e_n(q) && \text{for all } c \in L^\times, q \in I^n(L).
\end{align*}
For a concise discussion on the invariants $e_n$, an excellent reference is~\cite[\S1.1]{berhuy2007cohomological}. The existence of these invariants is a very deep result. However, the invariants $e_0$ (dimension modulo~2), $e_1$ (discriminant), and $e_2$ (Clifford invariant), are more classical; see \cite[\S1]{tignol2010cohomological} for a nice exposition. The existence of $e_3$ (Arason invariant \cite{arason1975cohomologische}) and $e_4$ \cite{jacob1989degree} was also established prior the resolution of the Milnor Conjectures, and then the rest of the $e_n$'s were dealt with in one fell swoop as a consequence of Voevodsky's work.

\subsection{Galois cohomology of $(8,8)$-product algebras} \label{sec:gal-com-of-(8,8)}

Let $G$ be the automorphism group of the split octonion algebra over $k$; that is, $G$ is the split simple algebraic group of type $G_2$. Let $S_2$ be the constant finite algebraic group of order 2.
By Theorem~\ref{thm.aut2}, the automorphism group of the split (8,8)-product algebra is $G^2 \rtimes S_2 = (G_2 \times G_2)\rtimes S_2$. The split short exact sequence of algebraic groups
$G^2 \to G^2\rtimes S_2 \to  S_2$
induces an exact sequence of pointed sets, in which we denote the third arrow by $b_1$:
\begin{equation} \label{eq:split-exact}
\begin{tikzcd}
	S_2 \ar[r] & H^1(k,G^2) \ar[r]& H^1(k,G^2\rtimes S_2) \ar[r,"b_1"] & k^\times/k^{\times 2} \ar[r] & 1.
\end{tikzcd}
\end{equation}

\begin{lemma} \label{lem.centroid}
	If $\beta \in H^1(k,G^2\rtimes S_2)$ and $(A,-)$ is an $(8,8)$-product algebra corresponding to $\beta$, then $b_1(\beta)$ is the isomorphism class of the centroid of $\Skew(A,-)^-$.
\end{lemma}

\begin{proof}
	 Clearly $b_1$ and the map $[(A,-)] \mapsto [\operatorname{Centr}(\Skew(A,-)^-)]$ are both nonzero normalised cohomological invariants $H^1(*,G^2\rtimes S_2) \to H^1(*,S_2)$. By \cite[Proposition 31.15]{knus1998book}, there is exactly one such invariant.
\end{proof}

Now suppose $E/k$ is a quadratic \'etale extension corresponding to some  $\varepsilon \in H^1(k,S_2) = k^\times / k^{\times 2}$. The group $S_2$ acts on $E/k$ by $k$-automorphisms, and  by functoriality it acts on $H^1(k, R_{E/k}(G_E))$.

\begin{lemma} \label{lem.fibres} The fibre $b_1^{-1}(\varepsilon)$ is in natural bijective correspondence with the orbit space \[\frac{H^1(k,R_{E/k}(G_{E}))}{S_2}.\]	
\end{lemma}

\begin{proof}
If $C$ is the split octonion algebra over~$E$, then the automorphism group of $\cor_{E/k}(C)$ is $R_{E/k}(G_E)\rtimes S_2$ by Theorem \ref{thm.aut2}, and we have an exact sequence:
\begin{equation} \label{eq.twisted-sequence}
	\begin{tikzcd}[sep=small] 
	S_2 \ar[r] & H^1(k,R_{E/k}(G_{E})) \ar[r]& H^1(k,R_{E/k}(G_{E})\rtimes S_2) \ar[r] & k^\times/k^{\times 2} \ar[r] & 1.
\end{tikzcd}
\end{equation}
This is the sequence  \eqref{eq:split-exact} twisted by a cocycle $b \in Z^1(k,G^2\rtimes S_2)$ representing the isomorphism class of $\cor_{E/k}(C)$. The lemma follows from  \cite[I.\S5.5 Corollary 2]{serre1997galois}.
\end{proof}

\subsection{Example} \label{example:split-case}
	Consider \eqref{eq.twisted-sequence} in the split case when $E = k\times k$. Then $R_{E/k}(G_E) = G^2$. We can identify $H^1(k,G^2)$ with $H^1(k,G)\times H^1(k,G)$ as $S_2$-sets, where $S_2$ acts on $H^1(k,G)\times H^1(k,G)$ by swapping. From this point of view, $H^1(k,G^2)$ is the set of ordered pairs of octonion $k$-algebras and  the map $H^1(k,G^2)\to H^1(k,G^2\rtimes S_2)$ sends the class of $(C_1, C_2)$ to the class of $C_1 \otimes C_2$.   The image of this map is $\ker(b_1)$, and this is of course the set of isomorphism classes of decomposable bi-octonion algebras. The characterisation in Lemma~\ref{lem.fibres} of $\ker(b_1)$ as the orbit space $H^1(k, G^2)/S_2$ is compatible with our understanding that $C_1 \otimes C_2 \simeq C_1' \otimes C_2'$  if and only if   $C_{1} \simeq C_{\sigma(1)}'$ and $C_2 \simeq C_{\sigma(2)}'$ for some $\sigma \in S_2$.
	
\subsection{Partitioning the cohomology set} \label{sec:partitioning} By Lemma \ref{lem.fibres}, we have a canonical isomorphism:
\begin{equation} \label{eq.coprod}
H^1(k,G^2\rtimes S_2) \simeq \coprod_{a \in k^\times/k^{\times 2}} \frac{H^1(k, R_{k(\sqrt{a})/k}(G_{k(\sqrt{a})}))}{S_2}
\end{equation}
By Shapiro's Lemma \cite[Lemma 29.6]{knus1998book}, for all quadratic \'etale extensions $E/k$ there is a canonical isomorphism $H^1(k, R_{E/k}(G_E)) \simeq H^1(E,G)$. The pointed set $H^1(E,G)$ is identified with the set of $E$-isomorphism classes of octonion algebras over $E$, and the quotient of $H^1(E,G)$ by $S_2$ is identified with the set of $k$-isomorphism classes of octonion algebras over~$E$. In other words, \eqref{eq.coprod} is  a cohomological version of Theorem~\ref{thm.equiv2} for (8,8)-product algebras.

\subsection{Classification of bi-octonion algebras by successive invariants}

We can define two successive invariants of bi-octonion algebras which classify them up to isomorphism. The first invariant is $b_1$ from Lemma~\ref{lem.centroid}. If $b_1(A,-)  = [E] \in H^1(k, \ZZ/2\ZZ)$ then $(A,-) \simeq \cor_{E/k}(C)$ for a certain octonion algebra over $E$, unique up to $k$-isomorphism. The group $S_2\simeq \Aut_k(E) = \{1, \iota\}$ acts on $H^3(E,\ZZ/2\ZZ)$ by functoriality; on symbols this action is $(a){\cdot}(b){\cdot}(c) \mapsto ({^\iota a}){\cdot}({^\iota b}){\cdot}({^\iota c})$.
 Let $n_C$ be the norm of $C$ and define  $b_{[E]}(A,-) = \{e_3(n_C), e_3({^\iota n_C})\} \in H^3(E, \ZZ/2\ZZ)/S_2$.  Recall that octonion algebras over $E$ are classified up to $E$-isomorphism by the invariant $C \mapsto e_3(n_C)$ \cite[Theorem~5.4]{petersson1996elementary}, and therefore they are classified up to $k$-isomorphism by the invariant $b_{[E]}$. In other words, the map
 \[
 b_{[E]}: \frac{H^1(k, R_{E/k}(G_E))}{S_2} \longrightarrow \frac{H^3(E, \ZZ/2\ZZ)}{S_2}
 \] is both well-defined and injective. In summary, we have two invariants $b_1$ and $b_{[*]}$ that, when applied successively, classify bi-octonion algebras up to $k$-isomorphism.
 
 Note however, that $b_{[*]}$ is not a mod 2 cohomological invariant in the sense of \ref{sec:cohomological-invariants}, i.e., it is {not} an element of $\Inv(G^2 \rtimes S_2)$. We discover later in Corollary~\ref{cor:non-classification_bi-octonion} that bi-octonion algebras are \emph{not} classified by  invariants in $\Inv(G^2 \rtimes S_2)$. This situation is comparable to the classification of cubic \'etale algebras by cohomological invariants \cite[Proposition~30.18]{knus1998book}. In both situations, the invariants need to be applied successively in order to obtain classifying data, and the second invariant takes values in an orbit space of a cohomology group. 
 
\subsection{Cohomological invariants of $G^2\rtimes S_2$ in degree $\le 3$} \label{sec:b1b3}
We have a degree~1 invariant $b_1: H^1(*,G^2\rtimes S_2) \to H^1(*,\ZZ/2\ZZ)$ from Lemma \ref{lem.centroid}, namely
$
b_1(A,-) = [E]
$
where $E$ is the centroid of $\Skew(A,-)^-$.

We can define a degree 3 invariant as follows. Let $e_3: I^3(*) \to H^3(*,\ZZ/2\ZZ)$ be the Arason invariant, and define $b_3: H^1(*,G^2\rtimes S_2) \to H^3(*,\ZZ/2\ZZ)$:
\[b_3(A,-) = e_3(Q)\]
where $Q$ is an Albert form for $(A,-)$. Since $e_3$ is constant on similitude classes, no ambiguity arises if we choose a different Albert form $\langle c \rangle Q$ instead of $Q$.

\subsection{Multiplicative transfer} \label{sec:multiplicative-transfer}
The notion of multiplicative transfer first appeared in work by Rost \cite{rost2003multiplicative} and Tignol \cite{tignol1994norme}. Let $\iota$ be the nonidentity automorphism of $E/k$. Define ${^\iota V}$ to be the $E$-module $^\iota V = V$ with $e\cdot v = \iota(e)v$ for all $e \in E$, $v \in V$, and define $^\iota q(v) = \iota(q(v))$ for all $v \in {^\iota V}$. This defines a quadratic space $({^\iota V}, {^\iota q})$ over $E$. The \emph{multiplicative transfer} of $(V,q)$ is the $n^2$-dimensional $k$-quadratic space $N_{E/k}(V,q) = (N_{E/k}(V), N_{E/k}(q))$ defined as follows: $N_{E/k}(V)$ is the subspace of ${^\iota V}\otimes_E V$ fixed by the switch map $x \otimes y \mapsto y \otimes x$, and $N_{E/k}(q)$ is the restriction of ${^\iota q} \otimes_E q$ to $N_{E/k}(V)$. For one-dimensional quadratic forms, the multiplicative transfer behaves straightforwardly \cite[Lemma 2.6~(i)]{wittkop}:
\begin{align} \label{eq:N-one-d}
	N_{E/k}(\langle e \rangle) = \langle N_{E/k}(e) \rangle && \text{for all } e \in E.
\end{align}
	For quadratic forms in dimension $> 1$, it is less straightforward, for example if $E = k(\sqrt{d})$ is a field then  by \cite[Lemma~2.13]{wittkop},
\begin{equation*} 
	N_{E/k}(\mathbb{H}) = \langle 2 \rangle \llangle d \rrangle + \mathbb{H} = \langle 2, -2d, 1, -1 \rangle.
\end{equation*}

Unlike the additive transfer $T_{E/k}$, the multiplicative transfer  $N_{E/k}$ is not generally compatible with Witt equivalence, so it does not define an operation on $W(k)$. However, $N_{E/k}$ does extend to the Grothendieck--Witt ring; specifically, \cite[Lemma~2.6~(iii) \& Satz~2.9]{wittkop} show that there is a unique multiplicative map $N_{E/k}: \widehat{W}(E) \to \widehat{W}(k)$ such that $N_{E/k}([q]) = [N_{E/k}(q)]$ for all quadratic forms $q$ over~$E$, and $N_{E/k}(-1) = 3 - \langle 2 \rangle\llangle d \rrangle$.
  In the split case where $E = k\times k$, and $q = (q_1, q_2)$ for quadratic forms $q_i$ over $k$, we have $N_{E/k}(q) \simeq q_1 \otimes q_2$.

The following theorem says what $N_{E/k}$ does to Pfister forms; it is a direct consequence of \cite[Satz 2.16~(ii)]{wittkop}, whose statement also appears without proof in \cite{rost2006galois}.

\begin{theorem}[Rost, Wittkop] \label{thm:norm}
	Let $E=k(\sqrt{d})$ be a quadratic field extension, and let $q = \llangle c_1, \dots, c_n \rrangle$ where $c_i \in E$. If $\tr_{E/k}(c_i) = 0$ for some $i$, then \[N_{E/k}(q) \simeq 2^{n-1}(2^n-1)\mathbb{H} \perp 2^{n-1}\llangle d \rrangle.\] Otherwise, \[N_{E/k}(q) \perp 2^{n-1}\mathbb{H} \simeq 2^{n-1}\llangle d  \rrangle \perp \bigotimes_{i = 1}^n \llangle \tr_{E/k}(c_i),-dN_{E/k}(c_i) \rrangle.\]
	Consequently, in $W(k)$ we have $N_{E/k}(q) - 2^{n-1}\llangle  d \rrangle \in I^{2n}(k)$.
\end{theorem}

\subsection{The quadratic trace form} If $(A,-)$ is any structurable algebra, the bilinear trace $T_A: A \times A \to k$ is defined as
\[
	T_A(x,y) = \tr(L_{x \bar y + y \bar x}).
\]
We call $T_A(x) = T_A(x,x)$ the quadratic trace form. The bilinear trace is an invariant form in the sense that Allison defined it in \cite[\S6]{allison1978class}:
\begin{align*}
	T_A(\bar{x}, \bar{y}) = T_A(x,y) && \text{and} && T_A(zx,y) = T_A(x,\bar{z}y) && \text{for all } x,y,z \in A.
\end{align*}
If $(A,-)$ is simple then $T_A$ is either nondegenerate or zero, because its radical is an ideal stabilised by the involution.
 According to the main theorem and the endnote of \cite{schafer1989invariant}, $T_A$ is a scalar multiple of every nonzero invariant symmetric bilinear form on $A$.

\begin{lemma} \label{lem.Trace}
		Suppose  $C$ is an octonion algebra over a quadratic \'etale extension $E/k$ with norm $n$, and let $(A,-) = \cor_{E/k}(C)$. Then $T_A = \langle 128 \rangle N_{E/k}(n)$.
\end{lemma}

\begin{proof}
	The symmetric bilinear form associated to $N_{E/k}(n)$ is nondegenerate and invariant (cf.\! \cite[Proposition 2.2~(i)]{allison1988tensor}), so it is a scalar multiple of $T_A$. The scaling factor is $128$ because  $T(1) = \tr(L_{2}) = 128$ and $N_{E/k}(n)(1) = n(1){^\iota n}(1) =  1$.
\end{proof}

\subsection{A cohomological invariant of $G^2\rtimes S_2$ in degree $6$} \label{sec:b6}

Suppose $L/k$ is a field extension and $(A,-) = \cor_{E/L}(C)$ for some quadratic \'etale extension $E/L$ and an octonion algebra $C$ over $E$. Define the invariant $b_6: H^1(*,G^2\rtimes S_2) \to H^6(*,\ZZ/2\ZZ)$ by
\[
	b_6(A,-) = e_6(N_{E/L}(n_C) - 4n_E)
\]
where $n_C \in W(E)$ and $n_E \in W(L)$ are the standard norms of $C$ and $E$ respectively. (Note that $n_E = \llangle d \rrangle$ if $E = L(\sqrt{d})$ is a field and $n_E = \mathbb{H}$ if $E = L\times L$ is split.) This invariant is well-defined because by Theorem \ref{thm:norm}, $N_{E/L}(n_C) -4n_E  \in I^6(L)$ if $E$ is a field, and $N_{E/L}(n_C)$ is a $6$-Pfister form if $E$ is not a field.

This invariant has an even more direct description, in light of Lemmas \ref{lem.centroid} and \ref{lem.Trace}, namely:
\[
	b_6(A,-) = e_6(\langle 128\rangle T - 4S)
\]
where $T\in W(L)$ is the quadratic trace form of $(A,-)$ and $S\in W(L)$ is the norm form of the centroid of $\Skew(A,-)^-$.

\subsection{Example: real bi-octonion algebras} \label{sec:reals}

There are exactly two isomorphism classes of real octonion algebras (the split octonion algebra $\mathbb{O}_{\rm split}$ and the division octonion algebra $\mathbb{O}_{\rm div}$). Let $\mathbb{O}_\mathbb{C}$ denote the unique (split) complex octonion algebra. By Theorem~\ref{thm.equiv1}, there are four isomorphism classes of real bi-octonion algebras, namely $\mathbb{O}_{\rm split} \otimes \mathbb{O}_{\rm div}$, $\mathbb{O}_{\rm div} \otimes \mathbb{O}_{\rm div}$, $\mathbb{O}_{\rm split} \otimes \mathbb{O}_{\rm split}$, and $\cor_{\mathbb{C}/\mathbb{R}}(\mathbb{O}_{\mathbb{C}})$. The first one, $\mathbb{O}_{\rm split} \otimes \mathbb{O}_{\rm div}$ has an Albert form of signature $(11,3)$, and the other three algebras have hyperbolic Albert forms. Therefore there are exactly two isotopy classes of real bi-octonion algebras.
The real bi-octonion algebras all have distinct invariants:

\begin{center}
\begin{tabular}{c|x{2cm}x{2cm}x{2cm}}
$(A,-)$ & $b_1(A,-)$ & $b_3(A,-)$ & $b_6(A,-)$ \\ \hline
$\mathbb{O}_{\rm split} \otimes \mathbb{O}_{\rm split}$ & $0$ & $0$ & $0$  \\
$\mathbb{O}_{\rm split} \otimes \mathbb{O}_{\rm div}$ & $0$ & $(-1)^3$ & $0$ \\
$\mathbb{O}_{\rm div} \otimes \mathbb{O}_{\rm div}$ & $0$ & $0$ & $(-1)^6$  \\
$\cor_{\mathbb{C}/\mathbb{R}}(\mathbb{O}_\mathbb{C})$ & $(-1)$ & $0$ & $0$ \\
\end{tabular}
\end{center}

We return to assuming that $k$ is an arbitrary field of characteristic not $2$ or $3$. The following theorem provides some applications of the cohomological invariants of $G^2\rtimes S_2$.

\begin{theorem} \label{thm:symbols}
	Let $(A,-)$ be a bi-octonion algebra.
	\begin{enumerate}[\rm (i)] \item  $b_3(A,-)$ has symbol length $\le 3$. If the symbol length is equal to $3$ then $(A,-)$ is a division algebra, and if $b_3(A,-) = 0$ then $(A,-)$ is isotopic to the split bi-octonion algebra.
		\item $b_6(A,-)$ is a symbol.
		\item If $(A,-)$ is not a division algebra then $b_6(A,-) \in (-1) {\cdot} H(k)$.	In particular, if $\sqrt{-1} \in k$ then $b_6(A,-) \ne 0$ implies $(A,-)$ is a division algebra.
	\end{enumerate}
\end{theorem}

Recall from~\ref{sec:reals} that there is a non-division real bi-octonion algebra with nonzero $b_6$, so the assumption $\sqrt{-1} \in k$ is necessary in the last sentence of the theorem.

\begin{proof}
	(i) By definition $b_3(A,-) = e_3(Q)$ where $Q \in I^3_{14}(k)$ is an Albert form for $(A,-)$. A result proved independently by Izhboldin and Hoffmann (see \cite[Remark~17.7]{izhboldin2000} and \cite[Proposition~2.3]{hoffmann1998}) says that any $Q \in I_{14}^3(k)$ is a sum of at most three forms that are similar to $3$-Pfister forms, hence the symbol length of $e_3(Q)$ is at most~3. If $(A,-)$ is not a division algebra, then $Q$ is isotropic and Pfister's theorem says it is similar to a difference of two 3-Pfister forms (see \ref{sec:I12}), hence the symbol length of $e_3(Q)$ is $\le 2$. If $e_3(Q) = 0$ then $Q \in I_{14}^4(k)$ and is therefore hyperbolic by the Arason--Pfister Hauptsatz, so  by Corollary~\ref{cor:isotopic-similar} $(A,-)$ is isotopic to the split bi-octonion algebra.
	
	(ii) If $(A,-) = C_1\otimes C_2$ is decomposable, then $b_6(A,-)$ is the symbol: \[b_6(A,-) = e_6(N_{k\times k / k}((n_1, n_2))) = e_6(n_1\cdot n_2) = e_3(n_1)e_3(n_2)\]
	where $n_i$ is the norm of $C_i$. If $(A,-)= \cor_{E/k}(C)$ where $E = k(\sqrt{d})$ is a field    then $b_6(A,-) = e_6(N_{E/k}(n_C) - 4\llangle d \rrangle)$. Say  $n_C = \llangle c_1, c_2, c_3 \rrangle$. By Theorem~\ref{thm:norm}, $b_6(A,-)$ is either zero or equal to the symbol \[e_6\big(\bigotimes_{i = 1}^3\llangle \tr_{E/k}(c_i), -dN_{E/k}(c_i) \rrangle\big) = \prod_{i=1}^3(\tr_{E/k}(c_i)) {\cdot} (-dN_{E/k}(c_i)). 
	\]
	(iii) 	Suppose $(A,-) = \cor_{E/k}(C)$ where $E$ is a quadratic \'etale $k$-algebra and $n:C \to E$ is the norm of~$C$. If $(A,-)$ is not a division algebra then  $n'$ represents an element $c\in k$ (Corollary~\ref{cor:cor}). If $n'$ represents 0 then $n$ is hyperbolic and $b_6(A,-) = 0$. Otherwise, $n'$ represents some $c \in k^\times$. In this case, $n = \llangle c, c_1, c_2 \rrangle$ for some $c_1, c_2 \in E^\times$ and \begin{align*}b_6(A,-) &= (2c){\cdot} (-dc^2) {\cdot} (\tr_{E/k}(c_1)) {\cdot} (-dN_{E/k}(c_1)){\cdot}(\tr_{E/k}(c_2)){\cdot}(-dN_{E/k}(c_2)) \\
	&= (-1) {\cdot} (2c) {\cdot} (\tr_{E/k}(c_1)) {\cdot} (-dN_{E/k}(c_1)){\cdot}(\tr_{E/k}(c_2)){\cdot}(-dN_{E/k}(c_2))  \end{align*} 
	because $(-dc^2) = (-1) + (d)$ and  $(d) {\cdot}(-dN_{E/k}(c_1)) = (d) {\cdot}(N_{E/k}(c_1\sqrt{d})) = 0$ by \cite[Proposition~III.9.15~(2)]{berhuy}. \end{proof}

\subsection{Central simple associative algebras with orthogonal involution of degree $4$} \label{sec:csa4}

Most of the results of this section apply, with minor modifications, also to $(4,4)$-product algebras. This class of objects is better known as central simple algebras with orthogonal involution of degree 4. It should be no trouble to reproduce the proofs, so we shall only summarise. Let $t$ be the transpose involution on a matrix algebra. The automorphism group of $(M_4(k),t)$ is $\mathbf{PGO}_4 \simeq (\mathbf{PGL}_2 \times \mathbf{PGL}_2) \rtimes S_2$. There are natural-in-$k$ bijections \[H^1(k, \mathbf{PGO}_4) \simeq \mathsf{Prod}_{4,4}(k) \simeq \mathsf{Comp}_4\mathsf{\acute{E}t}_2(k).\] 

One can define three nontrivial mod 2 cohomological invariants of $\mathbf{PGO}_4$. There are various ways to think about the unique nontrivial degree 1 invariant $y_1$: it is the connecting map $\delta: H^1(*, \mathbf{PGO}_4) \to H^1(*,\ZZ/2\ZZ)$, it is the discriminant in the sense of \cite[Definition~7.2]{knus1998book}, it is the map sending $(A,-)$ to the centroid of the 6-dimensional semisimple Lie algebra $\Skew(A,-)$ of type $A_1 \times A_1$, and it is the map sending $(A,-)$ to the centre $Z(Q)$ of the unique (up to $k$-isomorphism) quaternion algebra $Q$ such that $(A,-) = \cor_{Z(Q)}(Q)$.

The next nontrivial invariant $y_2$ is of degree 2, and again there are several ways to think about it: it is the map $(A,\sigma) \mapsto [A] \in {}_2\operatorname{Br}(*) = H^2(*,\ZZ/2\ZZ)$, and it is the map sending $(A,\sigma)$ to $e_2(Q)$ where $Q$ is an Albert form of $(A,\sigma)$. The next nontrivial invariant is of degree~4, and it is entirely analogous to the invariant described in \ref{sec:b6}. We shall describe it here for the record.
Suppose $(A,\sigma)$ is a central simple associative algebra with orthogonal involution of degree~4, and let $\operatorname{Trd}$ be its reduced trace. The bilinear form $b(x,y) = \operatorname{Trd}(x \sigma(y) + y\sigma(x))$ is clearly an invariant symmetric bilinear form on $A$. So is the form $N_{E/k}(n)$ where $n$ is the norm on the unique quaternion algebra $Q$ over a quadratic \'etale extension $E$ such that $(A,\sigma) = \cor_{E/k}(Q)$. Since invariant symmetric forms are unique up to scaling, we have $b = \langle 8 \rangle N_{E/k}(n)$, and by Theorem~\ref{thm:norm} $N_{E/k}(n) -2 n_E$ equals a $4$-Pfister form in $W(k)$. We can define a degree 4 invariant $y_4$ by taking $y_4(A,\sigma) = e_4(q)$ where $q$ is this 4-Pfister form.

\begin{proposition}
	If $(A,\sigma)$ is not a division algebra, then $y_4(A,\sigma) \in (-1){\cdot}H^3(k,\ZZ/2\ZZ)$.
 	\end{proposition}

\begin{proof}
The proof is entirely similar to \ref{thm:symbols}~(iii), using the fact that if $Q$ is a quaternion algebra over a quadratic \'etale extension $E/k$, and $\cor_{E/k}(Q)$ is not a division algebra, then the pure norm of $Q$ is writable as $\llangle a, b \rrangle$ for some $a \in k^\times$ \cite[Corollary~16.28]{knus1998book}. 
\end{proof}

Under the assumption that $\sqrt{-1} \in k$, Rost, Serre, and Tignol showed that any central simple associative algebra $A$ of degree 4 has the property that the quadratic form $x \mapsto \operatorname{Trd}(x^2)$ on $A$ is Witt equivalent to the sum $q_2 \perp q_4$ of a $2$-Pfister and a $4$-Pfister form \cite{rost2006forme}. These summands are unique, so they become cohomological invariants $H^1(*,\mathbf{PGL}_4) \to H^4(*,\ZZ/2\ZZ)$. The form $q_4$ is hyperbolic if and only if $A$ is cyclic, and $q_2$ is hyperbolic if and only if $[A]$ has order 2 in the Brauer group, which is the case if $A$ supports an orthogonal involution. Not surprisingly, the composition \[\begin{tikzcd} H^1(*, \mathbf{PGO}_4) \ar[r] &  H^1(*,\mathbf{PGL}_4) \ar[r,"q_4"] & H^4(*, \ZZ/2\ZZ)\end{tikzcd}\] is the same as the invariant $y_4$ \cite[Exemple]{rost2006forme}.

\part{Classifying cohomological invariants} \label{pt:II}

 Our approach to classifying the invariants of $\mathbf{Spin}_{14}$ goes as follows. We approach $\mathbf{Spin}_{14}$ from the following system of subgroups of the extended Clifford group $\mathbf{\Omega}_{14}$ of the hyperbolic $14$-dimensional quadratic form:
\begin{equation*}
	\begin{tikzcd}
	G_2\times G_2 \ar[r,hook] & (G_2\times G_2)\rtimes S_2 \ar[r,hook] & \mathbf{\Omega}_{14} \ar[r,hookleftarrow] & \mathbf{Spin}_{14}	\end{tikzcd}
\end{equation*}
Since $H^1(k,\mathbf{\Omega}_n) \simeq PI_{n}^3(k)$, this leads to the following diagram of rings: 
\begin{equation*}
\begin{tikzcd}
\Inv(G_2\times G_2)^{S_2}  & \ar[l] \Inv((G_2\times G_2)\rtimes S_2) \ar[r,hookleftarrow]  &  \Inv(PI_{14}^3) \ar[r,hookrightarrow]   & \Inv(\mathbf{Spin}_{14})
\end{tikzcd}
\end{equation*}
The two monomorphisms in this diagram are important: they are explained by the maps  $H^1(*,(G_2\times G_2)\rtimes S_2))\to PI_{14}^3(*)$ and $H^1(*,\mathbf{Spin}_{14}) \to PI_{14}^3(*)$ being surjections. Of all these sets of invariants, $\Inv(G_2 \times G_2)^{S_2}$ is the easiest to determine and it becomes the cornerstone of our analysis. All the $S_2$-invariant invariants of $G_2 \times G_2$ extend to $(G_2 \times G_2)\rtimes S_2$, so the arrow on the left is split surjective. The unique degree 1 invariant of $(G_2 \times G_2) \rtimes S_2$ is clearly in the kernel of this arrow, and we show that it does in fact generate the kernel. This step makes use of a lemma (\ref{lem:gille}) by Philippe Gille on unramified elements in cycle modules. In this way we classify  the contents of  $\Inv((G_2\times G_2)\rtimes S_2)$, and from there it is straightforward to determine which invariants of $(G_2\times G_2)\rtimes S_2$ lie in the image of $\Inv(PI_{14}^3)$. Thereby we have determined $\Inv(PI_{14}^3)$, a crucial subset of $\Inv(\mathbf{Spin}_{14})$. Interestingly, $\Inv(PI_{14}^3)$ is not generally a free module if $-1$ is not a square in $k$.  For the next step, we use the system of subgroups
\begin{equation*}
		\begin{tikzcd} 
		 \mathbf{Spin}_{14}\ar[rr,hook]  && \mathbf{\Omega}_{14} \\
		\mathbf{Spin}_{12} \ar[r,hook] \ar[u,hook] & \mathbf{\Gamma}^+_{12}   \ar[r,hook] & \mathbf{\Omega}_{12}\ar[u,hook]
		\end{tikzcd}
\end{equation*}
from which the solid arrows in the following diagram are obtained:
\begin{equation*}
		\begin{tikzcd} 
		\Inv(PI_{14}^3) \ar[d] \ar[rr,hook]&& \Inv(\mathbf{Spin}_{14}) \ar[d] \ar[r,dashed] & \Inv(PI_{14}^3)  \ar[d] \\
		\Inv(PI_{12}^3) \ar[r,hook] &  \Inv(I_{12}^3) \ar[r,hook] \ar[rr, dashed, bend right] & \Inv(\mathbf{Spin}_{12})  \ar[r, dotted] &    \Inv(PI_{12}^3)	\end{tikzcd}
\end{equation*}
The dashed arrows are ``residue maps" obtained using the concept of a fibration of functors (Definition \ref{sec:fibrations}). It was previously known that $\Inv(I_{12}^3) \simeq \Inv(\mathbf{Spin}_{12})$ if  $\sqrt{-1} \in k$, and actually this is true even without that assumption -- we prove it in \ref{thm:I12-Spin12}. This allows us to complete the rightmost square with the dotted arrow, and it turns out that this square is commutative. A special feature of our fibrations is that the sequences \[\begin{tikzcd} \Inv(PI_{n}^3)\ar[r] & \Inv(\mathbf{Spin}_n) \ar[r] & \Inv(PI_{n}^3) \end{tikzcd}\] derived from them are exact in the middle for $n = 12, 14$. The contents of $\Inv(PI_{14}^{3})$, $\Inv(PI_{12}^3)$, and $\Inv(\mathbf{Spin}_{12})$ are known at this late stage of the process, and it turns out that all of this information is sufficient to determine $\Inv(\mathbf{Spin}_{14})$ using an elementary argument, but only if $-1$ is a square in $k$. Attempts to prove this without assuming $\sqrt{-1} \in k$ failed at the very last step, because in this configuration of residues and restriction maps, there appear many terms involving $(-1)$ and the presence of these terms gives us insufficient control over the right hand side of the diagram above.

\section{Invariants of $\mathbf{PGO}_4$ and  $(G_2 \times G_2) \rtimes S_2$} \label{sec:classification1}

This section aims at classifying the cohomological invariants of bi-octonion algebras. The same method works to classify the invariants of $\mathbf{PGO}_4$, something which does not seem to have been done in the literature, so we state and prove both results in Theorem \ref{thm:inv-bioctonions}.
%Recall the notation $G = G_2$ for the automorphism group of the split octonion algebra.

\subsection{Cohomological invariants of $G_2 \times G_2$} \label{sec:invariantsG2}
	Let $e: H^1(*,G_2) \to H^3(*,\ZZ/2\ZZ)$ be the unique nontrivial normalised cohomological invariant of $G_2$ \cite[Pt.\ 1 \S18.4]{garibaldi2003cohomological}. Define  invariants 
	$	e', e'', r_3, r_6 \in \Inv(G_2 \times G_2)$
	\begin{align*}
	e'(\alpha, \beta) = e(\alpha), &&
	 e''(\alpha, \beta) = e(\beta), 
	&& r_3(\alpha, \beta) = e(\alpha)+e(\beta), &&
	r_6(\alpha,\beta) = e(\alpha){\cdot} e(\beta).
	\end{align*}

Recall from Example~\ref{example:split-case} how $S_2$ acts on $H^1(*,G_2 \times G_2)$. In the obvious way, $S_2$ acts on $\Inv(G_2 \times G_2)$ too.

\begin{lemma} \label{lem:InvG2}
	$\Inv(G_2 \times G_2)$ is the free $H(k)$-module with basis $\{1, e', e'', r_6 \}$, and  $\Inv(G_2\times G_2)^{S_2}$ is the free $H(k)$-module with basis $\{1, r_3, r_6\}$.
\end{lemma}
\begin{proof}
	We use the fact that $\Inv(G_2)$ is a free $H(k)$-module with basis $\{1,e\}$ \cite[Pt.\ 1 Theorem 18.1]{garibaldi2003cohomological}. Applying \cite[Pt.\ 1 Exercise 16.5]{garibaldi2003cohomological},  $\Inv(G_2 \times G_2)$ is a free $H(k)$-module with basis $\{1,e',e'',r_6\}$. Since $S_2$ fixes $1$ and $r_6$ and it swaps $e'$ with $e''$, it follows that $\{1,r_3,r_6\}$ is an $H(k)$-basis for $\Inv(G_2 \times G_2)^{S_2}$.
\end{proof}

The fibres of the map $H^1(*,G_2 \times G_2) \to H^1(*, (G_2 \times G_2) \rtimes S_2)$ are orbits of $H^1(k,G_2 \times G_2)$ by $S_2$; see \ref{lem.fibres}--\ref{example:split-case}. This means there is an $H(k)$-linear map \begin{equation} \label{eq:G2G2sequence}\Inv((G_2 \times G_2)\rtimes S_2) \to \Inv(G_2 \times G_2)^{S_2}\end{equation} defined by restricting invariants to the image of $H^1(*,G_2 \times G_2)$ in $H^1(*,(G_2 \times G_2)\rtimes S_2)$.

\subsection{Cohomological invariants of $\mathbf{PGL}_2 \times \mathbf{PGL}_2$}
	Let $f: H^1(*, \mathbf{PGL}_2) \to H^2(*, \ZZ/2\ZZ)$ be the unique nontrivial normalised cohomological invariant of $\mathbf{PGL}_2$, sending a quaternion algebra to its Brauer class. We can define invariants $r_2, r_4: H^1(*,\mathbf{PGL}_2 \times \mathbf{PGL}_2)$,
	\begin{align*}
	r_2(\alpha, \beta) = f(\alpha)+f(\beta), &&
	r_4(\alpha,\beta) = f(\alpha){\cdot} f(\beta).
	\end{align*}
	As in the case with $G_2$, it is clear that $\Inv(\mathbf{PGL}_2 \times \mathbf{PGL}_2)$ is the free $H(k)$-module with basis $\{1, r_2, r_4 \}$. There is an exceptional isomorphism $\mathbf{PGO}_4 \simeq (\mathbf{PGL}_2 \times \mathbf{PGL}_2) \rtimes S_2$, and a map
	\begin{equation} \label{eq:PGL2PGL2}
	\Inv(\mathbf{PGO}_4) \to \Inv(\mathbf{PGL}_2 \times \mathbf{PGL}_2)^{S_2}
	\end{equation}
	which is entirely analogous to \eqref{eq:G2G2sequence}.

\begin{lemma} \label{lem:surjective}
	The maps \eqref{eq:G2G2sequence} and \eqref{eq:PGL2PGL2} are split surjective.	
\end{lemma}

\begin{proof}
	Recall the definitions of $b_1, b_3, b_6$ from \ref{sec:b1b3} and \ref{sec:b6}, and $s_1, s_2, s_4$ from \ref{sec:csa4}. The first map \eqref{eq:G2G2sequence} sends: \begin{align*} 1 \mapsto 1, &&  b_1 \mapsto 0, && b_3 \mapsto r_3, && b_6 \mapsto r_6.\end{align*}
It is surjective by Proposition \ref{lem:InvG2}, and clearly splits. The second map \eqref{eq:PGL2PGL2} sends
\begin{align*} 1 \mapsto 1, &&  s_1 \mapsto 0, && s_2 \mapsto r_4, && s_4 \mapsto r_4\end{align*}
and is just as obviously split surjective.
\end{proof}

 The next step is to show that in fact the kernels of \eqref{eq:G2G2sequence} and \eqref{eq:PGL2PGL2} are the ideals $H(k){\cdot} b_1$ and $H(k){\cdot} s_1$ respectively.

\subsection{Cycle modules} Cycle modules were introduced by Rost in \cite{rost1996chow}, as a generalisation of, say, the Galois cohomology functor $H(*) = \bigoplus_{i \ge 0}H^i(*,\ZZ/2\ZZ)$ if $\Char(F) \ne 2$, as well as the Milnor $K$-theory $K_*(*)$. A cycle module over a field $F$ is a functor $M = \bigoplus_{i \in \mathbb{Z}} M_i$ from the category of finitely generated field extensions of $F$ to the category of $\mathbb{Z}$-graded abelian groups, equipped with: \begin{enumerate}[(i)] \item  norm homomorphisms $N_{L_1/L_2}: M(L_2) \to M(L_1)$, graded of degree 0,  for every finite extension $L_2/L_1$ of finitely generated fields over $F$,
 	\item residue homomorphisms $\partial_\upsilon: M(L) \to M(E)$, graded of degree $-1$, for every extension $L/F$ having a discrete valuation $\upsilon$ with residue field $E$, such that $\upsilon$ corresponds to the valuation at a codimension 1 point in a normal proper $F$-variety $X$ with $L/F \simeq F(X)/F$.
 	\item a graded left $K_*(L)$-module structure on $M(L)$, for every finitely generated $L/F$.
  	\end{enumerate}
 	In addition, there are a number of axioms which should be satisfied; see \cite[\S1--2]{rost1996chow} or \cite[\S2]{merkurjev2008unramified}.
 	
 	For a cycle module $M$ over $F$, an equidimensional $F$-variety $X$, and an integer $d \ge 0$, we define the group
 	\[
 	A^0(X,M_d) = \ker\Big( \bigoplus_{x \in X^{(0)}}M_d(F(x)) \overset{\partial}{\longrightarrow} \bigoplus_{y \in X^{(1)}} M_{d-1}(F(y)) \Big).
 	\]
 	If $X$ is of dimension $d_X$, this group is $A_{d_X}(X,M,d-d_X)$ with the homological notation of \cite[p.\ 356]{rost1996chow}. Here, $X^{(i)}$ stands for the set of codimension $i$ points of $X$, and $F(x)$ is the residue field associated to $x$. The map $\partial$ is defined as in \cite[p.\ 337]{rost1996chow} or \cite[pp.\ 54--55]{merkurjev2008unramified}: for $x \in X^{(0)}$, the $x,y$-component $\partial_x^y$ is trivial if $y\notin \overline{\{x \}}$ and otherwise it equals $\sum N_{F(\upsilon)/F(x)} \circ \partial_\upsilon$ where $\upsilon$ ranges over the discrete valuations on $F(x)$ corresponding to points lying over $y$ in the normal closure of $\overline{\{x \}}$. If $X$ is normal and irreducible then $X^{(0)}$ has only the generic point $\xi$ and $A^0(X,M_d) = \bigcap_{x \in X^{(1)}} \ker \partial^{\xi}_{x}$	is the subset of $M(F(X))$ which is ``unramified at all irreducible divisors of $X$".

\begin{lemma}[Gille] \label{lem:gille}
		Let $E/F$ be a Galois field extension and consider the torus $T=R_{E/F}( \mathbf{G}_m)^r$ for $r \ge 1$. For each $d \ge 0$, we have
an isomorphism $$\ker( M_d(F) \to M_d(E) ) \overset{\sim}{\longrightarrow} \ker\big( A^0(T,M_d) \to A^0(T_E,M_d) \big).$$   
\end{lemma}

 \begin{proof}
 		The idea is to construct  a partial compactification $T \hookrightarrow U$ where $U$
is open in some affine space $X$ such that $U^{(1)}=X^{(1)}$.
We put $n=[E:F]$ and consider the embedding of $T_0=(\mathbf{G}_m)^{nr}$ in 
$U_0=\mathbf{A}^{nr} \setminus \bigcup\limits_{i <j} \{ x_i=x_j\}$. 
We observe that the closed $F$-subvariety \[Z_0= \bigcup\limits_{i} \{x_i= 0\} \setminus \bigcup_{i <j} \{x_i=x_j\}\]
of $U_0$ 
is isomorphic to $\bigsqcup\limits_{i =1,\dots,  nr} \mathbf{G}_{m}$.
In particular, $Z_0$ is smooth and we have  $U_0 \setminus Z_0=T_0$.

This embedding is $S_{nr}$-equivariant
and $S_n$-equivariant for the diagonal embedding $S_n \to S_{nr}$.
Twisting by the $S_n$-torsor $P=\operatorname{Isom}(F^n,E)$ provides 
open embeddings $T \hookrightarrow U \hookrightarrow \mathbf{A}^{r}(E)$. 
Furthermore $T$ is the complement of $Z={^PZ}_0$. 
We have  $Z_0 = \bigsqcup\limits_{i =1,\dots,  nr} \mathbf{G}_{m}$.  so that 
$Z= \bigsqcup\limits_{l =1,\dots,  r} \mathbf{G}_{m,E}$. In particular, $Z$ is an $E$-variety.
We consider the commutative diagram of long exact sequences (defined in \cite[\S5]{rost1996chow}):
 \[
\xymatrix{
0 \ar[r] & A^0(U, M_d) \ar[r] \ar[d] & A^0(T, M_d)  \ar[r]\ar[d] & A^0(Z, M_{d-1}) \ar[d]\\
0 \ar[r] & A^0(U_E, M_d) \ar[r] & A^0(T_E, M_d)  \ar[r] & A^0(Z_E, M_{d-1}) .
}
\]
 where the vertical maps are pull-backs for $U_E \to U$, $T_E \to T$, $Z_E \to Z$ respectively.
 The point is that the map $Z_E \to Z$ admits a splitting so that 
 the right vertical map is injective. By diagram chase, we get an isomorphism
 $$
 \ker\Big(A^0(U, M_d) \to A^0(U_E, M_d) \Big) \overset{\sim}{\longrightarrow}  
 \ker\Big(A^0(T, M_d) \to A^0(T_E, M_d) \Big) .
 $$
 Since $U$ is an open subset of $\mathbf{A}^{r}(E)$ containing all its points of codimension $1$, we have 
 $A^0(\mathbf{A}^{r}(E), M_d) =   A^0(U, M_d)$ and $A^0(\mathbf{A}^{r}(E)_E, M_d) =   A^0(U_E, M_d)$.
 But $\mathbf{A}^{r}(E)$ is an affine space so that $M_d(F)= A^0(\mathbf{A}^{r}(F), M_d)$
 and $M_d(E)= A^0(\mathbf{A}^{r}(E)_E, M_d)$. By combining those identities we get the desired 
 isomorphism 
$\ker( M_d(F) \to M_d(E) ) \overset{\sim}{\longrightarrow} \ker\big( A^0(T,M_d) \to A^0(T_E,M_d) \big)$.   
 \end{proof}

\begin{lemma} \label{lem:kernels}
	The kernels of the maps  \eqref{eq:G2G2sequence} and \eqref{eq:PGL2PGL2} are the ideals $H(k){\cdot} b_1$ and $H(k){\cdot} s_1$ respectively.
\end{lemma}

\begin{proof}
 	Let $(G,r) = (\mathbf{PGL}_2, 2)$ or $(G_2, 3)$. Fix a quadratic field extension $E/k$. For any field $L/k$ there is a surjection
 	\[
 	   H^1(L\otimes_k E, (\bm{\mu}_2)^r) \simeq H^1(L, R_{E/k}(\bm{\mu}_2)^r) \longrightarrow  H^1(L, R_{E/k}(G)) \simeq H^1(L \otimes_k E, G)
 	\]
 	taking parameters $(c_1,\dots,c_r)$ to the composition algebra over $L\otimes_k E$ with norm $\llangle c_1, \dots, c_r \rrangle$.
 	Precomposition with $H^1(L, R_{E/k}(\bm{\mu}_2)^r) \to  H^1(L, R_{E/k}(G))$ gives an injective map \[\Inv^d(R_{E/k}(G)) \hookrightarrow \Inv^d(R_{E/k}(\bm{\mu}_2)^r)\] for all $d \ge 0$. The torus $T = R_{E/k}(\mathbf{G}_m)^r$ is a classifying variety for $R_{E/k}(\bm{\mu}_2)^r$ in the sense of \cite[Pt.\ 2 \S3]{garibaldi2003cohomological}, so there is an injective map \[\Inv^d(R_{E/k}(\bm{\mu}_2)^r) \hookrightarrow H^d(k(T), \ZZ/2\ZZ),\] namely evaluation at the generic $R_{E/k}(\bm{\mu}_2)^3$-torsor $\xi$ over $k(T)$. Invariants of $R_{E/k}(\bm{\mu}_2)^r$ are unramified at every codimension 1 point of $T$ \cite[Pt.\ 1 Theorem~11.7]{garibaldi2003cohomological}, which means their values at the generic torsor are contained in $A^0(T, H^d(*,\ZZ/2\ZZ))$. 	In summary, we have an injective map
 	\begin{align*}
 	 \Inv^d(R_{E/k}(G)) &\hookrightarrow  A^0(T,H^d(*,\ZZ/2\ZZ)) \subset H^d(k(T), \ZZ/2\ZZ). \\
 	 b &\mapsto b(\xi)
 	\end{align*}
% 	 Now suppose $b \in \Inv^d((G_2 \times G_2) \rtimes S_2)$ is in the kernel of  the map \eqref{eq:G2G2sequence}, meaning that for all field extensions $L/k$ and all $x$ in the image of $H^1(L,G_2 \times G_2) \to H^1(L, (G_2 \times G_2) \rtimes S_2)$, we have $b(x) = 0$. Chasing $b$ around the diagram, 
	Note that $R_{E/k}(G)_E = G^2_E$ and consider the commutative diagram where the first two vertical arrows are just precomposition with the forgetful functor $\mathsf{Fields}_{/E} \to \mathsf{Fields}_{/k}$, and the horizontal arrows on the left come from restricting invariants to the subfunctor $H^1(*,R_{E/k}(G))/S_2$ of $H^1(*,G^2\rtimes S_2)$, as in \ref{sec:partitioning}:
 	 \[
 	 \begin{tikzcd}
 	 \Inv^d\big(G^2\rtimes S_2\big) \ar[r] \ar[d] & \Inv^d\big(R_{E/k}(G)\big)^{S_2} \ar[r, hook] \ar[d] & A^0(T, H^d(*,\ZZ/2\ZZ)) \ar[d] \\
 	 \Inv^d\big((G^2\rtimes S_2)_E\big) \ar[r] & \Inv^d\big(G^2_E\big)^{S_2} \ar[r, hook] & A^0(T_E, H^d(*,\ZZ/2\ZZ)).
 	 \end{tikzcd}
 	 \]
 	 If $b \in \ker(\Inv^d(G^2 \rtimes S_2) \to \Inv^d(G^2)^{S_2})$, this means that for all field extensions $L/k$ and all $x$ in the image of $H^1(L,G^2) \to H^1(L, G^2 \rtimes S_2)$, we have $b(x) = 0$. In particular, $b$ is mapped to 0 in $\Inv^d(G^2_E)^{S_2}$ and hence $b(\xi) \in A^0(T, H^d(*,\ZZ/2\ZZ))$ is in the kernel of $A^0(T, H^d(*,\ZZ/2\ZZ)) \to A^0(T_E, H^d(*,\ZZ/2\ZZ))$. By Lemma \ref{lem:gille}, $b(\xi)$ is a constant from $H^d(k, \ZZ/2\ZZ)$. This in turn implies that $b$ is constant on the subset $H^1(L, R_{E/k}(G))/{S_2} \subset  H^1(L, G^2 \rtimes S_2)$, for all fields $L/k$ (and it is constantly zero on $H^1(L, G^2)/S_2$ by assumption). Equivalently, $b$ factors through the unique nontrivial degree 1 invariant $H^1(*, G^2 \rtimes S_2) \to H^1(*, \ZZ/2\ZZ)$ and so there is a unique $h \in H^{d-1}(k,\ZZ/2\ZZ)$ such that $b = h{\cdot} b_1 $ if $G = G_2$, or $b  = h{\cdot} s_1 $ if $G = \mathbf{PGL}_2$.
\end{proof}

\begin{theorem} \label{thm:inv-bioctonions} \begin{enumerate}[\rm (i)] \item[]
\item $\Inv(\mathbf{PGO}_4)$ is a free $H(k)$-module with basis $\{1, s_1, s_2, s_4\}$.
\item
$\Inv((G_2 \times G_2) \rtimes S_2)$ is a free $H(k)$-module with basis $\{1, b_1, b_3, b_6\}$. 
\end{enumerate}
\end{theorem}

\begin{proof}
	We have shown in Lemmas \ref{lem:surjective} and  \ref{lem:kernels} that the sequence
	\[
\begin{tikzcd} 0 \ar[r] &  \Inv(S_2) \ar[r] &  \Inv((G_2 \times G_2) \rtimes S_2) \ar[r] &  \Inv(G_2 \times G_2)^{S_2} \ar[r] & 0  \end{tikzcd}
	\]
	is split exact, so $\Inv((G_2 \times G_2)\rtimes S_2) \simeq \Inv(S_2) \oplus \Inv(G_2 \times G_2)^{S_2}$, and keeping track of the isomorphism shows that $\{1, b_1, b_3, b_6\}$ is the basis for $\Inv((G_2 \times G_2) \rtimes S_2)$ that we get from this splitting. The proof of (i) is identical with $G_2$ replaced by $\mathbf{PGL}_2$, and it yields the basis $\{1, s_1, s_2, s_4\}$.
\end{proof}

\section{Invariants of $I_{12}^3$ and $\mathbf{Spin}_{12}$} \label{sec:Spin12}

Garibaldi in \cite[\S20]{garibaldi2009cohomological}  showed that $\Inv(I_{12}^3) \simeq \mathbf{Spin}_{12}$ is a free $H(k)$-module with generators in degrees 3, 5, 6, under the rather strong assumption that $\sqrt{-1} \in k$. We shall revisit the classification to remove that assumption, in the process setting up a number of techniques and concepts that prove to be very useful in the next section.

\subsection{The functors $I_{n}^2$, $I_{n}^{3}$, and $PI_{n}^3$ as Galois cohomology sets.} \label{sec:cohomological-In}
Let $(V,q)$ be a quadratic space.
The Galois cohomology of the exact sequence \begin{equation} \label{diag:spin} \begin{tikzcd} 1 \ar[r] & \bm{\mu}_2 \ar[r] & \mathbf{Spin}(V,q)  \ar[r] &  \mathbf{O}^+(V,q) \ar[r] & 1 \end{tikzcd}\end{equation} is well-understood: see \cite[IV.11.2]{berhuy}, \cite[p.~437]{knus1998book}, or \cite[16.2]{garibaldi2009cohomological}. The set $H^1(k, \mathbf{O}^+(V,q))$ injects into $H^1(k,\mathbf{O}(V,q))$ and so it corresponds bijectively with the set of isometry classes of quadratic forms with the same dimension and discriminant as $(V,q)$.
The image of $H^1(k, \mathbf{Spin}(V,q)) \to H^1(k, \mathbf{O}^+(V,q))$ is the set of isometry classes of quadratic forms with the same dimension, discriminant, and Clifford algebra as $(V,q)$. By appending \eqref{diag:spin} above the top row of \eqref{diag:omega-exact}, one demonstrates as in \cite[pp.~461--462]{chernousov2014essential} that there is  a one-to-one correspondence:

 \begin{center}
\fbox{\begin{minipage}[c][1cm]{3.5cm} \centering
$H^1\big(k,\mathbf{\Gamma}^+(V,q)\big)$
\end{minipage}}\quad  $\longleftrightarrow$\quad\fbox{\begin{minipage}[c][1cm]{10cm} \centering
Isometry classes of quadratic forms with the same dimension, discriminant, and Clifford algebra as $(V,q)$.
\end{minipage}}
\end{center}

Assume now that $V$ is even-dimensional. The extended Clifford group $\mathbf{\Omega}(V,q)$ is connected and reductive. Let $Z = Z(C^+(V,q))$. We can show by chasing the diagram \eqref{diag:omega-exact}, or referring to \cite[Lemma 13.20]{knus1998book} and using the smoothness of $\mathbf{\Omega}^+(V,q)$, that $\mathbf{\Omega}^+(V,q)$ is generated by its centre $\mathbf{GL}_1(Z)$ and the subgroup $\mathbf{\Gamma}^+(V,q)$.  Hence $\mathbf{\Gamma}^+(V,q)$ is a normal subgroup of $\mathbf{\Omega}(V,q)$. We extend \eqref{diag:omega-exact} to make the first two columns into exact sequences.
	\begin{equation*} 		\begin{tikzcd}
		& 1 \ar[d] & 1 \ar[d] \\
		1 \arrow[r] &  \mathbf{G}_m \arrow[r] \ar[d] & \mathbf{\Gamma^+}(V,q) \arrow[r,"\chi"] \ar[d] & \mathbf{O}^+(V,q) \arrow[r] \ar[d] & 1\\
		1 \arrow[r] &  \mathbf{GL}_1(Z) \arrow[r] \ar[d, "z \mapsto z^{-2}N_{Z/k}(z)" left] & \mathbf{\Omega}(V,q) \ar[d] \arrow[r,"\chi'"] & \mathbf{PGO}^+(V,q) \arrow[r] & 1
		\\
		 & \mathbf{G}_{m,Z}^1 \ar[d] \ar[r] & T \ar[d] \\
		 & 1 & 1
		\end{tikzcd}
	\end{equation*}
	Here,  $\mathbf{G}_{m,Z}^1= \ker(N_{Z/k}: \mathbf{GL}_1(Z) \to \mathbf{G}_m) \simeq \mathbf{GL}_1(Z)/\mathbf{G}_m$, and $T = \mathbf{\Omega}(V,q)/\mathbf{\Gamma}^+(V,q)$. We have from \cite[Corollary~13.16]{knus1998book} the fact that $Z^\times \cap \Gamma^+(V,q) = \{z \in Z^\times \mid z^2 \in k^\times\}$. Hence $\mathbf{GL}_1(Z) \cap \mathbf{\Gamma}^+(V,q)$ is the kernel of the map $\mathbf{GL}_1(Z) \to \mathbf{G}_{m,Z}^1$, $z \mapsto z^{-2}N_{Z/k}(z)$. By the isomorphism theorem \cite[Theorem~5.52]{milne},
	\[T  = \frac{\mathbf{\Gamma}^+(V,q).\mathbf{GL}_1(Z)}{ \mathbf{\Gamma}^+(V,q)} \simeq \frac{\mathbf{GL}_1(Z)}{\mathbf{GL}_1(Z) \cap \mathbf{\Gamma}^+(V,q)} \simeq \mathbf{G}_{m,Z}^1.\]
	Standard arguments, as in \cite[p.~416]{knus1998book}, show that $H^1(k, \mathbf{G}_{m,Z}^1) \simeq k^\times/N_{Z/k}(Z^\times)$. If  $N_{Z/k}$ is surjective then $H^1(k, \mathbf{G}_{m,Z}^1) = 1$ and $H^1(k, \mathbf{\Gamma}^+(V,q)) \to H^1(k,\mathbf{\Omega}(V,q))$ is surjective. The fibre in $H^1(k, \mathbf{\Gamma}^+(V,q))$ over some $\omega \in H^1(k,\mathbf{\Omega}(V,q))$ is   the set of all $\gamma \in H^1(k,\mathbf{\Gamma}^+(V,q))$ such that $\chi_*(\gamma)$ maps to $\chi'_*(\omega)$ in $H^1(k, \mathbf{PGO}^+(V,q))$. This means that the fibres of $H^1(k,\mathbf{\Gamma}^+(V,q))\to H^1(k,\mathbf{\Omega}(V,q))$ are   similitude classes.
		In summary, for any even-dimensional quadratic space $(V,q)$ such that $N_{Z/k}$ is surjective, we have a one-to-one correspondence
		 \begin{center}
\fbox{\begin{minipage}[c][1cm]{3.5cm} \centering
$H^1\big(k,\mathbf{\Omega}(V,q)\big)$
\end{minipage}}\quad  $\longleftrightarrow$\quad\fbox{\begin{minipage}[c][1cm]{10cm} \centering
Similitude classes of quadratic forms with the same dimension, discriminant, and Clifford algebra as $(V,q)$.
\end{minipage}}
\end{center}

If $q$ is the hyperbolic $n$-dimensional form, we adopt the notations $\mathbf{O}_{n}^+ = \mathbf{O}^+(V,q)$, $\mathbf{\Gamma}_{n}^+ = \mathbf{\Gamma}^+(V,q)$, $\mathbf{\Omega}_{n} = \mathbf{\Omega}(V,q)$, etc.  From the discussion above and in \ref{sec:quadratic-forms} and \ref{sec:gal-com-GO}, we summarise that for all even $n > 0$ there are isomorphisms of functors:
\begin{align*}
H^1(*, \mathbf{O}_n) &\simeq I_{n}(*), & H^1(*, \mathbf{O}^+_n) &\simeq I_{n}^2(*), &  H^1(*, \mathbf{\Gamma}^+_{n}) &\simeq {I}^3_n(*),\\
H^1(*, \mathbf{GO}_n) &\simeq PI_n(*) & H^1(*, \mathbf{GO}^+_n) &\simeq PI_{n}^2(*), & H^1(*,\mathbf{\Omega}_{n}) &\simeq {PI}^3_n(*).
\end{align*}
In Serre's notation from \cite{garibaldi2003cohomological}, $I_n  = \operatorname{Quad}_n$ and $I_{n}^2  = \operatorname{Quad}_{n,(-1)^{n/2}} $.

\subsection{The ideals $J_m(k)$}
Let us define the chain of ideals $J_{1}(k)  \subset J_{2}(k) \subset \dots \subset H(k)$ by 
\begin{align*}
	J_{1}(k) = \{ \lambda \in H(k) \mid h{\cdot}(-1) = 0 \}, &&
	J_{2}(k) = \{ \lambda \in H(k) \mid h{\cdot}(-1){\cdot}(-1) = 0\}, &&
	 \dots
\end{align*}
It is clear that $(-1){\cdot} J_{m+1}(k) \subset J_{m}(k)$ for all $m \ge 1$.
It is a fact that $J_{m}(k) = H(k)$ if and only if  $k$ has length $\le 2^{m-1}$ (in the sense of \cite[\S XI.2]{lam}); on the other extreme, if $k$ is a real-closed field then $J_{m}(k) = \{0\}$ for all $m \ge 1$.
\subsection{Cohomological invariants of $I_{n}^2$}
In \cite[Pt.\ 1]{garibaldi2003cohomological}, Serre classified not only the cohomological invariants of $\mathbf{O}_n$, but also the cohomological invariants of $\mathbf{O}_n^+$. The latter group for $n = 2$ mod $4$ was probably the first known instance of an algebraic group whose ring of mod~2 cohomological invariants is not always a free $H(k)$-module.

Let $L/k$ be any field extension. Given a quadratic form $q = \langle a_1, \dots, a_n \rangle$, $a_i \in L^\times$, the Stiefel--Whitney classes $w_i(q) \in H^i(L, \ZZ/2\ZZ)$ are defined as follows and are independent of the choice of diagonalisation of $q$:
\begin{align*}
	w_0 &= 1, \\
	w_1 &= \sum_i (a_i) = (a_1a_2\dots a_n), \\
	w_2 &= \sum_{i < j} (a_1){\cdot}(a_j), \\
	&\vdots \\
	w_n &= (a_1){\cdot}(a_2){\cdot} \dots {\cdot} (a_n),\\
	w_i &= 0 \text{ for all } i > n.	
\end{align*}
A property of the Stiefel--Whitney classes is that \begin{align} \label{eq:total-sw}w_i(q \perp q') = \sum_{j = 0}^i w_j(q) w_{i-j}(q'); \end{align} see \cite[Pt.~1 \S17.2]{garibaldi2003cohomological} or \cite[(5.7)]{elman2008algebraic}.
Given $h \in J_{1}(k)$ and $q = \langle a_1, \dots, a_n \rangle \in I^2_n(L)$, define \[b^h(q) = h{\cdot}(a_1){\cdot}(a_2){\cdot} \cdots {\cdot} (a_{n-1}). \]
Serre showed in \cite[Proposition~20.1]{garibaldi2003cohomological} that if $n > 2$ and $h \in J_{1}(k)$, the element $b^{h}(q)$ does not depend on the choice of diagonalisation of $q$. Moreover, $h \mapsto b^{h}$ is an injective map $J_{1}(k) \to \Inv(I_{2}^n)$.

\begin{theorem}[Serre {\cite[Pt.\ 1 Theorems~17.3, 20.6]{garibaldi2003cohomological}}] \label{thm:serre}
	\begin{enumerate}[\rm (i)]
		\item[]
		\item $\Inv(I_n)$ is a free $H(k)$-module with basis $\{1, w_1, \dots, w_n\}$.
		\item If $n = 0$ mod $4$,  $\Inv(I_n^2)$ is a free $H(k)$-module with basis $\{1, w_2, \dots, w_{n-2}, b^1 \}$.
		\item If $n = 2$ mod $4$ and $n \ge 4$,  $\Inv(I_n^2)$ is a direct sum of the free $H(k)$-module with basis $\{1, w_2, \dots, w_{n-2}\}$ and the $H(k)$-module $\{b^\lambda \mid \lambda \in J_{1}(k)\} \simeq J_{1}(k)$.
	\end{enumerate}
\end{theorem}

\subsection{The exterior square} \label{sec:exterior-square}Let $(V,q)$ be an $n$-dimensional quadratic space over $k$, with $n \ge 2$. The exterior square of $(V,q)$ is the $\binom{n}{2}$-dimensional $k$-quadratic space $(\Lambda^2 V, \lambda^2 q)$ defined by $\lambda^2 q(v_1 \wedge v_2) = \det(q(v_i,v_j))$ for all $v_1, v_2 \in V$.  Clearly, if $q = \langle a_1, \dots, a_n \rangle$ is a diagonalisation of $q$, then 
$\lambda^2(q) = \bigperp_{1 \le i < j \le n} \langle a_i a_j \rangle$. It is clear that for all $c \in k^\times$,
\begin{equation} \label{eq:lambda2-sim}
	\lambda^2(\langle c \rangle q) = \lambda^2(q).	
\end{equation}
The $\lambda^2$-operation extends to a unique map $\lambda^2: \widehat{W}(k) \to \widehat{W}(k)$ such that $\lambda^2([q]) = [\lambda^2(q)]$, and the following equation holds for all $x, y \in \widehat{W}(k)$ \cite[(19.5)]{garibaldi2009cohomological}:
\[
	\lambda^2(x-y) = \lambda^2(x) - xy + \dim y + \lambda^2(y).
\]

\subsection{Degree $2n$ invariants of $I^n$} \label{sec:P_n}
 There is a canonical  homomorphism \begin{align*}
 		I(k) \to \widehat{W}(k), && q \mapsto \hat q =  q - \tfrac{\dim q}{2} \mathbb{H} \end{align*}
 		whose image we denote by
 		$
 		\widehat{I}(k) = \{ \hat q \mid q \in I(k) \} \subset \widehat{W}(k).
 		$
The projection $\pi: \widehat{W}(k) \to W(k)$ restricts to an isomorphism $\pi|_{\widehat{I}(k)}: \widehat{I}(k) \overset{\sim}{\to} I(k)$.
 				For $n \ge 1$, there is a map $P_n: I(k) \to W(k)$ defined as follows:
		\begin{equation*} 
		P_n(q) = \pi(\lambda^2(\hat q)) - 2^{n-1}q.
		\end{equation*}
		By \cite[Example~19.7]{garibaldi2009cohomological}, the following equation holds in $W(k)$, for all even-dimensional quadratic forms $q$:
		\begin{equation} \label{eq:actual-Pn}
		P_n(q) = \frac{\dim q}{2} + \lambda^2(q) - 2^{n-1}q.
		\end{equation}
		
		The maps $P_n$ are  neither additive nor multiplicative, but they have the following important properties \cite[p.~57]{garibaldi2009cohomological}: 
		\begin{align}
		&P_n(\llangle a_1, \dots, a_n \rrangle) = 0 & \text{for all } a_1, \dots, a_n \in k^\times, \label{eq:Pn-pfister}\\
		&P_n(\langle c \rangle q) = P_n(q) + 2^{n-1} \llangle c \rrangle q &\text{for all } c \in k^\times, q \in I(k), \label{eq:Pn-sim}\\
		&P_n(x + y) = P_n(x) + xy + P_n(y) & \text{for all } x, y \in I(k), \label{eq:Pn-x+y}	   %P_n(q + I^{n+1}(k)) & \subset P_n(q) + I^{2n+1}(k) &\text{for all } q \in I(k).
		\end{align}
		Using \eqref{eq:Pn-pfister}--\eqref{eq:Pn-x+y}, it is  easy to show that if $\phi_i$ are $n$-Pfister forms and $c_i \in k^\times$, 
		\begin{equation} \label{eq:big-Pn}
		P_n \big(\sum_i \langle c_i \rangle \phi_i \big) = \sum_{i < j} \langle c_i c_j \rangle \phi_i \phi_j +2^{n-1}\sum_i \llangle c_i \rrangle \phi_i.
		\end{equation}
		This gives a concrete way of expressing $P_n(q)$ for any $q \in I^n(k)$, and it also  implies that
		\begin{equation}
		 \label{eq:In-I2n}		P_n(I^n(k)) \subset I^{2n}(k).
		 \end{equation}

For all $n$, the  composition of $P_n: I^n \to I^{2n}$ with $e_{2n}: I^{2n} \to H^{2n}(*, \ZZ/2\ZZ)$ is a cohomological invariant.

\subsection{Known invariants of $I_{12}^3$} 
There are several nontrivial invariants in $\Inv(I_{12}^3)$ whose existence is established in \cite[\S22.3]{garibaldi2009cohomological}. The first nontrivial invariant is \[z_3(q) = e_3(q) \in H^3(k, \ZZ/2\ZZ).\] This is just the restriction of the Arason invariant 
$e_3: I^3 \to H^3(*, \ZZ/2\ZZ)$  to $12$-dimensional forms. The next nontrivial invariant is defined as
\[
	z_5(q) = e_5(\llangle c\rrangle  P_2(r)) = (c){\cdot} e_4(P_2(r)) \in H^5(k,\ZZ/2\ZZ),
\]
where $q = \llangle c \rrangle r$ is a factorisation of $q$ such that $c \in k^\times$ and $r \in I_{6}^2(k)$. This makes sense because Pfister's Theorem (see \ref{sec:I12}) shows that such a factorisation always exists, because \eqref{eq:In-I2n} shows that $P_2(r) \in I^4(k)$, and because \cite[Corollary~20.7]{garibaldi2009cohomological} shows that the class of $\llangle c \rrangle P_2(r) \in I^5(k)$ does not depend on the way of factorising $q$. 
Using the  parameterisation \eqref{eq:q in I12} of quadratic forms in $I^{3}_{12}(k)$, we can write down the values taken by these invariants:

\begin{lemma} \label{lem:I12-invariants}
	If $q =  \langle d \rangle \llangle c \rrangle (\psi_1'  \perp \langle -1 \rangle \psi_2') \in I_{12}^3(k)$ where  $\psi_i = \llangle x_i, y_i \rrangle$, $c, d, x_i, y_i \in k^\times$:
	\begin{align*}
		z_3(q) &=   (c){\cdot}e_2(\psi_1) + (c) {\cdot} e_2(\psi_2) \\&= (c){\cdot} (x_1) {\cdot} (y_1) + (c) {\cdot} (x_2) {\cdot}(y_2).\notag \\
			z_5(q) &=  (c){\cdot} e_2(\psi_1){\cdot}e_2(\psi_2) + (-1){\cdot}(c) {\cdot}(d) {\cdot} e_2(\psi_1) + (-1) {\cdot} (c) {\cdot} (-d) {\cdot} e_2(\psi_2) \\
			&= (c) {\cdot} (x_1) {\cdot}(y_1) {\cdot}(x_2){\cdot}(y_2) + (-1){\cdot}(c){\cdot} (d) {\cdot}(x_1){\cdot} (y_1) + (-1) {\cdot} (c){\cdot} (-d) {\cdot}(x_2){\cdot} (y_2). \notag
	\end{align*}
\end{lemma}

\begin{proof}
The form $q$ is Witt equivalent to $q \perp \mathbb{H} = \langle d \rangle (\llangle c \rrangle \psi_1 \perp \langle -1 \rangle\llangle c \rrangle  \psi_2)$. The invariant $z_3$ is the restriction of the Arason invariant, so $z_3(q) = e_3(q) = e_3(\llangle c \rrangle \psi_1) + e_3(\llangle c \rrangle \psi_2)$ and  the first formula is clear. The invariant $z_5$ is defined such that $z_5(q) = (c){\cdot} e_4(P_2(\langle d \rangle \psi_1 + \langle -d \rangle \psi_2))$ and one can derive the second formula using either \eqref{eq:big-Pn} or \cite[Example~20.9]{garibaldi2009cohomological}.
\end{proof}

\begin{lemma}
	If $q \in I^{3}_{12}(k)$ is isotropic, $z_5(q)$ is a symbol in $(-1){\cdot}H^4(k, \ZZ/2\ZZ)$.
\end{lemma}

\begin{proof}
	If $q \in I^3_{12}(k)$ is isotropic, then it is isometric to some $p \perp \mathbb{H}$ where $p \in I^3_{10}(k)$. But, as it is well-known \cite[Theorem~2.1]{hoffmann1998}, every form in $I^3_{10}(k)$ is isotropic and similar to a $3$-Pfister form, so we can write $q = \langle d \rangle \llangle c \rrangle(\llangle x, y \rrangle ' \perp \langle 1,-1,1 \rangle)$ for some $d, x, y, z \in k^\times$. Hence by Lemma \ref{lem:I12-invariants}, $z_5(q) = (-1){\cdot}(c){\cdot}(d){\cdot}(x){\cdot}(y)$.
\end{proof}

Given the lemma above, $h{\cdot}z_5$ vanishes on isotropic forms for all $h \in J_{1}(k)$ so we can use Rost's technique \cite[Proposition~10.2]{garibaldi2009cohomological} to define a set of invariants $\{z^h \mid h \in J_{1}(k) \}$ by 
\[
z^h(q) = h {\cdot} (q(v)) {\cdot} z_5(q)
\]
where $v \in k^{12}$ is any anisotropic vector for $q$. If $q$ is as in Lemma \ref{lem:I12-invariants}, then $q$ represents $-dx_1$, but $(-dx_1){\cdot} (x_1) = (d){\cdot}(x_1)$, hence for all $h \in J_{1}(k)$
\begin{equation} \label{eq:zh}
z^h(q) = h {\cdot} (-dx_1) {\cdot} z_5(q) = h {\cdot} (d) {\cdot} (c) {\cdot} (x_1) {\cdot}(y_1) {\cdot}(x_2){\cdot}(y_2).
\end{equation}

If $J_1(k) = H(k)$ then $z^h = h{\cdot}z^1$ for all $h \in H(k)$. We write $z_6 = z^1$; this is the invariant from \cite[\S20.13]{garibaldi2009cohomological}.

\subsection{Fibrations} \label{sec:fibrations}
Fibrations have been used extensively for determining the essential dimensions of certain algebraic groups, including groups like $\mathbf{Spin}_n$ and $\mathbf{\Gamma}^+_n$ \cite{berhuy2003essential, brosnan2010essential, chernousov2014essential}. The concept turns out to  extremely useful for cohomological invariants.

%In most of the literature, for example \cite{garibaldi2009cohomological, garibaldi2003cohomological, macdonald2008cohomological}, the chief technique for classifying cohomological invariants is to use surjections in Galois cohomology. For example, if $A\subset B$ are algebraic groups such that $H^1(L,A) \to H^1(L,B)$ is a surjection for all fields $L/k$, and if we know $\Inv(A)$, then we can use that knowledge to determine $\Inv(B)$ because it embeds into $\Inv(A)$. However, if we know $\Inv(B)$ then it is less obvious how to use that knowledge to determine $\Inv(A)$. This is exactly the situation where fibrations become important.

\begin{definition*} \cite[Definition 1.12]{berhuy2003essential} Let $F: \mathsf{Fields}_{/k} \to \mathsf{Groups}$ be a functor. For functors $A, B: \mathsf{Fields}_{/k} \to \mathsf{Sets}$, a \emph{fibration} from $A$ to $B$ is a morphism $\pi: A \to B$ and an action of $F$ on $A$ such that for all fields $L/k$:
\begin{enumerate}[(i)]
\item   $\pi_L: A(L) \to B(L)$ is surjective,
\item  $\pi_L$ is $F(L)$-equivariant with respect to the trivial action of $F(L)$  on $B(L)$, and
\item  $F(L)$ acts transitively on each fibre of $\pi_L$.
\end{enumerate}
The meaning of (ii) and (iii) is that the orbits of $F(L)$ are precisely the fibres of $\pi_L$; this is why it is called a fibration. It is standard to denote a fibration by 
\[
\begin{tikzcd}F \ar[r,squiggly]	& A \ar[r,two heads,"\pi"] & B. \end{tikzcd}
\]
\end{definition*}

As a first example, for all $m, n \ge 0$ there is an obvious fibration
\begin{equation}  \label{eq:first-fibration}
\begin{tikzcd}[row sep=small]
H^1(*,\bm{\mu}_2) \ar[r,squiggly]	& I^m_{n} \ar[r,two heads] & PI^m_{n} \end{tikzcd}
\end{equation}
where $H^1(L,\bm{\mu}_2) = L^\times/L^{\times 2}$ acts on $I^m_{n}$ by $cL^{\times 2} \cdot q = \langle c \rangle q$.

\begin{lemma} \label{lem:fibration-from-exact-sequence}
An exact sequence of algebraic groups
\[
\begin{tikzcd}
	1 \ar[r] & A \ar[r] & X \ar[r] & Y \ar[r] & 1,
\end{tikzcd}
\]
where $A$ is an abelian central subgroup of $X$, yields a fibration \begin{equation*} 
\begin{tikzcd}H^1(*,A) \ar[r,squiggly]	& H^1(*,X) \ar[r,two heads,"\pi"] & B \end{tikzcd}
\end{equation*}
where $B$ is the subfunctor of $H^1(*,Y)$ such that $B(L)$ is the image of $H^1(L,X)$ in $H^1(L,Y)$.
\end{lemma}

\begin{proof}
The action of the group $H^1(*,A)$ on $H^1(*,X)$ is just pointwise multiplication on the level of cocycles, and the fact that this is a fibration is just \cite[I.\S5~Proposition~42]{serre1997galois}.
\end{proof}

The following proposition   resembles \cite[Propositions~7.1~\&~7.4]{garrel2020witt}, but we are working more generally here.

\begin{proposition} \label{prop:fibration}
Let 	$\begin{tikzcd} H^1(*,\bm{\mu}_{2m}) \ar[r,squiggly]	& A \ar[r,two heads,"\pi"] & B \end{tikzcd}$ be a fibration.
\begin{enumerate}[\rm (i)]
\item For each $a \in \Inv(A)$, there is a unique invariant $\bar{\partial}a$ in the image of  $\pi^*: \Inv(B) \to \Inv(A)$ such that
\[
a(c\cdot x) - a(x) = (c)\cdot \bar{\partial}a(x)
\]
for all field extensions $L/k$, $c \in L^\times/L^{\times 2m}$, and $x \in A(L)$.
\item Let $\partial a \in \Inv(B)$ be the unique invariant such that $\pi^*(\partial a) = \bar{\partial} a$. Then $\partial: \Inv(A) \to \Inv(B)$ is an $H(k)$-module homomorphism, homogeneous of degree $-1$, and
 the following sequence of $H(k)$-modules is exact:
 \begin{equation*} \label{eq:exactness-fibration}
\begin{tikzcd}
0 \ar[r] & \Inv(B) \ar[r,"\pi^*"] & \Inv(A) \ar[r,"\partial"] & \Inv(B).
\end{tikzcd}
\end{equation*}
\item Suppose $\begin{tikzcd} H^1(*,\bm{\mu}_{2m}) \ar[r,squiggly]	& A' \ar[r,two heads,"\pi'"] & B' \end{tikzcd}$ is another fibration and there are morphisms $f: A \to A'$ and $g: B\to B'$ such that $\pi' \circ f = g \circ \pi$ and $f: A(L) \to A'(L)$ is $L^\times/L^{\times 2m}$-equivariant for all fields $L/k$. Then the following diagram commutes:
\[	\begin{tikzcd}
\Inv(A') \ar[r,"\partial"] \ar[d, "f^*"] & \Inv(B') \ar[d, "g^*"]\\
\Inv(A) \ar[r, "\partial"]  & \Inv(B)
\end{tikzcd}\]
\end{enumerate}
\end{proposition}

We shall call $\partial$ the \emph{residue map} with respect to the fibration, and we say that $\partial a \in \Inv(B)$ is the residue of $a \in \Inv(A)$. 

\begin{proof}
	(i) Let $a \in \Inv(A)$. For $x \in A(k)$ we have a normalised invariant in $\Inv(\bm{\mu}_{2m})$:	\[
cL^{\times 2m} \mapsto a(cL^{\times 2m}\cdot x) - a(x).
\]
The set of normalised mod 2 cohomological invariants of $\bm{\mu}_{2m}$  is a rank one free $H(k)$-module generated by the invariant $s$ that sends $cL^{\times 2m}\in H^1(L,\bm{\mu}_{2m})$ to $(c) \in H(L)$ \cite[Proposition~2.6]{garibaldi2009cohomological}.  Hence there is a unique element $\bar \partial a(x) \in H(L)$ such that\[
a(cL^{\times 2m} \cdot x) - a(x) = (c)\cdot \bar \partial a(x)
\]
for all field extensions $L/k$, $x \in A(L)$, and $c \in L^\times$. Clearly $\bar \partial a: A \to H$ is a cohomological invariant. Now we claim that $\bar \partial a$ comes from $\Inv(B)$; that is, if $\pi(x) = \pi(x')$ then $\bar \partial a (x) = \bar \partial a(x')$. Since $L^\times/L^{\times 2m}$ acts transitively on the fibres of $\pi$, we have $x' = rL^{\times 2m} \cdot x$ for some $r \in L^\times$. Then
\begin{align*}
(c){\cdot} \bar \partial a (x') &= a(cL^{\times 2m} \cdot x') - a(x') \\ &= a(crL^{\times 2m} \cdot x) - a(rL^{\times 2m} \cdot x) =(cr)\cdot \bar \partial a (x) - (r)\cdot \bar{\partial} a (x) = (c){\cdot} \bar \partial a (x). 
\end{align*}
Therefore $\bar \partial a (x) = \bar \partial a( x')$ by uniqueness.

(ii) It is clear that that $\partial$ is $H(k)$-linear and that  $a \in \Inv^i(A)$ implies $\partial a \in \Inv^{i-1}(B)$. The sequence is exact at $\Inv(B)$ simply because $\pi$ is surjective. For exactness  at $\Inv(A)$, it is clear that $\partial a = 0$ if and only if $a$ is constant on the fibres of~$\pi: A(L) \to B(L)$ for all fields $L/k$, which means $a$ is the image of an invariant in $\Inv(B)$.

(iii) It suffices to show that $\bar \partial \circ f^* = f^* \circ \bar \partial$. Suppose $a = f^*(a') = a'\circ f \in \Inv(A)$. Then for all fields $L/k$, $x \in A(L)$, and $c \in L^\times$:
\begin{align*}
 (c){\cdot}\bar \partial a' (f(x)) &=	a'(cL^{\times 2m} \cdot f(x)) - a'(f(x)) \\& = a'(f(cL^{\times 2m}\cdot x))- a'(f(x)) \\&= a(cL^{\times 2m} \cdot x) - a(x) = (c) {\cdot} \bar \partial a(x).
\end{align*}
By uniqueness, $f^*(\bar \partial a') = \bar \partial a'\circ f = \bar \partial a = \bar \partial f^*(a')$.
\end{proof}

\begin{lemma} \label{lem:d'}
	Let $\partial': \Inv(I_{12}^3) \to \Inv(PI_{12}^3)$ be the residue map associated to the fibration $\begin{tikzcd}[row sep=small]
H^1(*,\bm{\mu}_2) \ar[r,squiggly]	& I^3_{12} \ar[r,two heads] & PI^3_{12} \end{tikzcd}
$ described in \eqref{eq:first-fibration}.  We have \begin{align*}
	\partial' z_3 &= 0\\
	\partial' z_5 &= (-1) {\cdot} z_3   \\
	\partial' z^h &= h{\cdot} z_5  & \text{for all $h \in J_1(k)$.}
\end{align*}
\end{lemma}

\begin{proof}
	Since $z_3$ is in the image of $\Inv(PI_{12}^3) \hookrightarrow \Inv(I_{12}^3)$,  Proposition~\ref{prop:fibration}~(ii) implies $\partial'z_3 = 0$. If $q = \langle d \rangle \llangle c \rrangle (\psi_1' \perp \langle -1 \rangle \psi_2') \in I_{12}^3(L)$ then we can reconcile using Lemma~\ref{lem:I12-invariants} that for all $b \in L^\times$, 
	\begin{align*}
	z_5(\langle b \rangle q) - z_5(	q) &= (-1){\cdot}(c){\cdot}\big((b d)- (d) \big) {\cdot} e_2(\psi_1) + (-1) {\cdot} (c) {\cdot} \big((-b d ) - (-d) \big) {\cdot} e_2(\psi_2) \\
	&= (-1){\cdot}(c){\cdot} (b){\cdot}e_2(\psi_1) + (-1){\cdot}(c) {\cdot} (b) {\cdot}e_2(\psi_2) \\
	&= (b){\cdot}(-1){\cdot} z_3(q)
	\end{align*}
	If $h \in J_1(k)$ then $h{\cdot}(-1) = 0$ so $h{\cdot} z_5(\langle b \rangle q) = h{\cdot} z_5(q)$. Taking an anisotropic vector $v$,
	\[
	z^h(\langle b \rangle q) - z^h(q) = h{\cdot}(\langle b \rangle q(v)){\cdot} z_5(\langle b \rangle q) - h{\cdot}(q(v)){\cdot}z_5( q) = (b){\cdot} h{\cdot}z_5(q),
	\]
	hence $\partial' z^h = h{\cdot} z_5$.
%	The fact that $\partial' z^h = z_5$ if $\sqrt{-1} \in k$ is just like the calculation from Lemma \ref{lem:d}, only easier. 
\end{proof}

There are general methods for determining the invariants of a direct product, for instance \cite[Lemma~6.7]{garibaldi2009cohomological}, but they require that the factors have free modules of invariants. In the absence of freeness, for instance when the group is $\mathbf{O}_{4\ell + 2}^+$ and $-1 \notin {k^{\times 2}}$, we can use the following lemma.

\begin{lemma} \label{lem:invariants-of-product}
	Let $G$ be an algebraic group and identify $\Inv(G)$ naturally with its image in $\Inv(\bm{\mu}_{2m} \times G)$. Then $\Inv(\bm{\mu}_{2m} \times G) = \Inv(G) \oplus s {\cdot} \Inv(G)$ where $s \in \Inv(\bm{\mu}_{2m} \times G)$ is the invariant: \[(cL^{\times 2m}, \zeta) \mapsto (c) \in H^1(L, \ZZ/2\ZZ)\] for all field extensions $L/k$,  $\zeta \in H^1(L, G)$, and $c \in L^\times$.	
\end{lemma}

\begin{proof}
The fibration
	 $\begin{tikzcd} H^1(*,\bm{\mu}_{2m}) \ar[r,squiggly]	& H^1(*, \bm{\mu}_{2m}\times G) \ar[r,two heads] & H^1(*, G) \end{tikzcd}$
	associated to the short exact sequence $\bm{\mu}_{2m} \to \bm{\mu}_{2m} \times G \to G$ induces the exact sequence of $H(k)$-modules
	\[
	\begin{tikzcd}
0 \ar[r] & \Inv(G) \ar[r,"\pi^*"] & \Inv(\bm{\mu}_{2m} \times G) \ar[r] & \Inv(G).
\end{tikzcd}
	\]
	 The residue map $\Inv(\bm{\mu}_{2m} \times G) \to \Inv(G)$ admits a splitting $a \mapsto s{\cdot}j^*(a)$ where $j: G \to \bm{\mu}_{2m} \times G$ is the natural inclusion, so we are done.
\end{proof}

\subsection{A surjection onto $I_{12}^3$} \label{sec:a-surjection} Because of Pfister's Theorem on $12$-dimensional quadratic forms in $I^3$ (see \ref{sec:I12}), we have for all fields $L/k$ a surjective map
\begin{align*}
	H^1(L, \bm{\mu}_2 \times \mathbf{O}^+_6 ) =   (L^\times/L^{\times 2}) \times I^2_{6}(L) &\longrightarrow I^3_{12}(L) = H^1(L, \mathbf{\Gamma}_{12}^+)\\ 
	(c L^{\times 2}, r) &\longmapsto \llangle c \rrangle r.
\end{align*}
Consequently, there is an injective homomorphism
\[
\Inv(I_{12}^3) = \Inv(\mathbf{\Gamma}_{12}^+) \hookrightarrow \Inv(\bm{\mu}_2 \times \mathbf{O}_6^+) = \Inv(\mathbf{O}_6^+)\oplus s{\cdot} \Inv(\mathbf{O}_{6}^+)\]
where $s(cL^{\times 2}, r) = (c)$ for all $c \in L^\times$, $r \in H^1(L, \mathbf{O}_6^+)$.
By Serre's Theorem \ref{thm:serre}, the right-hand side is a direct sum of the free module generated by $\{1, w_2, w_4, s, s{\cdot} w_2, s {\cdot} w_4\}$ and the module $\{ b^h \mid h \in J_{1}(k)\} \oplus \{s{\cdot} b^h \mid h \in J_{1}(k)\}$. To determine $\Inv(I_{12}^3)$, it remains to determine which of these invariants is in the image of $\Inv(I_{12}^3)$. Both compositions $H^1(k, \mathbf{O}_6^+) \to H^1(k, \bm{\mu}_2 \times \mathbf{O}_6^+) \to I_{12}^3(L)$ and $H^1(k, \bm{\mu}_2) \to H^1(k, \bm{\mu}_2 \times \mathbf{O}_6^+) \to I_{12}^3(L)$ are the zero maps, so the image of $\Inv(I_{12}^3)$ is contained in the proper submodule  
\begin{align} \label{eq:submodule}
H(k){\cdot}s{\cdot} w_2 \oplus H(k){\cdot}s{\cdot} w_4 \oplus \{s{\cdot} b^h \mid h \in J_{1}(k) \} \subset \Inv(\bm{\mu}_2 \times \mathbf{O}_6^+).\end{align}
The following technical lemma settles the matter: this submodule is the image of $\Inv(I_{12}^3)$.

\begin{lemma} \label{lem:serre-calculation}
	If $q = \llangle c \rrangle r$ where $r = \langle d \rangle (\psi_1'\perp \langle -1 \rangle \psi_2')$,  $\psi_i = \llangle x_i, y_i \rrangle$, and $c, d, x_i, y_i \in k^\times$, then
	\begin{align*}
	(c){\cdot} w_2(r) &= z_3(q),	\\
	(c){\cdot} w_4(r) &= z_5(q) + (-1){\cdot}(-1){\cdot} z_3(q), 	\\
	(c){\cdot} b^h(r) &= z^h(q) & \text{ for all } h \in J_{1}(k). 
	\end{align*}
\end{lemma}
\begin{proof}
	We have $w_2(r) = e_2(r)$ if the Witt class of $r$ is in $I^2(k)$ \cite[p.\ 31]{elman2008algebraic} so $(c){\cdot}w_2(r) = (c){\cdot} e_2(r) = e_3(\llangle c \rrangle d) = z_3(q)$, hence the first identity. Towards the second identity,  \eqref{eq:total-sw} implies
	 \begin{equation} \label{eq:w4} w_4(r) = w_1(\langle d \rangle \psi_1'){\cdot} w_3(\langle -d \rangle \psi_2') + w_2(\langle d \rangle \psi_1'){\cdot} w_2(\langle -d \rangle \psi_2') + w_3(\langle d \rangle \psi_1') {\cdot} w_1(\langle -d \rangle \psi_2'). \end{equation}
	 Now, $w_1(\langle d \rangle \psi_1') = w_1(\langle -dx_1, -dy_1, dx_1y_1\rangle) = (d)$ and $w_1(\langle -d \rangle \psi_2') = (-d)$. Further, \begin{align*}e_2(\langle d \rangle \psi_1) = w_2(\langle d \rangle \perp \langle d \rangle \psi_1') = w_2(\langle d \rangle \psi_1') + w_1(\langle d \rangle){\cdot}w_1(\langle d \rangle \psi_1') = w_2(\langle d \rangle \psi_1') + (d){\cdot}(d)	
 \end{align*}
 	hence $w_2(\langle d \rangle \psi_1') = e_2(\langle d \rangle \psi_1) + (d){\cdot}(d) = e_2(\psi_1) + (-1){\cdot}(d)$. Clearly then $w_2(\langle -d \rangle \psi_2') = e_2(\psi_2) + (-1){\cdot}(-d)$. After some basic manipulations with identities $(ab) = (a) + (b)$ and $(a)(-a) = 0$, 
 	\begin{align*}w_3(\langle d \rangle \psi_1') &= (-dx_1){\cdot}(-dy_1){\cdot}(dx_1y_1) %= [(-d){\cdot}(-x_1x_2) + (x_1){\cdot}(x_2)](dx_1x_2)  \\ &= (x_1){\cdot}(x_2){\cdot}(dx_1 x_2)
 	= (x_1){\cdot}(y_1){\cdot}(d), && w_3(\langle -d \rangle \psi_2') = (x_2){\cdot}(y_2){\cdot}(-d). \end{align*}
	Plugging these calculations into \eqref{eq:w4}, the first and third terms vanish and we are left with
	\begin{align*}
	w_4(r) &= [e_2(\psi_1) + (-1){\cdot}(d)]{\cdot}[e_2(\psi_2) + (-1){\cdot} (-d)]\\ &= e_2(\psi_1){\cdot} e_2(\psi_2) + (-1){\cdot}(-d) {\cdot} e_2(\psi_1) + (-1){\cdot}(d) {\cdot} e_2(\psi_2).
	\end{align*}
	Comparing with Lemma \ref{lem:I12-invariants}, we prove the second identity: 
	\begin{align*}
	(c){\cdot} w_4(r) - z_5(q) &= (-1){\cdot}(c){\cdot}e_2(\psi_1){\cdot}[(-d) + (d)] + (-1){\cdot}(c) {\cdot} e_2(\psi_2){\cdot}[(d) + (-d)] \\ &=(-1){\cdot}(-1){\cdot}(c){\cdot}[e_2(\psi_1) + e_2(\psi_2)] = (-1){\cdot}(-1){\cdot}z_3(q).
	\end{align*}
	We prove the third identity by comparing the following calculation with \eqref{eq:zh}:
	\begin{align*}b^h(r) &= h{\cdot} (-dx_1){\cdot}(-dy_1){\cdot}(dx_1 y_1){\cdot}(d x_2){\cdot}(d y_2)\\
	&= w_3(\langle d \rangle \psi_1') {\cdot}(dx_2){\cdot}(dy_2) = h{\cdot}(x_1){\cdot}(y_1){\cdot}(d){\cdot}[(d){\cdot}(-x_2y_2) + (x_2){\cdot}(y_2)] \\
	&= h{\cdot}(-1){\cdot}(x_1){\cdot}(y_1){\cdot}(d){\cdot}(-x_2y_2) +  h{\cdot}(d){\cdot}(x_1){\cdot}(y_1){\cdot}(x_2){\cdot}(y_2)\\
	&= h{\cdot}(d){\cdot}(x_1){\cdot}(y_1){\cdot}(x_2){\cdot}(y_2). \qedhere
	\end{align*}
\end{proof}

\begin{theorem}\label{thm:I12-Spin12}
	The natural inclusion $\Inv(I_{12}^3) \to \Inv(\mathbf{Spin}_{12})$ is an isomorphism, and $\Inv(I_{12}^3)$	 is a direct sum of the free $H(k)$-module with basis $\{1, z_3, z_5\}$ and the $H(k)$-module $\{z^h \mid h \in J_{1}(k)\} \simeq J_{1}(k)$.
\end{theorem}

\begin{proof}
The classification of invariants of $I_{12}^3$ is already completed by \ref{sec:a-surjection} and \ref{lem:serre-calculation}. So it remains to show that $\Inv(I_{12}^3) \simeq \Inv(\mathbf{Spin}_{12})$.

There is an inclusion $\bm{\mu}_4 \times \mathbf{O}_6^+ \subset \mathbf{Spin}_{12}$ such that $H^1(L, \bm{\mu}_4 \times \mathbf{O}_6^+) \to H^1(L, \mathbf{Spin}_{12})$ is surjective for all fields $L/k$ \cite[Example~17.12]{garibaldi2009cohomological}. The composition of this map with $H^1(L, \mathbf{Spin}_{12}) \to H^1(L, \mathbf{O}_{12}^+)$ is $(cL^{\times 4}, r) \mapsto \llangle c \rrangle r$. In other words, it is just an extension of the surjection from \ref{sec:a-surjection}. There are two fibrations (the horizontal lines) connected by the canonical surjective maps, denoted $S$ and $T$,
\[
\begin{tikzcd}
	H^1(*, \bm{\mu}_2) \ar[r, squiggly] & H^1(*, \bm{\mu}_4 \times \mathbf{O}_6^+) \ar[r, two heads] \ar[d,"S"] & H^1(*, \bm{\mu}_2 \times \mathbf{O}_6^+) \ar[d,"T"] \\
	H^1(*, \bm{\mu}_2) \ar[r, squiggly] & H^1(*, \mathbf{Spin}_{12}) \ar[r, two heads] & I_{12}^3(*)
	\end{tikzcd}
\]
such that $S_L$ is $L^\times/L^{\times 2}$-equivariant for all fields $L/k$, and the square formed by $S, T$, and the horizontal arrows is commutative. By Lemma \ref{lem:invariants-of-product}, the map $\Inv(\bm{\mu}_2 \times \mathbf{O}_6^+) \rightarrow \Inv(\bm{\mu}_4 \times \mathbf{O}_6^+)$ is an isomorphism and by Proposition~\ref{prop:fibration}~(ii) the residue $\partial: \Inv(\bm{\mu}_4 \times \mathbf{O}_6^+) \to \Inv(\bm{\mu}_2 \times \mathbf{O}_6^+)$ associated to the first fibration is zero. Since $S^*$ is injective, Proposition~\ref{prop:fibration}~(iii) implies the residue $\partial: \Inv(\mathbf{Spin}_{12}) \to \Inv(I_{12}^3)$ associated to the second fibration is also zero. In turn, Proposition~\ref{prop:fibration}~(ii) implies the theorem.  \end{proof}

For an alternative  proof that $\Inv(I_{12}^3) \simeq \Inv(\mathbf{Spin}_{12})$, one can work out the image of the embedding $\Inv(\mathbf{Spin}_{12}) \to \Inv(\bm{\mu}_4 \times \mathbf{O}^+_6) \simeq \Inv(\bm{\mu}_2 \times \mathbf{O}^+_6) $. The submodule \eqref{eq:submodule} is a lower bound because we proved that this is the image of $\Inv(I_{12}^3) \to \Inv(\bm{\mu}_2 \times \mathbf{O}^+_6)$. It is also an upper bound, because the compositions $H^1(k,\mathbf{O}_6^+) \to H^1(k,\bm{\mu}_4 \times \mathbf{O}_6^+) \to H^1(k,\mathbf{Spin}_{12})$ and $H^1(k,\bm{\mu}_4) \to H^1(k,\bm{\mu}_4 \times \mathbf{O}_6^+) \to H^1(k,\mathbf{Spin}_{12})$ are the trivial maps.

\begin{corollary} \label{cor:Inv-PI12}
	The image of the homomorphism $\Inv(PI_{12}^3) \hookrightarrow \Inv(I_{12}^3)$ is 
	\[
	H(k){\cdot} 1 \oplus H(k){\cdot} z_3 \oplus \{ h{\cdot} z_5 \mid h \in J_1(k) \}.
	\]
\end{corollary}

\begin{proof}
		By Proposition \ref{prop:fibration}~(ii) the image is $\ker(\partial')$, which is known from Lemma \ref{lem:d'}.
\end{proof}

\section{Invariants of $I_{14}^3$ and $\mathbf{Spin}_{14}$}

\subsection{Known invariants of $I_{14}^3$} \label{sec:I14-known-invariants}
There are three nontrivial invariants in $\Inv(I_{14}^3)$ that are known from  \cite[\S22.3]{garibaldi2009cohomological}. These occur in degrees $3$, $6$, and $7$, but the degree~7 invariant is only known to exist for fields $k$ in which $\sqrt{-1} \in k$. The degree~3 invariant, which we denote by $a_3 \in \Inv(I_{14}^3)$, is the restriction of the Arason invariant 
$e_3: I^3(*) \to H^3(*, \ZZ/2\ZZ)$  to $14$-dimensional forms.  The degree 6 invariant is denoted by $a_6 \in \Inv(I_{14}^3)$ and it is the restriction of $e_6 \circ P_3 : I^3(*) \to H^6(*,\ZZ/2\ZZ)$, where $P_3: I^3(*) \to I^6(*)$ is the functor defined in \ref{sec:P_n}. The third invariant (defined as long as $\sqrt{-1} \in k$) is denoted by $a_7 \in \Inv(I_{14}^3)$ and for a quadratic form $Q \in I_{14}^3(k)$ it takes the value
\[
a_7(Q) = (Q(v)){\cdot} a_6(Q)
\]
	where $v \in k^{14}$ is any anisotropic vector for $Q$. When $Q \in I^3_{14}$ is similar to the difference of two $3$-Pfister forms, as in Corollary~\ref{cor:Rost}~(1), it is easy to give an expression for $a_6(Q)$:
\begin{lemma} \label{lem:a6-decomposable}
	If $Q = \langle c \rangle(\phi_1' \perp \langle -1 \rangle \phi_2')$ where $\phi_i$ are $3$-Pfister forms over $k$, then 
	\begin{align*}
	a_6(Q) &= e_3(\phi_1){\cdot} e_3(\phi_2) + (-1){\cdot}(-1){\cdot} (c){\cdot} e_3(\phi_1) + (-1){\cdot}(-1){\cdot}(-c){\cdot} e_3(\phi_2).
	\end{align*}
\end{lemma}
\begin{proof}
	By \eqref{eq:big-Pn}, 
	$
	P_3(Q) = \langle -1 \rangle \phi_1\phi_2 + 4\llangle c \rrangle \phi_1 + 4 \llangle -c \rrangle \phi_2
	$
	and the lemma follows by applying $e_6$ to this expression.
\end{proof}

Provided $-1$ is a sum of two squares in $k$, we can also express $a_6(Q)$ quite easily for any $Q \in I^3_{14}(k)$, and this is done in \eqref{eq:a6:concrete} and \eqref{eq:a6:concrete2}, but in full generality it is hard. The following lemma generalises \cite[Proposition~22.2~(2)]{garibaldi2009cohomological}. Unfortunately \cite[Proposition~22.2~(1)]{garibaldi2009cohomological} is not correct:   Lemma \ref{lem:a6-decomposable} implies that the isotropic form $Q = \llangle -1, t_1, t_2 \rrangle' \perp \langle -1\rangle \llangle -1, t_3, t_4 \rrangle'$ over $\mathbb{R}(t_1, t_2, t_3, t_4)$ evaluates to a non-symbol \[a_6(Q) = (-1){\cdot}(-1){\cdot}(t_1){\cdot}(t_2){\cdot}(t_3){\cdot}(t_4) +  (-1){\cdot}(-1){\cdot} (-1){\cdot} (-1){\cdot}(t_3){\cdot}(t_4).\]

\begin{lemma} \label{lem:a6-vanish}
	If $Q\in I_{14}^3(k)$ is isotropic then $a_6(Q) \in (-1){\cdot} H^5(k,\ZZ/2\ZZ)$.	
\end{lemma}

\begin{proof}
	An isotropic $Q \in I_{14}^3(k)$ is of the form $Q = \langle c \rangle(\phi_1' \perp \langle -1 \rangle \phi_2')$ where the $\phi_i = \llangle x, y_i, z_i \rrangle$ have a common slot (\ref{sec:I12}). By Lemma \ref{lem:a6-decomposable}, $a_6(Q) = e_3(\phi_1){\cdot}e_3(\phi_2)$ modulo $(-1){\cdot}H^5(k,\ZZ/2\ZZ)$, and \[e_3(\phi_1){\cdot}e_3(\phi_2) = (x){\cdot}(y_1){\cdot}(z_1){\cdot}(x){\cdot}(y_2){\cdot}(z_2) = (-1){\cdot}(x){\cdot}(y_1){\cdot}(z_1){\cdot}(y_2){\cdot}(z_2). \qedhere \]  \end{proof}

With this lemma at hand, one can use generalise the invariant $a_7$. For $h \in J_1(k)$, let 
\[
a^h(Q) = h{\cdot} Q(v){\cdot} a_6(Q)
\]
where $v$ is an anisotropic vector for $Q$. This defines an invariant because $h{\cdot} a_6$ vanishes on isotropic forms in $I_{14}^3(k)$ and \cite[Proposition~10.2]{garibaldi2009cohomological} implies that the quantity $h{\cdot} Q(v) {\cdot} a_6(Q)$ does not depend on the choice of $v$. If $-1 \in k^{\times 2}$, then $J_1(k) = H(k)$ and $a^h = h{\cdot} a^1$ for all $h \in H(k)$.

\begin{lemma} \label{lem:a6-sim}
	For all $Q \in I_{14}^3(k)$ and $c \in  k^\times$,
	\[
	a_6(\langle c \rangle Q) - a_6(Q) = (-1) {\cdot} (-1){\cdot} (c) {\cdot} a_3(Q).
	\]	
\end{lemma}
\begin{proof}
	This follows directly from \eqref{eq:Pn-sim}.
\end{proof}

	We would like to compare the degree 6 invariant of $H^1(*,(G_2 \times G_2)\rtimes S_2)$ with the degree 6 invariant of $I_{14}^3(*)$, but at first glance it is not clear how to do this because there is no morphism between these two functors. (Recall that the Albert form of a bi-octonion algebra is only defined up to similitude, and $a_6$ is not always compatible with similitudes.) 
		Theorem~\ref{thm:big-calc} is the best we can do: it bounds the difference between the  quadratic trace invariant $b_6$ of  a bi-octonion algebra and the invariant $a_6$ of a quadratic form similar to one of its Albert forms. The main technical device is the following formula which makes a seriously useful connection between the operations of additive transfer (\ref{sec:transfers-and-squares}), multiplicative transfer (\ref{sec:multiplicative-transfer}), and exterior square (\ref{sec:exterior-square}):

\begin{theorem}[Rost, Wittkop {\cite[Satz~2.12]{wittkop}}] \label{thm:lambda2}
	Let $E= k(\sqrt{d})$ be a quadratic field extension, and let $x \in \widehat{W}(E)$. Then
	\[
	\lambda^2(T_{E/k}(x)) = T_{E/k}(\lambda^2(x)) + \langle d \rangle N_{E/k}(x).
	\]
\end{theorem}

\begin{theorem} \label{thm:big-calc}
Suppose $Q \in I_{14}^3(k)$ and $\beta \in H^1(k, (G_2\times G_2)\rtimes S_2)$  have the same image in $PI_{14}^3(k)$. Then
\[
a_6(Q) - b_6(\beta)  \in (-1) {\cdot} (-1) {\cdot} H^4(k, \ZZ/2\ZZ).
\]
\end{theorem}

\begin{proof}
	Case 1: Suppose  $b_1(\beta) = 0$. Then $\beta$ is the isomorphism  class of a decomposable bi-octonion algebra $C_1 \otimes C_2$ where $C_i$ are some octonion algebras with norms $n_i$. And $Q$ is similar to the Albert form of $C_1 \otimes C_2$, say $Q = \langle c \rangle (n_1'\perp\langle -1 \rangle n_2')$ for some $c \in k^\times$. We have $b_6(\beta) = e_6(n_1 \cdot n_2) = e_3(n_1){\cdot} e_3(n_2)$. By Lemma \ref{lem:a6-decomposable}, $a_6(Q) - b_6(Q)  \in (-1){\cdot}(-1){\cdot} H^4(k, \ZZ/2\ZZ)$. 
	
	Case 2: Suppose $b_1(\beta) \ne 0$. The $\beta_1(\beta)$ is the class of a quadratic field extension $E/k$, and $\beta$ is the isomorphism class of an indecomposable bi-octonion algebra $\cor_{E/k}(C)$ where $C$ is some octonion algebra over $E$ with norm $n$. Now $Q$ is similar to the Albert form of $\cor_{E/k}(C)$, say $Q = T_{E/k}(\langle \sqrt{d} \rangle n')$ where $d \in k^\times\setminus k^{\times 2}$ is an element whose square root generates $E$.
	Note that $Q \perp \mathbb{H} \simeq T_{E/k}(\langle \sqrt{d}\rangle n)$, so $Q = T_{E/k}(\langle \sqrt{d} \rangle n)$ in ${W}(k)$. Now, by \eqref{eq:actual-Pn}:
	\begin{equation*}
	P_3(Q) = P_3(T_{E/k}(\langle \sqrt d \rangle n)) = 8 + \lambda^2(T_{E/k}(\langle \sqrt{d}\rangle n))-4T_{E/k}(\langle \sqrt{d} \rangle n).
	\end{equation*}
	Applying Theorem \ref{thm:lambda2} yields	\begin{align*}
	P_3(Q) &= 8 + T_{E/k}(\lambda^2(\langle \sqrt{d} \rangle n))+ \langle d \rangle N_{E/k}(\langle \sqrt{d} \rangle n)-4T_{E/k}(\langle \sqrt{d} \rangle n).
	\end{align*}
	We have $\lambda^2(\langle \sqrt{d} \rangle n) = \lambda^2(n) = 4n'$ by \eqref{eq:lambda2-sim} and \cite[Lemma~19.8]{garibaldi2009cohomological}. And we have $N_{E/k}(\langle \sqrt{d} \rangle n) = N_{E/k}(\langle \sqrt{d}\rangle) N_{E/k}(n) = \langle{-d}\rangle N_{E/k}(n)$ by multiplicativity of $N_{E/k}$ and \eqref{eq:N-one-d}. It follows that 	\begin{equation*}
	P_3(Q) = 8 + T_{E/k}(4n') - N_{E/k}( n)- 4T_{E/k}(\langle \sqrt{d} \rangle n)
	\end{equation*}
Since $T_{E/k}$ is $H(k)$-linear, we may write
\begin{align*}
T_{E/k}(4n') = 4T_{E/k}(n') &= - 4T_{E/k}(\langle 1 \rangle) + 4T_{E/k}(\langle 1 \rangle+n') = -4\langle 2, 2d \rangle + 4T_{E/k}(n).
\end{align*}
Therefore
\begin{align*}
P_3(Q) &= 8 - 4 \langle 2, 2d \rangle + 4T_{E/k}(n) - 4T_{E/k}(\langle \sqrt{d} \rangle n) - N_{E/k}( n) \\
&= 8-4 \langle 2, 2d \rangle + 4T_{E/k}(\langle 1,-\sqrt{d}\rangle n)-N_{E/k}(n) \\
&= 8-4 \langle 2, 2d \rangle + 4T_{E/k}(\llangle \sqrt{d} \rrangle n)-N_{E/k}(n). 
\end{align*}
Now,
\begin{align*}
	P_3(Q) + N_{E/k}(n) - 4 \llangle d \rrangle  &= 8-4\langle 2, 2d \rangle - 4\langle 1,- d \rangle  + 4T_{E/k}(\llangle \sqrt{d} \rrangle n)\\ &= 4\llangle 2,-d\rrangle +4T_{E/k}(\llangle \sqrt{d} \rrangle n) = 4T_{E/k}(\llangle \sqrt{d} \rrangle n)
\end{align*}
since $4 \llangle 2, -d \rrangle \simeq 8\mathbb{H}$ by a straightforward calculation using the usual isometry criterion for binary forms \cite[Proposition~5.1]{lam}. Finally, since $T_{E/k}(\llangle \sqrt{d} \rrangle n) \in I^4(k)$ \cite[Corollary 34.17]{elman2008algebraic}, this implies \begin{align*}
a_6(Q) - b_6(\beta) &= a_6(Q) + b_6(\beta) = e_6(P_3(Q) + N_{E/k}(n) - 4 \llangle d \rrangle)\\
&= e_6(4T_{E/k}(\llangle \sqrt{d} \rrangle n)) = (-1){\cdot} (-1){\cdot} e_4(T_{E/k}(\llangle \sqrt{d} \rrangle n))
\end{align*}
and therefore $a_6(Q)-b_6(\beta) \in (-1){\cdot} (-1){\cdot} H^4(k,\ZZ/2\ZZ)$.
\end{proof}

As a consequence of Theorems \ref{thm:symbols}~(ii) and \ref{thm:big-calc}, we also have a generalisation of \cite[Proposition~22.2~(3)]{garibaldi2009cohomological}:

\begin{corollary}
	If $-1$ is a sum of two squares	in $k$, then $a_6(Q)$ and $a_7(Q)$ are symbols for all $Q \in I_{14}^3(k)$.
\end{corollary}

In fact, if $-1$ is a sum of two squares then the symbols $a_6(Q)$ and $a_7(Q)$ can be determined explicitly from the proof of Theorem~\ref{thm:symbols}~(ii). Namely, if $Q$ is of the form $Q = \langle c \rangle (\llangle x_1, x_2, x_3 \rrangle' \perp \langle -1 \rangle \llangle y_1, y_2, y_3 \rrangle')$ then
\begin{equation} \label{eq:a6:concrete}
a_6(Q) = (x_1) {\cdot}(x_2) {\cdot} (x_3) {\cdot} (y_1) {\cdot} (y_2) {\cdot} (y_3).
\end{equation}
If $Q =  T_{E/k}(\langle \delta \rangle \llangle z_1, z_2, z_3 \rrangle')$ for some quadratic field extension $E/k$, $z_i \in E$, and $\delta \in \ker(\tr_{E/k})$ then either $a_6(Q) = 0$ (if one of the $z_i$ has zero trace) or
\begin{equation} \label{eq:a6:concrete2}
	a_6(Q) = \prod_{i = 1}^3 (\tr_{E/k}(z_i)){\cdot}(-\delta^2 N_{E/k}(z_i)).
\end{equation}

\subsection{Classifying the invariants of $PI_{14}^3$}

Since the functors $H^1(*,(G_2\times G_2)\rtimes S_2) \to PI_{14}^3$  and $I_{14}^3 \to PI_{14}^3$ are surjective, they induce injective homomorphisms $\Inv(PI^3_{14}) \hookrightarrow \Inv(I_{14}^3)$ and $\Inv(PI^3_{14}) \hookrightarrow \Inv((G_2\times G_2)\rtimes S_2)$.

\begin{proposition} \label{prop:isotopy-invariants}
	Suppose $\beta, \beta' \in H^1(k, (G_2\times G_2)\rtimes S_2)$ have the same image in $PI_{14}^3(k)$. Then $b_3(\beta) = b_3(\beta')$ and $b_6(\beta) - b_6(\beta') \in (-1){\cdot}{(-1)}{\cdot}H^4(k, \ZZ/2ZZ)$.
	
	(Equivalently, if $(A,-)$ and $(A',-)$ are isotopic bi-octonion algebras, then $b_3(A,-) = b_3(A',-)$ and $b_6(A,-) = b_6(A',-)$ modulo $(-1){\cdot}{(-1)}{\cdot}H^4(k, \ZZ/2\ZZ)$.) 
\end{proposition}

\begin{proof}
The equivalence of the two statements follows from Proposition \ref{prop:isotopic-similar}. Assuming that $(A,-)$ and $(A', -)$ are isotopic bi-octonion algebras with similar Albert forms $Q$ and $Q'$ respectively, we have by definition, $b_3(A,-) - b_3(A',-) = e_3(Q) - e_3(Q') = 0$. By Lemma~\ref{lem:a6-sim} and Theorem~\ref{thm:big-calc},
	\[ (b_6(A,-)- a_6(Q)) + (a_6(Q) - a_6(Q')) + (a_6(Q')-b_6(A',-))    \in (-1) {\cdot} (-1) {\cdot} H^4(k,\ZZ/2\ZZ). \qedhere \] 
\end{proof}

\begin{theorem}
	The image of the homomorphism $\Inv(PI_{14}^3) \hookrightarrow \Inv((G_2\times G_2)\rtimes S_2)$ is
	\[
	H(k) {\cdot} 1 \oplus H(k) {\cdot} b_3 \oplus J_2(k) {\cdot} b_6.
	\]
\end{theorem}

\begin{proof}
 	Proposition \ref{prop:isotopy-invariants} implies
 	\[
 	H(k) {\cdot} 1 \oplus H(k) {\cdot} b_3 \oplus J_2(k) {\cdot} b_6 \subset \im\big(\Inv(PI_{14}^3) \to \Inv((G_2\times G_2)\rtimes S_2)\big).
 	\]
 	For the reverse inclusion, recall from Theorem \ref{thm:inv-bioctonions} that $\Inv((G_2\times G_2)\rtimes S_2)$ is the free $H(k)$-module with basis $\{1, b_1, b_3, b_6 \}$. Suppose $\lambda, \mu, \nu \in H(k)$ are such that
 	\[b = \lambda {\cdot} b_1 + \mu {\cdot} b_3 + \nu {\cdot} b_6\]
 	is in the image of  $\Inv(PI_{14}^3)$. This assumption means that $b(A,-) = b(A',-)$ for all pairs of bi-octonion algebras $(A,-)$ and $(A',-)$ over any field extension $L/k$, as long as they have similar Albert forms (equivalently, are isotopic).
 	
 	Consider the field $K = k(t)$ and the algebras: \begin{align*}  (B,-) &=  \text{the split bi-octonion algebra over }K, \\
 	 (B',-) &=  \cor_{K(\sqrt{t})/K}(C) \text{ where } C \text{ is the split octonion algebra over } K(\sqrt{t}) 
 \end{align*}
 The Albert form of $(B',-)$ is hyperbolic because it is the additive transfer of the hyperbolic form $\llangle 1,1,1 \rrangle$ over $K(\sqrt{t})$. Hence $0 = b(B,-) = b(B',-)$. Clearly $b_3(B',-) = 0$. By \cite[Lemma~2.13]{wittkop} together with basic fact that $4\langle 2 \rangle = 4$ in $W(k)$ for any field $k$, \[b_6(B',-) = e_6(N_{K(\sqrt{t})/K}(4\mathbb{H})-4 \llangle t \rrangle) = e_6(4\langle 2 \rangle \llangle t \rrangle - 4 \llangle t \rrangle) = e_6(0) = 0.\]
 	 This all implies
	\[
	0 = b(B',-) = b_1(B',-) =  \lambda {\cdot} (t),
	\]
	so $\lambda = 0$. Now let $K' = k(t_1, t_2, t_3, t_4)$ and consider the decomposable bi-octonion algebras \begin{align*} (D,-) = C_1\otimes C_2 && (D',-) = C_1'\otimes C_2',\end{align*} where the $C$'s are octonion algebras over $K'$ with the following norms:	\begin{align*}
	n_{C_1} &= \llangle t_1, t_2, t_3 \rrangle, & n_{C_2} &= \llangle t_1, t_2, t_4 \rrangle, \\
	n_{C_1'} &= \llangle t_1, t_2, t_3^{-1}t_4 \rrangle,	& n_{C_2'} &= \llangle 1,1,1\rrangle \text{ (i.e., hyperbolic).}
	\end{align*}
	The Albert form of $(D,-)$ is similar to the Albert form of $(D',-)$, because in $W(K')$:
	\begin{align*}
		\llangle t_1, t_2,t_3 \rrangle - \llangle t_1, t_2, t_4 \rrangle &= \llangle t_1, t_2 \rrangle (\llangle t_3 \rrangle - \llangle t_4 \rrangle)= \llangle t_1, t_2 \rrangle \langle -t_3, t_4 \rangle \\ &= \langle -t_3 \rangle \llangle t_1, t_2,t_3^{-1}t_4 \rrangle = \langle -t_3\rangle (\llangle t_1, t_2, t_3^{-1}t_4 \rrangle - \llangle 1,1,1 \rrangle).
	\end{align*}
	In particular, $b_3(D,-) = b_3(D',-) = (t_1){\cdot} (t_2){\cdot}(t_3^{-1}t_4)$. It is also clear that $b_1(D,-) = b_1(D',-) = 0$ because both algebras are decomposable. Since we assumed that $b$ is an isotopy invariant, we have $0 = b(D,-) - b(D',-) = \nu {\cdot} b_6(D,-) - \nu {\cdot} b_6(D',-)$. The $b_6$'s are straightforward to calculate:
	\begin{align*}
	b_6(D,-) &= (t_1){\cdot}(t_2){\cdot}(t_3){\cdot}(t_1){\cdot}(t_2) {\cdot}(t_4) = (-1){\cdot}(-1){\cdot}(t_1){\cdot}(t_2){\cdot}(t_3) {\cdot}(t_4),\\
		b_6(D',-) &= (1){\cdot}(1){\cdot}(1) {\cdot}(t_1){\cdot}(t_2){\cdot}(t_3^{-1}t_4) = 0.
	\end{align*}
	This implies \[0 = b(D,-) - b(D',-) = \nu {\cdot} (-1){\cdot} (-1) {\cdot} (t_1){\cdot}(t_2){\cdot} (t_3) {\cdot} (t_4).\] Symbols of the form $(t_{i_1}){\cdot \cdots \cdot} (t_{i_j})$ in $H(K')$ are  $H(k)$-linearly independent \cite[Lemma~1.1]{berhuy2011cohomological}, so we conclude that 
	 $\nu {\cdot}(-1){\cdot}(-1)   = 0$.
\end{proof}

\begin{corollary} \label{cor-InvPI in InvI}
	The image of the homomorphism $\Inv(PI_{14}^3) \hookrightarrow \Inv(I_{14}^3)$ is 
	\[
	H(k) {\cdot} 1 \oplus H(k) {\cdot} a_3 \oplus J_2(k) {\cdot} a_6.
	\]
\end{corollary}

\begin{proof}
	Let $\overline {a_3}$ and $\overline {\nu {\cdot} a_6}$, for $\nu \in J_2(k)$, be the unique elements of $\Inv(PI_{14}^3)$ whose images in $\Inv(G^2\rtimes S_2)$ are $b_3$ and $\nu {\cdot} b_6$ respectively. Then $\Inv(PI_{14}^3)$ is generated by $\overline{a_3}$ and $\{\overline{\nu {\cdot} a_6} \mid \nu \in J \}$. The image of $\overline{a_3}$ in $\Inv(I_{14}^3)$ is $a_3$ (simply by comparing the definitions), while the image of $\overline{\nu {\cdot} a_6}$ in $\Inv(I_{14}^3)$ is $\nu {\cdot} a_6$ (because of Theorem \ref{thm:big-calc}).
\end{proof}

\subsection{An important fibration} \label{sec:important-fibration}
Let $n$ be even and let  $C$ be the centre of $\mathbf{Spin}_n$, so  $C \simeq \bm{\mu}_2 \times \bm{\mu}_2$ or $C \simeq \bm{\mu}_4$ according as $n = 0$  mod $4$ or $n = 2$ mod $4$ \cite[p.~517]{milne}. We have a short exact sequence $\begin{tikzcd} C \ar[r] & \mathbf{Spin}_{n} \ar[r] & \mathbf{PGO}^+_n \end{tikzcd}$, and we know that $\im(H^1(*,\mathbf{Spin}_n)\to H^1(*,\mathbf{PGO}_n^+)) \simeq H^1(*,\mathbf{\Omega}_n) \simeq PI^3_n(*)$, so by Lemma \ref{lem:fibration-from-exact-sequence} there is a fibration \begin{equation*}
\begin{tikzcd}
H^1(*,C) \ar[r,squiggly]	& H^1(*,\mathbf{Spin}_{n}) \ar[r,two heads] &  PI_{n}^3. 
\end{tikzcd}
\end{equation*}
We shall describe in concrete detail the action of $H^1(*,C)$ on $H^1(*,\mathbf{Spin}_n)$ in terms of quadratic forms.
If $e_1, e_2$ are the two nontrivial orthogonal central idempotents in the even Clifford algebra of $\frac{n}{2}\mathbb{H}$,   then the points of $C$ are:
\begin{align*}
C(R) &=	 	\{\zeta e_1 + \zeta^{-1} e_2 \mid \zeta \in \bm{\mu}_4(R) \} & \text{if } n = 2 \mod 4,\\
C(R) &= \{ \xi_1 e_1 + \xi_2 e_2 \mid \xi_i \in \bm{\mu}_2(R) \} & \text{if } n = 0 \mod 4.
\end{align*}
This fixes an isomorphism $C \simeq \bm{\mu}_4$ or $C \simeq \bm{\mu}_2 \times \bm{\mu}_2$. 
The kernel $J \simeq \bm{\mu}_2$ of the map $\mathbf{Spin}_n \to \mathbf{O}^+_n$ has  points $J(R) = \{\xi 1 \mid \xi \in \bm{\mu}_2(R) \} \subset C(R)$. If $n = 2$~mod~$4$, then $C/J$ can be identified with $\bm{\mu}_2$ by $(\zeta e_1 + \zeta^{-1} e_2)J \mapsto \zeta^2$. Now we have identifications $H^1(k, C) = k^{\times}/k^{\times 4}$ and $H^1(k, C/J) = k^{\times}/k^{\times 2}$. The map $H^1(k, C) \to H^1(k, C/J)$ sends $xk^{\times 4} \mapsto xk^{\times 2}$. If $n = 0$~mod~$4$ then $C/J$ can be identified with $\bm{\mu}_2$ by $(\xi_1 e_1 + \xi_2 e_2 )J \mapsto \xi_1 \xi_2$. In this case, the set $H^1(k, C)$ is identified with $k^{\times}/k^{\times 2} \times k^{\times}/k^{\times 2}$, and the map $H^1(k, C) \mapsto H^1(k, C/J)$ sends $(xk^{\times 2}, yk^{\times 2}) \mapsto xyk^{\times 2}$. 

For $\eta \in H^1(L, \mathbf{Spin}_{n})$, let $q_\eta \in I^3_{n}(L)$ be the corresponding quadratic form. Galois cohomology produces the following commutative diagram with exact rows:
\begin{equation*}
\begin{tikzcd}
	\mathbf{PGO}_n^+(*) \ar[r, "\delta"] & H^1(*,C) \ar[r] \ar[d] & H^1(*,\mathbf{Spin}_{n}) \ar[r] \ar[d, "\eta \mapsto q_\eta"] & PI_{n}^3 \ar[d] \\
	\mathbf{PGO}_n^+(*) \ar[r, "\delta"] & H^1(*,C/J) \ar[r] & H^1(*,\mathbf{O}^+_{n}) \ar[r] & PI_{n}^2
\end{tikzcd}
\end{equation*}
The group $H^1(k, C/J) = k^{\times}/k^{\times 2}$ acts on $H^1(k,\mathbf{O}_n^+) = I^2_n(k)$ by $ck^{\times 2} \cdot q = \langle c \rangle q$. (To prove this, rather than dealing with cocycles, one can work out that stabiliser of $q$ is $G(q)/k^{\times 2}$, using \cite[I.\S5~Proposition~39~(iii)]{serre1997galois} and the calculation of the first connecting map $\delta$ from  \cite[Proposition~13.33]{knus1998book}.) Consequently, the commutativity of the diagram implies that the action of $H^1(k, C)$ on $H^1(k, \mathbf{Spin}_n)$ has the following effect, for all $\eta \in H^1(k, \mathbf{Spin}_{n})$:
\begin{align}
	\label{eq:q-c.eta} q_{xk^{\times 4} \cdot \eta} &= \langle x \rangle q_\eta  &  \text{if } n = 2 \mod 4,\\ \notag
	q_{(xk^{\times 2}, y k^{\times 2})\cdot \eta} &= \langle xy \rangle q_\eta & \text{if } n = 0 \mod 4.
\end{align}

For the next lemma, let $a_3$, $a_6$, and $a^h$, $h \in J_2(k)$, be the images of the identically named invariants of $I_{14}^3$ under the map $\Inv(I_{14}^3) \hookrightarrow \Inv(\mathbf{Spin}_{14})$; i.e.\ $a_i(\eta) = a_i(q_\eta)$ and $a^h(\eta) = a^h(q_\eta)$ for all $\eta \in H^1(L, \mathbf{Spin}_{14})$. (Doing this avoids having to introduce yet more notation.) The mod 2 part of  $\mathbf{Spin}_{14}$'s Rost invariant is therefore called~$a_3$.

\begin{lemma} \label{lem:d}Let $\partial: \Inv(\mathbf{Spin}_{14}) \to \Inv(PI_{14}^3)$ be the residue map associated to the fibration $\begin{tikzcd}
H^1(*,\bm{\mu}_4) \ar[r,squiggly]	& H^1(*,\mathbf{Spin}_{14}) \ar[r,two heads] &  PI_{14}^3 \end{tikzcd}$. We have \begin{align*}
	\partial a_3 &= 0\\
	\partial a_6 &= (-1) {\cdot} (-1)  {\cdot} a_3\\
	\partial a^h &= h{\cdot} a_6 & \text{ for all $h \in J_1(k)$.}
\end{align*}
\end{lemma}

\begin{proof}
	Since $a_3$ is in the image of $\Inv(PI_{14}^3) \hookrightarrow \Inv(\mathbf{Spin}_{14})$, Proposition~\ref{prop:fibration}~(ii) implies $\partial a_3 = 0$. Now suppose $L/k$ is any field extension, $\eta \in H^1(L, \mathbf{Spin}_{14})$, and $[c] \in H^1(L, \bm{\mu}_4) = L^{\times}/L^{\times 4}$. If $q_\eta \in I^{3}_{14}(L)$ is the quadratic form corresponding to $\eta$, then $q_{[c]\cdot \eta} = \langle c \rangle q_{\eta}$ by \eqref{eq:q-c.eta}. By Lemma~\ref{lem:a6-sim}, we have for all $\eta \in H^1(L, \mathbf{Spin}_n)$,
	\[
	a_6([c]\cdot \eta)- a_6(\eta) = (-1){\cdot} (-1) {\cdot} (c) {\cdot} a_3(\eta).
	\]
	Hence $\partial a_6 = (-1){\cdot}(-1){\cdot}(c){\cdot}a_3$. Now $h \in J_1(k)$ means $h{\cdot} (-1) = 0$ and hence $h{\cdot}a_6([c] \cdot \eta) = h{\cdot} a_6(\eta)$ for all $\eta \in H^1(L, \mathbf{Spin}_{14})$. If $v \in k^{14}$ is an anisotropic vector for $q_\eta$,
	\begin{align*}
	a^h([c]\cdot \eta) - a^h(\eta) &= h{\cdot} (q_{[c]\cdot \eta}(v)){\cdot} a_6([c]\cdot\eta) - h{\cdot}(q_\eta(v)) {\cdot} a_6(q_\eta)\\ &= h{\cdot} (\langle c\rangle  q_\eta(v)){\cdot} a_6(\eta) - h{\cdot} (q_\eta (v)){\cdot} a_6(q_\eta) = (c){\cdot} h{\cdot} a_6(\eta). \qedhere
	\end{align*}
\end{proof}

The subform $6\mathbb{H} \subset 7\mathbb{H}$ induces natural inclusions \begin{align*}
 	f: \mathbf{Spin}_{12} &\to \mathbf{Spin}_{14}\\
 	g: \mathbf{\Gamma}^+_{12} &\to \mathbf{\Gamma}^+_{14} \\
 	h: \mathbf{\Omega}_{12}&\to \mathbf{\Omega}_{14},
 \end{align*}
each of which is an extension of the former. Recalling from \ref{sec:cohomological-In} that $H^1(*, \mathbf{\Omega}_n) = PI_{n}^3$ and $H^1(*, \mathbf{\Gamma}^+_n) = I_{n}^3$, these maps induce (injective) morphisms \begin{align*} &h^*: PI_{12}^3 \to PI_{14}^3, & [q] &\mapsto [q \perp \mathbb{H}], \\
	&g^*: I_{12}^3 \to I_{14}^3, & q &\mapsto q \perp \mathbb{H}\end{align*} 
	and corresponding ring homomorphisms $H: \Inv(PI_{14}^3) \to \Inv(PI_{12}^3)$, $G: \Inv(I_{14}^3) \to \Inv(I_{12}^3)$, and $F: \Inv(\mathbf{Spin}_{14}) \to \Inv(\mathbf{Spin}_{12})$.

\begin{lemma} \label{lem:I14-I12}
	We have
\begin{align*}
G(a_3) &= z_3 \\
G(a_6) &= (-1){\cdot} z_5 \\
G(a^h) &= 0 & \text{for all $h \in J_{1}(k)$.}	
%G(a_7) &= (-1){\cdot} z_6 & \text{(if $\sqrt{-1} \in k$).}	
\end{align*}
%In particular, if $\sqrt{-1} \in k$ then $a_6$ and $a_7$ vanish on isotropic forms in $I_{14}^3$.
\end{lemma}

\begin{proof}
	Suppose $q = \langle d \rangle \llangle c \rrangle (\psi_1' \perp \langle -1 \rangle \psi_2') \in I_{12}^3(k)$ where $\psi_i = \llangle x_i, y_i \rrangle$, using the parameterisation \eqref{eq:q in I12}. By definition $a_3(q \perp \mathbb{H}) = z_3(q) = e_3(q)$.
	By Lemmas~\ref{lem:a6-decomposable} and~\ref{lem:I12-invariants}, 
	\begin{align*}
	a_6(q \perp \mathbb{H}) 
	&= e_3(\llangle c \rrangle \psi_1){\cdot} e_3(\llangle c \rrangle \psi_2) + (-1){\cdot} (-1) {\cdot} (d){\cdot} e_3(\llangle c \rrangle \psi_1) + (-1){\cdot}(-1){\cdot}(-d){\cdot}e_3(\llangle c \rrangle \psi_2) \\
	&= (c){\cdot}e_2(\psi_1) {\cdot} (c) {\cdot} e_2(\psi_2) + (-1){\cdot}(-1){\cdot}(d){\cdot}(c){\cdot}e_2(\psi_1) + (-1){\cdot} (-1){\cdot}(-d){\cdot} (c) {\cdot} e_2(\psi_2)\\
	&= (-1){\cdot}z_5(q).
	\end{align*}	
	If $h \in J_{1}(k)$ then $a^h(q \perp \mathbb{H}) = 0$ because $h{\cdot}a_6$ vanishes on isotropic forms.
\end{proof}

Recall from Theorem \ref{thm:I12-Spin12} that the natural inclusion $\Inv(I_{12}^3) \hookrightarrow \Inv(\mathbf{Spin}_{12})$ is an isomorphism. And recall that we denote by $q_\eta \in I^3_n(L)$ the (isometry class of the) quadratic form corresponding to a cohomology class $\eta \in H^1(L, \mathbf{Spin}_{n})$

\begin{lemma} \label{lem:commutes}
The following diagram  of $H(k)$-modules commutes:
\begin{equation*}
	\begin{tikzcd}
		\Inv(I_{14}^3) \ar[d,"G"] \ar[r, hookrightarrow] & \Inv(\mathbf{Spin}_{14}) \ar[r,"\partial"] \ar[d, "F"]  & \Inv_{\rm norm}(PI^3_{14}) \ar[d,"H"] \\
		\Inv(I_{12}^3) \ar[rr, "\partial'", bend right] \ar[r, hookrightarrow] & \Inv(\mathbf{Spin}_{12}) & \Inv_{\rm norm}(PI^3_{12})
	\end{tikzcd}
\end{equation*}
\end{lemma}

\begin{proof}
The small square on the left obviously commutes because the maps in it arise from the (contravariant) functoriality of $\Inv$. The rest of the proof is really about unpacking the definitions. Let $a \in \Inv(I_{14}^3)$ and denote by $\tilde a$ its image in $\Inv(\mathbf{Spin}_{14})$. 
By definition, $\partial \tilde a \in \Inv(PI_{14}^3)$ is the unique invariant satisfying
	\[
	\tilde a(cL^{\times 4} \cdot \eta) - \tilde a(\eta) = (c){\cdot} \partial \tilde a([q_\eta])
	\]
	for all field extensions $L/k$, $c \in L^\times$, and $\eta \in H^1(L, \mathbf{Spin}_{14})$. Similarly, for all $z \in \Inv(I_{12}^3)$, $\partial' z \in \Inv_{\rm norm}(PI_{12}^3)$ is the unique invariant satisfying 
	\[
	z(\langle c \rangle q) - z(q) = (c) {\cdot} \partial'z([q])
	\]
	for all field extensions $L/k$, $c \in L^\times$, and $q \in I_{12}^3(L)$. Suppose $G(a) = z$. Then for all $L/k$ and  $q \in I_{12}^3(L)$, we have
	\begin{equation} \label{eq:za}
	z(\langle c \rangle  q) - z(q) = a(	\langle c \rangle q \perp \mathbb{H}) - a(q \perp \mathbb{H}) = a(\langle c \rangle(q \perp \mathbb{H})) - a(q \perp \mathbb{H}).
	\end{equation}
	If $\eta \in H^1(L, \mathbf{Spin}_{14})$ is a preimage of $q \perp \mathbb{H}$ (i.e.\ $q_\eta = q \perp \mathbb{H}$) then \eqref{eq:q-c.eta} implies that  the right-hand side of \eqref{eq:za} is equal to 
	\[
	\tilde a (c L^{\times 4} \cdot \eta) - \tilde a (\eta) = (c) {\cdot} \partial \tilde a([q_{\eta}]).
	\]
	By definition, $\partial \tilde a([q_\eta]) = \partial \tilde a([q \perp \mathbb{H}]) = H(\partial \tilde a)([q])$. Therefore we have shown that $\partial 'z = \partial'G(a) = H(\partial \tilde a)$, which was the goal.
\end{proof}

\begin{theorem} \label{main}
	Assume $\sqrt{-1} \in k$. The map $\Inv(I_{14}^3) \to \Inv(\mathbf{Spin}_{14})$ is an isomorphism and $\Inv(I_{14}^3)$ is a free $H(k)$-module generated by the invariants $\{1, a_3, a_6, a_7\}$.  
\end{theorem}

\begin{proof}
	By Theorem \ref{thm:I12-Spin12}, the map $\Inv(I_{12}^3) \hookrightarrow \Inv(\mathbf{Spin}_{12})$ is an isomorphism and by Lemma~\ref{lem:commutes} the following square commutes:
	\[
	\begin{tikzcd}
		\Inv(\mathbf{Spin}_{14}) \ar[r,"\partial"] \ar[d,"F"] & \Inv(PI^3_{14}) \ar[d,"H"] \\
		\Inv(\mathbf{Spin}_{12}) \simeq \Inv(I_{12}^3) \ar[r, "\partial'"] & \Inv(PI^3_{12})
	\end{tikzcd}
	\]
	Let $a \in \Inv(\mathbf{Spin}_{14})$.  By Theorem \ref{thm:I12-Spin12}, there exist unique $\alpha, \beta, \gamma, \varepsilon \in H(k)$ such that  \[F(a) = \alpha{\cdot}1 + \beta{\cdot} z_3 + \gamma{\cdot} z_5 + \varepsilon{\cdot}z_6.\] By Corollary~\ref{cor-InvPI in InvI}, there exist unique $\lambda, \kappa, \omega \in H(k)$ such that \[\partial a = \lambda{\cdot}1 + \kappa {\cdot} a_3 + \omega {\cdot} a_6.\]
	Then according to Lemmas \ref{lem:d'} and \ref{lem:I14-I12},
	\[
\varepsilon{\cdot}{z_5} = \partial'F(a) = H(\partial a) = \lambda{\cdot}1 + \kappa{\cdot}z_3
	\]
But $1, z_3, z_5$ are $H(k)$-linearly independent, so this implies $\varepsilon = \lambda= \kappa = 0$. Therefore 
\begin{align*}
F(a) &= \alpha{\cdot}1 + \beta{\cdot}z_3 + \gamma{\cdot}z_5\\
\partial a &= \omega {\cdot} a_6
\end{align*}
Now let $a' = \alpha{\cdot} 1 + \beta{\cdot}a_3 + \omega{\cdot}a_7 \in \Inv(\mathbf{Spin}_{14})$. By Lemma \ref{lem:d}, $\partial  a- \partial a' = 0$ so Proposition~\ref{prop:fibration}~(ii) yields that $a = a' + y$ for some $y = \nu {\cdot} a_3 + \mu{\cdot} a_6$ in the image of $\Inv(PI_{14}^3) \hookrightarrow \Inv(\mathbf{Spin}_{14})$. By Lemma \ref{lem:I14-I12},
\[
	\gamma{\cdot} z_5 = F(a)-F(a') =  F(y) = \nu{\cdot} z_3,
\] 
which implies that $\gamma = \nu = 0$. Therefore 
\[
a = a' + y = \alpha{\cdot}1 + \beta{\cdot} a_3 + \omega{\cdot} a_7 + \mu{\cdot} a_6.
\]
To prove that $\{1, a_3, a_7, a_6 \}$ is $H(k)$-linearly independent, consider the quadratic form 
\[
Q = \langle t_7 \rangle (\llangle t_1, t_2, t_3 \rrangle' \perp \langle -1 \rangle \llangle t_4, t_5, t_6 \rrangle
\]
over the field $k(t_1, \dots, t_7)$. Then	$a_3(Q) = (t_1){\cdot} (t_2) {\cdot} (t_3) + (t_4){\cdot} (t_5){\cdot} (t_6)$, while
	$a_6(Q) =   (t_1){\cdot} (t_2) {\cdot} (t_3) {\cdot} (t_4){\cdot} (t_5){\cdot} (t_6)$ (see  
\eqref{eq:a6:concrete}) and $a_7(Q) =  (t_1){\cdot} (t_2) {\cdot} (t_3) {\cdot} (t_4){\cdot} (t_5){\cdot} (t_6){\cdot} (t_7)$. These values are $H(k)$-linearly independent in $H(k(t_1, \dots, t_7))$ \cite[Lemma 1.1]{berhuy2011cohomological}.
\end{proof}

\subsection{The mod $2$ invariants  fail to classify anisotropic quadratic forms in~$I_{14}^3$}
In answer to a question of Lam, it was shown by Izhboldin  \cite[Theorem~4.4]{izhboldin1998motivic} that there exists a field extension $F/k$ with a pair of quadratic forms $q_1, q_2 \in I^{3}_{14}(F)$ such that $q_1 - q_2 \in I^4(F)$ but $q_1$ and $q_2$ are not similar. (Although the statement of \cite[Theorem~4.4]{izhboldin1998motivic} only asserts the existence of a field $F$ with this property,  the proof rests on Hoffmann's  example \cite[Theorem~4.3]{hoffmann1998similarity} of a pair of non-similar anisotropic $8$-dimensional forms which are Witt-equivalent modulo $I^4(F)$ and whose Clifford algebras have index 4. It is clear from \cite[Lemma~4.2]{hoffmann1998similarity} that one can arrange for $F$ to be an extension of the original $k$.) Moreover, Izhboldin's $q_1$ and $q_2$ are constructed in such a way that they have a common $6$-dimensional anisotropic subform  \cite[p.\ 348]{izhboldin1998motivic}. In particular, $q_1$ and $q_2$ represent a common element $c \in F^\times$.

Clearly, $q_1 - q_2 \in I^4(F)$ implies $a_3(q_1) = e_3(q_1) = e_3(q_2) = a_3(q_2)$. We may assume that  $\sqrt{-1} \in F$, and then $q_1 - q_2 \in I^4(F)$ also implies that $P_3(q_1) = P_3(q_2)$ if \cite[Proposition~19.12~(3)]{garibaldi2009cohomological}, hence $a_6(q_1) = a_6(q_2)$. The fact that $q_1$ and $q_2$ represent a common element $c$ further implies $a_7(q_1) = (c){\cdot} a_6(q_1) = (c){\cdot}a_6(q_2) = a_7(q_2)$. As a consequence:

\begin{corollary}
	Assume $\sqrt{-1} \in k$. There exists a field extension $F/k$ and a pair of quadratic forms $q_1, q_2 \in I^{3}_{14}(F)$ such that $a(q_1) = a(q_2)$ for all cohomological invariants $a \in \Inv(I^{3}_{14})$ but $q_1$ is not similar to $q_2$.
\end{corollary}

The quadratic forms $q_i \in I_{14}^3(F)$ constructed in \cite[Theorem~4.4]{izhboldin1998motivic} are each similar to a difference of two $3$-Pfister forms (i.e.,  of the form (1) in Corollary~\ref{cor:Rost}). This means (by Corollary~\ref{cor:Str-PI}) that there exist decomposable bi-octonion algebras $(A_1, -)$ and $(A_2, -)$ over $F$ such that the Albert form of $A_i$ is similar to $q_i$. The invariants of these algebras agree, despite the algebras  not being isomorphic (or even isotopic,  by Corollary \ref{cor:isotopic-similar}):  \begin{align*}
&b_1(A_1, -) = b_1(A_2,-) = 0\\
&b_3(A_1, -) = a_3(q_1) = a_3(q_2) = b_3(A_2,-)	\\
&b_6(A_1, -) = a_6(q_1) = a_6(q_2) = b_6(A_2,-)
\end{align*}

\begin{corollary} \label{cor:non-classification_bi-octonion}
	There exists a field extension $F/k$ and bi-octonion algebras $(A_1, -)$ and $(A_2,-)$ such that $b(A_1,-) = b(A_2,-)$ for all cohomological invariants $b \in \Inv((G_2 \times G_2) \rtimes S_2)$ but $(A_1,-)$ and $(A_2,-)$ are not isotopic.
\end{corollary}

In contrast to these two corollaries, the invariant $a_3 \in (I_{14}^3)$ does separate \emph{isotropic} quadratic forms in $I_{14}^3$ up to similarity (see \cite[Corollary]{hoffmann1999conjecture}). Consequently the invariant $b_3 \in \Inv((G_2\times G_2)\rtimes S_2)$ does separate \emph{non-division} bi-octonion algebras up to isotopy.

\subsection{Summary of invariant classification results}
The table below gives a summary of our results on the classification of mod 2 cohomological invariants for various functors.

\begin{table}[hbt]
\begin{center} {\scriptsize 
\begin{tabular}{x{3cm}|x{2.0cm}x{3.5cm}x{4cm}x{0.8cm}} 
$F$ & Restrictions & $\Inv(F)$ & Generators &  Ref.   \\ \hline \\
$\mathbf{PGO}_4$-torsors & none & $H(k)^{\oplus 4}$ & $\{1, y_1, y_2, y_4 \}$ & \ref{thm:inv-bioctonions}
\\ &  \\
$(G_2 \times G_2) \rtimes S_2$-torsors& none & $H(k)^{\oplus 4}$ &   $\{1, b_1, b_3, b_6\}$ & \ref{thm:inv-bioctonions} \\ & \\ 
$PI_{12}^3$ & none & $H(k)^{\oplus 2} \oplus J_1(k)$ & $\{1, z_3\}\cup \{h{\cdot} z_5 \mid h \in J_1(k)\}$ & \ref{cor:Inv-PI12}
\\ & \\
$I_{12}^3$ & none & $H(k)^{\oplus 3} \oplus J_1(k)$ &  $\{1, z_3, z_5\}\cup \{z^h \mid h \in J_1(k)\}$ & \ref{thm:I12-Spin12}
\\ &  \\
$\mathbf{Spin}_{12}$-torsors & none & isomorphic to $\Inv(I_{12}^3)$  & & \ref{thm:I12-Spin12}\\ & \\
$PI_{14}^3$ & none & $H(k)^{\oplus 2} \oplus J_2(k)$ & $\{1, a_3\}\cup \{h{\cdot} a_6 \mid h \in J_2(k)\}$ & \ref{cor-InvPI in InvI}\\ & \\
$I_{14}^3$ & $\sqrt{-1} \in k$ & $H(k)^{\oplus 4}$ & $\{1, a_3, a_6, a_7\}$& \ref{main} \\ & \\
$\mathbf{Spin}_{14}$-torsors & $\sqrt{-1} \in k$ & isomorphic to $\Inv(I_{14}^3)$  & & \ref{main}\\ \
\end{tabular} }
\caption{Structure of various $H(k)$-modules $\Inv(F)$ where $k$ is a field of characteristic not 2 or 3 and $F$ is a functor $\mathsf{Fields}_{/k} \to \mathsf{Sets}$.} \label{table.summary-of-invariants}
\end{center}
\end{table}

\let\oldaddcontentsline\addcontentsline% Store \addcontentsline
\renewcommand{\addcontentsline}[3]{}% Make \addcontentsline a no-op
\bibliographystyle{acm}
\bibliography{mybib}
\let\addcontentsline\oldaddcontentsline% Restore \addcontentsline

\end{document}